\documentclass[a4paper]{amsart}
\usepackage[T1]{fontenc}
\usepackage[utf8]{inputenc}
\usepackage{fullpage}
\usepackage{array}
\usepackage{multirow}
\makeatletter
\newcommand*\ExpandableInput[1]{\input{#1} }
\makeatother

\usepackage{longtable}
\usepackage{booktabs}

\usepackage{caption}
\usepackage{algorithmic,algorithm}

\usepackage{amsmath}
\usepackage{amsthm}
\usepackage{amssymb}
\usepackage[dvipsnames,table]{xcolor}
\usepackage[colorlinks=True, citecolor=blue, linkcolor=blue]{hyperref}
\usepackage[nameinlink]{cleveref}
\usepackage{placeins}
\usepackage{tikz-cd}
\usepackage[normalem]{ulem}
\usepackage{booktabs}

\usepackage{colortbl}
\usepackage{subcaption}
\definecolor{lightgray}{gray}{0.9}

\theoremstyle{definition}
\newtheorem{definition}{Definition}[section]
\newtheorem{example}[definition]{Example}
\newtheorem{remark}[definition]{Remark}

\theoremstyle{plain}
\newtheorem{theorem}[definition]{Theorem}
\newtheorem{lemma}[definition]{Lemma}
\newtheorem{proposition}[definition]{Proposition}
\newtheorem{corollary}[definition]{Corollary}

\DeclareMathOperator{\GL}{GL}

\DeclareMathOperator{\rk}{rk}

\DeclareMathOperator{\id}{id}

\DeclareMathOperator{\NS}{NS}
\DeclareMathOperator{\Aut}{Aut}

\DeclareMathOperator{\sign}{sign}

\DeclareMathOperator{\image}{Im}
\DeclareMathOperator{\HH}{H}

\DeclareMathOperator{\Tr}{Tr}
\DeclareMathOperator{\Nr}{N}
\DeclareMathOperator{\ord}{ord}

\DeclareMathOperator{\scale}{\mathfrak{s}}
\DeclareMathOperator{\norm}{\mathfrak{n}}

\newcommand{\QQ}{\mathbb{Q}}
\newcommand{\FF}{\mathbb{F}}
\newcommand{\NN}{\mathbb{N}}
\newcommand{\RR}{\mathbb{R}}
\newcommand{\CC}{\mathbb{C}}
\newcommand{\ZZ}{\mathbb{Z}}
\newcommand{\PP}{\mathbb{P}}
\newcommand{\DD}{\mathbb{D}}
\renewcommand{\AA}{\mathbb{A}}
\renewcommand{\O}{\mathcal{O}}
\newcommand{\X}{\mathcal{X}}

\newcommand{\M}{\mathcal{M}}
\renewcommand{\H}{\mathcal{H}}
\newcommand{\F}{\mathcal{F}}
\newcommand{\E}{\mathcal{E}}
\newcommand{\C}{\mathcal{C}}

\newcommand{\G}{\mathcal{G}}
\newcommand{\B}{\mathcal{B}}
\newcommand{\A}{\mathcal{A}}
\renewcommand{\L}{\mathcal{L}}

\renewcommand{\P}{\mathcal{P}}
\newcommand{\Gr}{\mathrm{Gr}}

\newcommand{\fp}{\mathfrak{p}}
\newcommand{\FP}{\mathfrak{P}}

\newcommand{\FD}{\mathfrak{D}}

\renewcommand\o{
  \mathchoice
    {{\scriptstyle\mathcal{O}}}%
    {{\scriptstyle\mathcal{O}}}%
    {{\scriptscriptstyle\mathcal{O}}}%
    {\scalebox{.7}{$\scriptscriptstyle\mathcal{O}$}}%
  }

\newcommand{\even}{\rm{II}}

\newcommand{\disc}[1]{{D_{#1}}}

\newcommand{\re}{\operatorname{Re}}
\newcommand{\im}{\operatorname{Im}}
\DeclareMathOperator{\spin}{spin}

\makeatletter
\newcommand{\bigperp}{%
  \mathop{\mathpalette\bigp@rp\relax}%
  \displaylimits
}

\newcommand{\sprod}[2]{h( #1, #2)}
\newcommand{\sprodq}[2]{h( #1,#1)}
\newcommand{\sprodb}[2]{\langle #1, #2 \rangle}
\newcommand{\sprodbq}[1]{\langle #1, #1 \rangle}

\newcommand{\bigp@rp}[2]{%
  \vcenter{
    \m@th\hbox{\scalebox{\ifx#1\displaystyle2.1\else1.5\fi}{$#1\perp$}}
  }%
}
\makeatother

\DeclareCaptionFormat{algor}{%
  \hrulefill\par\offinterlineskip\vskip1pt%
    \textbf{#1#2}#3\offinterlineskip\hrulefill}
\DeclareCaptionStyle{algori}{singlelinecheck=off,format=algor,labelsep=space}
\makeatletter
\def\hrulefill{\leavevmode\leaders\hrule height 0.8pt\hfill\kern\z@}
\makeatother

\title[]{Finite subgroups of automorphisms of K3 surfaces}

\keywords{K3 surfaces, automorphisms}

\author{Simon Brandhorst}

\address{Simon Brandhorst,
Fakult\"at f\"ur Mathematik und Informatik, Universit\"at des Saarlandes, Campus E2.4, 66123 Saarbr\"ucken, Germany}
\email{brandhorst@math.uni-sb.de}

\author{Tommy Hofmann}
\address{Tommy Hofmann,
Naturwissenschaftlich-Technische Fakult\"at, Universit\"at Siegen, Walter-Flex-Straße 3, 57068 Siegen, Germany}
\email{tommy.hofmann@uni-siegen.de}
\makeatother
\thanks{Gefördert durch die Deutsche Forschungsgemeinschaft (DFG) – Projektnummer 286237555 – TRR 195.
S.B. is funded by the Deutsche Forschungsgemeinschaft (DFG, German Research Foundation) – Project-ID 286237555 – TRR 195.}

\subjclass[2020]{14J28, 14J50, 11E39, 11H56}

\def\ackintro{%
\subsection*{Acknowledgments}
The first author thanks K. Oguiso and the university of Tokio for their hospitality and the organizers of the conference Moonshine and K3 surface in Nagoya 2016 for inspiration. Further thanks go to the SageMath community for teaching him programming.
We thank K. Hashimoto and G. Nebe for discussions.
We thank N. Beli, J. Hsia, R.S. Schulze-Pillot for discussions on a typo in T. O'Meara's paper \cite{omeara1958}. We thank M. Kirschmer for sharing his insights on hermitian lattices and his software packages \cite{kirschmer_software} for computations with them.
}

\begin{document}
\begin{abstract}
 We give a complete classification of finite subgroups of automorphisms of complex K3 surfaces up to deformation. The classification is in terms of Hodge theoretic data associated to certain conjugacy classes of finite subgroups of the orthogonal group of the K3 lattice.
 The moduli theory of K3 surfaces, in particular the surjectivity of the period map and the strong Torelli theorem allow us to interpret this datum geometrically. Our approach is computer aided and involves hermitian lattices over number fields.
 \end{abstract}
\maketitle

\setcounter{tocdepth}{1}
\tableofcontents

\section{Introduction}
We work over the field $\CC$ of complex numbers.
A \emph{K3 surface} is a compact, complex manifold $X$ of dimension $2$ with a nowhere vanishing symplectic, holomorphic $2$-form $\sigma_X \in \HH^0(X,\Omega^2_X)$ and vanishing irregularity $h^1(X,\O_X)$.

Since a K3 surface does not admit non-trivial global vector fields, its automorphism group is discrete. For a very general K3 surface it is even trivial. However, there are (families of) K3 surfaces with a non-trivial and even infinite automorphism group. Typical examples of groups that appear as automorphism groups are $\ZZ * \ZZ$, $\ZZ^r$ or $\ZZ/2 \ZZ$.

K3 surfaces with a finite automorphism group have been classified by Nikulin \cite{nikulin1981,nikulin1984}, Vinberg \cite{vinberg2007} and Kondo \cite{kondo1989} with a further recent refinement due to Roulleau \cite{Roulleau2021}.
The purpose of this work is to classify finite subgroups
$G \leq \Aut(X)$, more precisely, pairs $(X,G)$ consisting of a K3 surface $X$ and a finite subgroup of automorphisms $G \leq \Aut(X)$.

Let $X$ be a K3 surface with cotangent sheaf $\Omega_X$. Its automorphisms act on the symplectic forms $\CC \sigma_X = \HH^0(X,\Omega^2_X)$ by scalar multiplication. We call the ones with trivial action \emph{symplectic} and the ones with a non-trivial action \emph{non-symplectic}.
The action on the symplectic form gives rise to an exact sequence
\[1 \to \Aut_s(X) \to \Aut(X) \to \GL(\CC \sigma_X)\]
where by $\Aut_s(X)$ we denote the normal subgroup of symplectic automorphisms.
Now, let $G \leq \Aut(X)$ be a finite subgroup and set $G_s = G \cap \Aut_s(X)$. Then for $n = |G/G_s|$, we get an exact sequence
\[1 \to G_s \to G \to \mu_n \to 1, \]
where $\mu_n$ is the cyclic group of order $n$. We call the index $n = [G:G_s]$ the \emph{transcendental value} of $G$.
Let $\varphi$ denote Euler's totient function. By a result of Oguiso and Machida \cite{Oguiso-Machida1998} we know that $\varphi(n) \leq 20$ and $n\neq 60$.

The distinction between symplectic and non-symplectic is crucial. For instance a non-symplectic automorphism of finite order may fix smooth curves whereas
a symplectic automorphism of finite order $k$ fixes only finitely many points and their number $n_k$ depends only on $k$.
Let $M_{24}$ denote the Mathieu group on $24$ points and $M_{23}$ the stabilizer group of a point. Then the number of fixed points of an element of $M_{23}$ of order $k$ depends only on $k$ and is equal to $n_k$. This observation sparked the following theorem of Mukai.

\begin{theorem}[\cite{mukai1988}]
A finite group admits a faithful and symplectic action on some K3 surface if and only if it admits an embedding into the Mathieu group $M_{23}$ which decomposes the $24$ points into at least 5 orbits.
\end{theorem}

Later, Xiao \cite{Xiao} gave a new proof shedding light on the combinatorics of the fixed points using the relation between
$X$, its quotient $X/G_s$ and its resolution which is a K3 surface again.
A conceptual proof involving the Niemeier lattices was given by Kondo \cite{kondo1998}.
Finally, Hashimoto \cite{Hashimoto2012} classified all the symplectic actions on the K3 lattice. Since the corresponding period domains are connected, this is a classification up to deformation.
See \cite{kondo2018} for a survey of symplectic automorphisms.

In view of these results it is fair to say that our knowledge of finite symplectic subgroups of automorphisms is fairly complete.

Similar to Hashimoto's classification for symplectic actions, our main result is a classification up to deformation (see \Cref{defequivalent} for a precise definition).
Let $X$ be a K3 surface and $G$ a finite subgroup of  automorphisms of $X$.
We call the largest subgroup $S \leq \Aut(X)$ such that the fixed lattices satisfy $\HH^2(X,\ZZ)^{G_s}=\HH^2(X,\ZZ)^{S}$ the \emph{saturation} of $G_s$.
Necessarily, the group $S$ is finite and symplectic. The group generated by $G$ and $S$ is finite as well. It is called the saturation of $G$.
We call $G$ saturated if it is equal to its saturation.
The subgroup $G \leq \Aut(X)$ is called non-symplectic, if $G_s \neq G$ and mixed if further $1\neq G_s$. If $G_s=1$, then it is called purely non-symplectic. If $G$ is non-symplectic then, $X$ is in fact projective. Therefore all K3 surfaces are henceforth assumed to be projective.

\begin{theorem}\label{mainthm}
There are exactly $4167$ deformation classes of pairs $(X,G)$ consisting of a complex K3 surface $X$ and a saturated, non-symplectic, finite subgroup $G \leq \Aut(X)$ of automorphisms.
For each such pair the action of $G$ on some lattice $L \cong \HH^2(X,\ZZ)$ is listed in \cite{K3Groups}.
\end{theorem}

While a list of the actions of these finite groups $G$ is too large to reproduce here,
we present a condensed
version of the data in Table~\ref{table:action} in \Cref{appendixA}.
More precisely we list all finite groups $G$ admitting a faithful, saturated, mixed action on some K3 surface, their symplectic subgroups as well as the number $k(G)$ of deformation types.

Since the natural representation
$\Aut(X) \to O(\HH^2(X,\ZZ))$ is faithful and K3 surfaces are determined up to isomorphism by their Hodge structure,
a large extent of geometric information is easily extracted from our Hodge-theoretical model of the family of surfaces and its subgroup of automorphisms.
For instance, one can compute the N\'eron--Severi and transcendental lattice of a very-general member of the family, the invariant lattice, invariant ample polarizations, the (holomorphic and topological) Euler characteristic of the fixed locus of an automorphism,
the isomorphism class of the group, its subgroup consisting of symplectic automorphisms, the number of connected components of the moduli space and the dimension of the moduli space.
The pairs $(X,G)$ with $G_s$ among the $11$ maximal groups have been classified in \cite{brandhorst-hashimoto2021}. In many cases projective models are listed.

\subsection*{Purely non-symplectic automorphisms}
On the other end of the spectrum are purely non-symplectic groups, which are the groups $G$ with $G_s=1$. These groups satisfy $G \cong \mu_n$ and, by a result of Oguiso and Machida \cite{Oguiso-Machida1998},
we know that $n$ satisfies $\varphi(n) \leq 20$ and $n\neq 60$.
To the best of our knowledge the following corollary completes the existing partial classifications for orders 4 \cite{order4}, 6 \cite{order6}, 8 \cite{order8}, 16 \cite{order16}, 20, 22, 24, 30 \cite{onedim}, $n$ with $(\varphi(n)\geq 12)$ \cite{brandhorst2019} and is completely new for orders 10, 12, 14 and 18. For order 26 it provides a missing case in \cite[Thm 1.1]{brandhorst2019}.
For order $6$, \cite[Thm. 4.1]{order6} misses the case of
a genus 1 curve and four isolated fixed points (0.6.2.29).

\begin{corollary}
  Let $k(n)$ be the number of deformation classes  of K3 surfaces with a purely non-symplectic automorphism acting by $\zeta_n$ on the symplectic form. The values $k(n)$ are given in \Cref{pure}.
\begin{table}[ht]
\caption{Counts of purely non-symplectic automorphisms.}\label{pure}
\rowcolors{1}{}{lightgray}
 \begin{tabular}[t]{cc|cc|cc|cc|cc|cc|cc}
 \toprule
  $n$ &$k(n)$ & $n$ &$k(n)$ &$n$ &$k(n)$ &$n$ &$k(n)$&$n$ &$k(n)$ &$n$ &$k(n)$&$n$ &$k(n)$\\
 \hline
$2$&$75$&$8$&$38$&$14$&$12$&$20$&$7$&$27$&$2$&$36$&$3$&$50$&$1$\\
$3$&$24$&$9$&$13$&$15$&$8$&$21$&$4$&$28$&$3$&$38$&$2$&$54$&$1$\\$4$&$79$&$10$&$37$&$16$&$7$&$22$&$4$&$30$&$10$&$40$&$1$&$60$&$0$\\
$5$&$7$&$11$&$3$&$17$&$1$&$24$&$9$&$32$&$2$&$42$&$3$&$66$&$1$\\
$6$&$150$&$12$&$48$&$18$&$16$&$25$&$1$&$33$&$1$&$44$&$1$&$$&$$\\
$7$&$5$&$13$&$1$&$19$&$1$&$26$&$3$&$34$&$2$&$48$&$1$&$$&$$\\
\bottomrule
\end{tabular}
\end{table}
\end{corollary}
The most satisfying picture is for non-symplectic automorphisms of odd prime order,
where the fixed locus alone determines the deformation class, see \cite{artebani-sarti-taki}.
The key tools to this result are the holomorphic and topological Lefschetz' fixed point formulas as well as Smith theory. These relate properties of the fixed locus with the action of the automorphism on cohomology.
Conversely, given the action on cohomology as per our classification, we determine its fixed locus.

In what follows $\sigma$ is an automorphism of order $n$ on a K3 surface acting by multiplication with $\zeta_n = \exp(2 \pi i/n)$ on the holomorphic $2$-form of $X$.
We denote by
\[X^\sigma = \{x \in X \mid \sigma(x) = x\}\]
the fixed point set of $\sigma$ on $X$.
A curve $C \subseteq X$ is fixed by $\sigma$ if $C \subseteq X^\sigma$ and it is called invariant
by $\sigma$ if $\sigma(C) = C$.
Let $P \in X^\sigma$ be a fixed point. By \cite[lemme 1]{cartan1954} $\sigma$ can be linearized locally at $P$. Hence there are local coordinates $(x,y)$ in a small neighborhood centered at $P$ such that
\[\sigma(x,y) = (\zeta_n^{i+1}x,\zeta_n^{-i} y)\quad \mbox{  with } \quad 0 \leq i \leq s = \left\lfloor\frac{n-1}{2}\right\rfloor.\]
We call $P$ a fixed point of type $i$ and denote the number of fixed points of type $i$ by $a_i$.
If $i = 0$, then $P$ lies on a smooth curve fixed by $\sigma$. Otherwise $P$ is an isolated fixed point. Note that for $n=2$ there are no isolated fixed points and at most $2$ invariant curves pass through a fixed point.

In general, the fixed point set $X^\sigma$ is a disjoint union of $N = \sum_{i=1}^s a_i$ isolated fixed points, $k$ smooth rational curves and either a curve of genus $>1$
or $0, 1, 2$ curves of genus 1.
Denote by $l$ the number of genus $g \geq 1$ curves fixed by $\sigma$. If no such curve is fixed, set $g=1$. We describe the fixed locus by the tuple
$((a_1, \dots, a_s), k, l, g)$. It is a deformation invariant.
To sum up:
\[X^\sigma = \{p_1, \dots, p_N\} \sqcup R_1 \sqcup \dots \sqcup R_k \sqcup C_1 \sqcup \dots \sqcup C_l\]
where the $R_i$'s are smooth rational curves and the $C_j$'s smooth curves of genus $g \geq 1$.

Let $L$ be a $\ZZ$-lattice and $f\in O(L)$ an isometry of order $n$.
Set $L_k:=\{x \in L \mid f^k(x)=x\}$.
The small local type of $f$ is the collection of genera $\G(L_k)_{k \mid n}$. If the genus of $L$ is understood, then we omit $\G(L_n)=\G(L)$ from notation.
Let $\Phi_k(x)\in \ZZ[x]$ denote the $k$-th cyclotomic polynomial.
The global type of $f$ consists of the small local type as well as the isomorphism classes of the $\ZZ$-lattices $\ker \Phi_k(f)$, where $k \mid n$. A genus of $\ZZ$-lattices is denoted by its Conway--Sloane symbol \cite{splag}. The type of an automorphism $\sigma$ of a K3 surface $X$ is defined as the type of the isometry $\sigma^{*-1}|\HH^2(X,\ZZ)$.

\begin{theorem}
 Let $X$ be a K3 surface and $\sigma \in \Aut(X)$ of order $n$ acting by $\zeta_n$ on $\HH^0(X,\Omega_X^2)$. The deformation class of $(X, \sigma)$ is determined by the small local type of $\sigma$ unless $\sigma$ is one the $6$ exceptional types in \Cref{exceptionaltype}.
 For each deformation class, the invariants $((a_1, \dots, a_s), k, l, g)$ of the fixed locus are given in \Cref{appendixB}.
\end{theorem}
\begin{remark}
 For each of the $5$ exceptional types of order $6$ there are exactly two deformation classes. They are separated by the global type.
 For order $4$, the two classes have the same global-type. They are separated by the isometry class of the glue between $L_2$ and $L_4$. It is given by the image $L_4 \to \disc{L_2}$ induced by orthogonal projection.
\end{remark}

 \begin{table}[h]
  \caption{Exceptional types of purely non-symplectic automorphisms.}\label{exceptionaltype}
  \rowcolors{1}{}{lightgray}
 \renewcommand{\arraystretch}{1.2}
 \begin{tabular}{llll|llll}
 \toprule
 $n$ & $\G_1$ & $\G_2$ & $\G_3$
 &$n$ & $\G_1$ & $\G_2$ & $\G_3$\\
 \hline
 $6$ & $\even_{1,8}2^9_1$& $\even_{1,17}3^{-3}$& $\even_{1,8}2^9_1$ &
 $6$ & $\even_{1,7}2^8_2$& $\even_{1,15}3^{3}$ & $\even_{1,7}2^8_2$\\
 $6$ & $\even_{1,5}2^6_4$& $\even_{1,15}3^{-4}$& $\even_{1,7}2^6_4$&
 $6$ & $\even_{1,3}2^4_6$& $\even_{1,11}3^{5}$ & $\even_{1,3}2^4_6$\\
 $6$ & $\even_{1,2}2^3_7$& $\even_{1,9}3^{-6}$ & $\even_{1,3}2^3_7$ &
 $4$ & $\even_{1,9}2^8$ & $\even_{1,17}2^2$& - \\
 \bottomrule
 \end{tabular}
 \end{table}

 \begin{remark}
The first $5$ exceptional types in \Cref{exceptionaltype} are due to a failure of the local to global principle for conjugacy of isometries. For the last one, the example shows that the small local type is not fine enough to determine local conjugacy.
\end{remark}

\subsection*{Enriques surfaces}
Since the universal cover $X$ of a complex Enriques surface $S$ is a K3 surface, our results apply to classify finite subgroups of automorphisms of Enriques surfaces. The kernel of $\Aut(S) \to \GL(\HH^0(2K_S))$ consists of the so called semi-symplectic automorphisms of $S$. They lift to automorphisms acting by $\pm 1$ on $\HH^0(X,\Omega_X^2)$.
Cyclic semi-symplectic automorphisms are studied by Ohashi in
\cite{ohashi2015}. Mukai's theorem on symplectic actions and the Mathieu group has an analogue for Enriques surfaces,
see \cite{mukai-ohashi2015}. However, not every semi-symplectic action is of `Mathieu type`.

\begin{corollary}
 A group $H_s$ admits a faithful semi-symplectic action on some complex Enriques surface if and only if $H_s$ embeds into one of the following $6$ groups

 \begin{center}
   \rowcolors{1}{}{lightgray}
 \renewcommand{\arraystretch}{1.2}
 \begin{tabular}{lc|lc|lc}
 \toprule
 $G$ & id & $G$ & id& $G$ & id\\
 \hline
 $A_6$ & $(360, 118)$ &
 $H_{192}$  & $ (192, 955)$  &
 $S_5$ &  $ (120, 34)$ \\
 $A_{4,4}$ & $ (288, 1026)$ &
 $2^4D_{10}$ &  $ (160, 234)$ &
 $N_{72}$ &  $ (72, 40)$ \\
 \bottomrule
 \end{tabular}
 \end{center}

 A group $H$ admits a faithful action on some complex Enriques surface if and only if it embeds into one of the following $9$ groups:

 \begin{center}
   \rowcolors{1}{}{lightgray}
 \renewcommand{\arraystretch}{1.2}
 \begin{tabular}{lc|lc|lc}
 \toprule
 $G$ & id & $G$ & id& $G$ & id\\
 \hline
 $A_6.\mu_2$ & $(720, 765)$ &
 $H_{192}$ &$(192, 955)$ &
 $\Gamma_{25}a_1.\mu_2$ & $(128, 929)$ \\
 $A_{4,4}.\mu_2$ & $(576, 8652)$ &
 $N_{72}.\mu_2$& $(144, 182)$&
 $S_5$& $(120, 34)$ \\
 $2^4D_{10}.\mu_2$ &$(320, 1635)$ &
 $(Q_8 * Q_8).\mu_4$ &$(128, 135)$ &
 $(C_2 \times D_8).\mu_4$ & $(64, 6)$ \\
 \bottomrule
 \end{tabular}
 \end{center}
\end{corollary}
\begin{proof}
 Let $S$ be an Enriques surface and $X$ its covering K3 surface. Let $\epsilon$ be the covering involution of $X \to S$. Let $\Aut(X,\epsilon)$ denote the centralizer of $\epsilon$.
 Then $1 \to \langle \epsilon \rangle \to \Aut(X,\epsilon) \to \Aut(S) \to 1$ is exact.
 In particular, if $H \leq \Aut(S)$ is a finite group then it is the image of a finite group $G \leq \Aut(X)$ containing the covering involution.
 Conversely, $\epsilon \in \Aut(X)$ is the covering involution of some Enriques surface
 if and only if $\HH^2(X,\ZZ)^\epsilon \in \even_{1,9}2^{10}$.
 Thus we can obtain the list of all finite groups acting on some Enriques surface by taking the corresponding list for K3 surfaces. For each group $G$ in the list one computes the Enriques involutions $\epsilon$, their centralizer $C(\epsilon)$ in $G$, $C(\epsilon)_s\cong H_s$ and the quotient $H\cong C(\epsilon)/\langle \epsilon \rangle$.
\end{proof}

Our method of classification applies as soon as a Torelli-type theorem is available, for instance
to supersingular K3 surfaces in positive characteristic and compact hyperk\"ahler manifolds.

\subsection*{Outline of the paper}

In \Cref{sec:lat} we recall basic notions of lattices with an emphasis towards primitive extensions and lattices with isometry.
The geometric setting of K3 surfaces is treated in \Cref{k3moduli}. We set up a coarse moduli space
parametrizing K3 surfaces together with finite subgroups of automorphisms. Next, we determine the connected components of the respective moduli spaces.
We show that this translates the problem of classifying pairs of K3 surfaces and finite subgroups of automorphisms
into a classification problem for lattices with isometry and extensions thereof.

The next sections deal with these algorithmic problems related to lattices, where it is shown that practical solutions exist.
In particular, in \Cref{ConjAlg} it is described how isomorphism classes of lattices with isometry can be enumerated.
This leads to questions related to canonical images of orthogonal and unitary groups, which are addressed in the final sections. For the classical case of $\ZZ$-lattices we review Miranda--Morison theory in \Cref{mmtheory}. For hermitian lattices we develop the necessary tools in \Cref{hermitianmm}.

Finally, in \Cref{sec:fixed} we classify the fixed point sets of purely non-symplectic automorphisms of finite order on complex K3 surfaces.

\ackintro

\section{Preliminaries on lattices and isometries}\label{sec:lat}
In this section we fix notation on lattices, and refer the reader to \cite{nikulin1980_forms,splag, kneser2002} for standard facts and proofs.

\subsection{Lattices}\label{sect:lattices}
Let $R$ be an integral domain of characteristic $0$ and $K$ its field of fractions.
We denote by $R^\times$ its group of units.
In this paper an $R$-\emph{lattice} consists of a finitely generated projective $R$-module $M$ and
a non-degenerate, symmetric bilinear form $\sprodb{\cdot} \colon M \times M \to K$.

We call it integral if the bilinear form is $R$-valued and we call it even if the square-norm of every element with respect to the bilinear form is in $2 R$. 
If confusion is unlikely, we drop the bilinear form from notation and denote for $x,y \in M$ the value
$\sprodb{x}{y}$ by $xy$ and $\sprodbq{x}$ by $x^2$.
The associated quadratic form is $Q(x) = x^2/2$.
We denote the \emph{dual lattice} of $M$ by $M^\vee$. We call $M$ unimodular if $M=M^\vee$.
For two lattices $M$ and $N$ we denote by $M\perp N$ their orthogonal direct sum.
The scale of $M$ is $\scale(M)=\sprodb{M}{M}$
and its norm $\norm(L)$
is the fractional ideal generated by $\sprodbq{x}$ for $x \in M$.
The set of self-isometries of $M$ is the \emph{orthogonal group} $O(M)$ of $M$.

We fix the following convention for the \emph{spinor norm}: Let $L$ be an $R$-lattice and $V = L \otimes K$. Let $v \in V$ with $v^2 \neq 0$. The reflection $\tau_v(x) = x - 2xv/v^2\cdot v$ is an isometry of $V$.
The spinor norm of $\tau_v$ is defined to be $Q(v)=v^2/2 \in k^\times/(k^\times)^2$.
By the Cartan--Dieudonn\'e theorem $O(V)$ is generated by reflections.
One can show that this defines a homomorphism $\spin: O(V) \to k^\times/(k^\times)^2$ by using the Clifford algebra of $(V,Q)$.

An embedding $M \to L$ of lattices is said to be \emph{primitive} if its cokernel is torsion-free. For $M\subseteq L$ we denote by $M^{\perp L}=\{x \in L \mid \sprodb{x}{M}=0\}$ the maximal submodule of $L$ orthogonal to $M$. If confusion is unlikely we denote it simply by $M^\perp$.
The minimum number of generators of a finitely generated $R$-module $A$ will be denoted by $l(A)$.

Let $L$ be an even integral $R$-lattice. Its \emph{discriminant group} is the group $D_L = L^\vee/ L$ equipped with the \emph{discriminant quadratic form}
$q_L\colon L \to K/2 R$.
Note that $l(\disc{L})\leq l(L)=\rk L$.
Denote by $O(D_L)$ its \emph{orthogonal group}, that is, the group of linear automorphisms preserving the discriminant form.
If $f \colon L \to M$ is an isometry of even $\ZZ$-lattices, then it induces an isomorphism $D_f \colon D_L \to D_M$.
Likewise we obtain a natural map $O(L) \to O(D_L)$, whose kernel is denoted by $O^\sharp(L)$. For an isometry $f \in O(L)$ and $H$ some subquotient of $L \otimes K$ preserved by $f$, we denote by $f|H$ the induced automorphism of $H$.
Let $G\leq O(L)$ be a subgroup. We denote the fixed lattice by $L^G=\{x \in L \mid \forall g \in G \colon g(x)=x \}$ and its orthogonal complement by $L_G=(L^G)^\perp$.

For $R=\RR$, let $s_+$ be the number of positive eigenvalues of a gram matrix and $s_-$ the number of negative eigenvalues. We call $(s_+,s_-)$ the \emph{signature pair} or just \emph{signature} of $L$.

\subsection{Primitive extensions and glue}
Let $R \in \{\ZZ, \ZZ_p\}$ and $L$ be an even integral $R$-lattice.
We call $M \perp N \subseteq L$ a \emph{primitive extension} of $M \perp N$ if $M$ and $N$ are primitive in $L$ and $\rk L = \rk M + \rk N$.
Since $L$ is integral, we have a chain of inclusions
\[M \perp N \subseteq L \subseteq L^\vee \subseteq M^\vee \perp N^\vee.\]
The projection $M^\vee \perp N^\vee \to M^\vee$ induces a homomorphism $L/(M \perp N) \to D_M$. This homomorphism is injective if and only if $N$ is primitive in $L$.
Let $H_M$ denote its image and define $H_N$ analogously. The composition
\[\phi \colon H_M \to L/(M \perp N) \to H_N\]
is called a \emph{glue} map. It is an anti-isometry, i.e. $q_M(x) = -q_N(\phi(x))$ for all $x \in H_M$. Note that $L/(M\perp N) \leq H_M \perp H_N \leq D_M \perp D_N$ is the graph of $\phi\colon H_M \to H_N$.

Conversely, any anti-isometry $\phi \colon H_M \to H_N$ between subgroups $H_M \subseteq D_M$ and $H_N \subseteq D_N$ is the glue map of a primitive extension: $M \perp N \subseteq L_\phi$ where $L_\phi$ is defined by
the property that $L_\phi/(M \perp N)$ is the graph of $\phi$.

The determinants of the lattices in play are related as follows:
\[\lvert\det L\rvert = \lvert D_M/H_M \rvert \cdot \lvert D_N/H_N \rvert = \lvert \det M \rvert \cdot \lvert \det N \rvert/[L : (M \perp N)]^2.\]
If $f_M \in O(M)$ and $f_N \in O(N)$ are isometries, then $g = f_M \oplus f_N$ preserves the primitive extension $L_\phi$ if and only if $\phi \circ D_{f_M} = D_{f_N} \circ \phi$. We call $\phi$ an \emph{equivariant glue map} with respect to $f_M$ and $f_N$.

\subsection{Lattices with isometry}
We are interested in classifying conjugacy classes of isometries of a given $\ZZ$-lattice.
If the lattice in question is definite, its orthogonal group is finite. Using computer algebra systems, one can compute the group as well as representatives for its conjugacy classes. This approach breaks down if $L$ is indefinite.

\begin{definition}
A \emph{lattice with isometry} is a pair $(L,f)$ consisting of a lattice $L$ and an isometry $f \in O(L)$.
We frequently omit $f$ from the notation and denote the lattice with isometry simply by $L$. Its isometry $f$ is then denoted by $f_L$.
We say that two lattices with isometry $M$ and $N$ are \emph{isomorphic}
if they are equivariantly isometric, i.e., if there exists an isometry
$\psi \colon M \to N$ with $\psi \circ f_{M} = f_{N} \circ \psi$.
We view $L$ as a $\ZZ[x,x^{-1}]$ module via the action $x \cdot m = f(m)$, $x^{-1} \cdot m = f^{-1}(m)$.
\end{definition}
The unitary group $U(L)$ of the lattice with isometry $L$ is the centralizer of $f_L$ in $O(L)$. This is nothing but $\Aut(L)$ in the category of lattices with isometry.

Terminology for lattices applies to lattices with isometry verbatim. For instance the discriminant group $D_L$ of a lattice with isometry comes equipped with the induced isometry $D_{f_L} \in O(D_L)$ and
$U(D_L)$ is the centralizer of $D_{f_L}$ in $O(D_L)$. We denote by
$G_L = \im( U(L) \to U(D_L))$.

\begin{proposition}\label{extensions}
Let $M,N$ be lattices with isometry.
Suppose that the characteristic polynomials of $f_M$ and $f_N$ are coprime.
Then the double coset
\[U(N) \backslash \{\mbox{equivariant glue maps } \disc{M} \supseteq H_M \xrightarrow{\phi } H_N \subseteq D_{N} \} / U(M)\]
is in bijection with the set of isomorphism classes of lattices with isometry $(L,f)$ with characteristic polynomial $\chi_f(x) = \chi_{f_{M}}(x)\chi_{f_{N}}(x)$ and $M \cong (\ker \chi_{f_M}(f))$ and $N \cong (\ker \chi_{f_N}(f))$.
\end{proposition}
\begin{proof}
We work in the category of lattices with isometry. Let $L_\phi$ and $L_{\psi}$ be primitive extensions of $M \perp N$ and $h\colon L_\phi \to L_\psi$ an isomorphism.
Then $h|_M \in U(M)$ and $h|_N \in U(N)$.
We have $D_{h|_N}\phi D_{h|_M}^{-1} = \psi$, so $\psi \in U(M) \phi U(N)$.
Conversely if $D_a \phi D_b = \psi$. Then $a \oplus b \colon L_\phi \to L_\psi$ is an isomorphism.
\end{proof}

\section{K3 surfaces with a group of automorphisms}\label{k3moduli}
See \cite{bhpv} or \cite{huybrechts2016} for generalities on K3 surfaces.
\subsection{$\mathbf{H}$-markings}
Let $X$ be a K3 surface. We denote by
\[\rho_X\colon \Aut(X) \rightarrow O(\HH^2(X,\ZZ)),\  g\mapsto (g^{-1})^*\]
the natural representation of the automorphism group $\Aut(X)$ on the second integral cohomology group of $X$.
It is faithful.
\begin{definition}\label{defequivalent}
Let $X$, $X'$ be K3 surfaces and $G \leq \Aut(X)$, $G' \leq \Aut(X')$. We call $(X,G)$ and $(X',G')$ \emph{conjugate}, if
there is an isomorphism $\phi\colon X \to X'$ such that
$\phi G \phi^{-1} = G'$.
  They are called \emph{deformation equivalent} if there exists a connected family $\X \to B$ of K3 surfaces, a group of automorphisms $\G \leq \Aut(\X/B)$ and
two points $b,b' \in B$ such that the restriction of $(\X,\G)$ to the
fiber above $b$ is conjugate to $(X,G)$ and to the fiber above $b'$ is conjugate to $(X',G')$.
\end{definition}

Let $L$ be a fixed even unimodular lattice of signature $(3,19)$.
An $L$-marking of a K3 surface $X$ is an isometry
$\eta \colon \HH^2(X,\ZZ) \to L$.
The pair $(X, \eta)$ is called an $L$-marked K3-surface.
A family of $L$-marked K3 surfaces is a family $\pi \colon \X \to B$ of K3 surfaces with an isomorphism of local systems $\eta \colon R^2\pi_*\underline{\ZZ} \to \underline{L}$. If the base $B$ is simply connected, then a marking of a single fiber extends to the whole family.

\begin{definition}
  Let $H \leq O(L)$ be a \emph{finite} subgroup. An $H$-marked K3 surface is a triple $(X, \eta, G)$ such that $(X, \eta)$ is an $L$-marked K3 surface and $G \leq \Aut(X)$ is a group of automorphisms with $\eta \rho_X(G) \eta^{-1} = H$.
 We say that $(X,G)$ is $H$-markable if there exists some marking by $H$.

 Two $H$-marked K3 surfaces $(X_1, \eta_1, G_1)$
 and $(X_2,\eta_2, G_2)$ are called \emph{conjugate} if there exists an isomorphism $f \colon X_1 \to X_2$ such that $\eta_1 \circ f^* = \eta_2$. In particular, $fG_1f^{-1}=G_2$. We call $H$ \emph{effective} if there exists at least one $H$-marked K3 surface.
\end{definition}

A family of $H$-marked K3 surfaces is a family of $L$-marked K3 surfaces
$\pi \colon \X \to B$ of K3 surfaces with an isomorphism of local systems $\eta \colon R^2\pi_*\underline{\ZZ} \to \underline{L}$ and group of automorphisms $\G \leq \Aut(\X/B)$ such that for each $b\in B$ the fiber $(\X_b,\eta_b,\G_b)$ is an $H$-marked K3 surface.

Let $(X,\eta,G)$ be an $H$-marked K3 surface. The action of $G$ on $\HH^{2,0}(X)$ induces via the marking $\eta$ a character $\chi\colon H \rightarrow \CC^\times$. We call such a character \emph{effective}. We denote complex conjugation by a bar $\bar\cdot$.
Set $H_s = \ker \chi$, and denote the $\chi$-eigenspace by
\[L_\CC^\chi = \{x \in L \otimes \CC \mid h(x) = \chi(h)\cdot x  \mbox{ for all }h \in H\}.\]
 Similarly let
\[L_\RR^{\chi+\bar\chi}=\{x \in L_\RR \mid (h+h^{-1})(x) = \chi(h)x+\bar \chi(h)x \text{ for all } h \in H\}.\]
The generic transcendental lattice $T(\chi)$ is the smallest primitive sublattice of $L$ such that $T(\chi) \otimes \CC$ contains $(L\otimes \CC)^\chi$. We call
$\NS(\chi) = T(\chi)^\perp$ the generic N\'eron--Severi lattice.
Recall that $L_{H_s}$ is the complement of the fixed lattice $L^{H_s}$.

\begin{proposition}\label{iseffective}
Let $H\leq O(L)$ be a finite group and $\chi \colon H \to \CC^\times$ a non-trivial character. Recall that $H_s := \ker \chi$. Then $\chi$ is effective if and only if the following hold:
\begin{enumerate}
 \item $L_{H_s}$ is negative definite;
 \item $L_{H_s}$ does not contain any $(-2)$-vector;
 \item the signature of $L^H$ is $(1,*)$;
 \item $L_\RR^{\chi+\bar \chi}$ is of signature $(2,*)$;
 \item $\NS(\chi)_H$ does not contain any vector of square $(-2)$.
 \end{enumerate}
If $\chi$ is trivial, then $\chi$ is effective if and only if (1) and (2) hold.
\end{proposition}
\begin{proof}
  If $\chi$ is trivial, this is known (cf. \cite[4.2, 4.3]{nikulin1980}).
  So let $\chi$ be non-trivial. We show that (1-5) are necessary. (1-2) follow from \cite[4.2]{nikulin1980}. Let $H$ be an effective subgroup and $(X,\eta,G)$ an $H$-marked K3 surface.
  Since $T(\chi)\otimes \CC$ contains the period $\eta(\omega)$ of $(X,\eta)$, we have $\eta(\omega+\bar \omega) \in L_\RR^{\chi+\bar \chi}$.
  Recall that $\omega.\bar \omega >0$. Thus $L_\RR^{\chi+\bar \chi}$ has at least two positive squares.
  Let $h \in \NS(X)$ be ample. Then $h' = \sum_{g\in G} g^*h$ is ample and $G$-invariant i.e. $h' \in H^2(X,\ZZ)^G=\eta^{-1}(L^H)$. Since $h'^2>0$, $L^H$ has at least one positive square. Since $H^{1,1}(X)$ and $H^{2,0}(X)\oplus H^{0,2}(X)$ are orthogonal, so are $h$ and $L_\RR^{\chi+\bar \chi}\subseteq \eta(H^{2,0}(X)\oplus H^{0,2}(X))$. Therefore all $3$ positive squares of $L$ are accounted for. This proves (3) and (4).
  For (5) we note that $\NS(\chi)\subseteq \eta(\NS(X))$.
  For $\delta \in \NS(X)$ with $\delta^2=-2$ we know that $\delta$ or $-\delta$ is effective by Riemann--Roch.
  Therefore $h'.\delta \neq 0$, because $h'$ is ample.
  Since $\NS(X)_G \subseteq h^\perp \cap \NS(X)$,
  $\NS(X)_G$ does not contain any vector of square (-2).
  The same holds true for $\NS(\chi)_H\subseteq \eta(\NS(X)_G)$.

Since the signature of $L^{\chi+\bar\chi}_\RR$ is $(2,*)$, we can find an element $\omega$ in $(L\otimes \CC)^\chi$ such that
  $\omega.\bar \omega>0$, $\omega^2=0$. Choosing $\omega$ general enough we achieve $\omega^\perp \cap T(\chi) = 0$. By the surjectivity of the period map \cite[VII (14.1)]{bhpv}, we can find an $L$-marked K3 surface $(X,\eta)$ with period $\omega$.
  By construction $\eta(T(X)) = T(\chi)$, so $X$ is projective and $\eta(\NS(X)) = \NS(\chi)$.
Since $L^H$ is of signature $(1,*)$ and $L_H$ does not contain any $(-2)$-roots, we find $h \in L^H$ with $h^2>0$ and $h^\perp \cap \NS(\chi)$ not containing any $-2$ roots either and after possibly replacing $h$ by $-h$, we can assume that $h$ lies in the positive cone. Thus $h$ lies in the interior of a Weyl chamber.
Since the Weyl group $W(\NS(X))$ acts transitively on the Weyl chambers of the positive cone, we find an element $w \in W(\NS(\chi))$ such that that $(\eta\circ w)^{-1}(h)$ is ample. Let $\eta'= \eta\circ w$. By construction, every element of $G' = \eta'^{-1} H \eta'$
preserves this ample class and the period of $X$. So $G'$ is a group of effective Hodge isometries. By the strong Torelli theorem (see e.g. \cite[VIII \S 11]{bhpv}), $G' = \rho_X(G)$ for some $G \leq \Aut(X)$.

\end{proof}

Note that for any effective character $\chi\colon H \to \CC^\times$ of $H$,
\[\ker \chi = H_s = \{h \in H \mid L^h \mbox{ is of signature } (3,*)\}\]
is independent of $\chi$. Indeed, because $L_{H_s}$ is negative definite, $L^{H_s}\subseteq L^h$ is of signature $(3,*)$ for any $h \in H_s$. On the other hand if $g\in H$ is not in $H_s$, then $\chi(g)\neq 1$ and so $L^g$ is orthogonal to $L_\RR^{\chi+\bar \chi}$ and contains $L^H$. Therefore $L^g$ is of signature $(1,*)$.
The kernel is a normal subgroup and $H/H_s \cong \mu_n$ is cyclic. We say that an effective group $H$ is \emph{symplectic} if $H=H_s$ and non-symplectic otherwise.

\begin{lemma}
 Let $H \leq O(L)$ be effective. There are at most two effective characters $\chi \colon H \to \CC^\times$. They are complex conjugate.
\end{lemma}
\begin{proof}
  Fix a generator $h H_s$ of $H/H_s$ and let $\chi$ be an effective character. It is determined by its value $\chi(h)$, which is a primitive $n$-th root of unity. Set $T=T(\chi)$.
Since $\chi$ is effective,
$L_\RR^{\chi +\bar \chi}= T_\RR^{\chi +\bar \chi}$ is of signature $(2,*)$.
It is equal to the $\chi(h)+\bar \chi(h)$ eigenspace of $(h+h^{-1})|T_\RR$. The other real eigenspaces of $(h+h^{-1})|T_\RR$ are negative definite. Thus any effective character $\chi'$ is equal to $\chi$ or $\bar \chi$.
\end{proof}
In particular this shows that $T(\chi)$ and $\NS(\chi)$ do not depend on the choice of the effective character $\chi$, but only on $H$. We may denote them as $T(H)$ and $\NS(H)$.
\subsection{Moduli spaces and periods}
Let $\M_H$ denote the fine moduli space parametrizing isomorphism classes of $H$-marked K3 surfaces $(X,\eta,G)$. It is a non-Hausdorff complex manifold.
It can be obtained by gluing the base spaces of the universal deformations of $(X,\eta,G)$, see \cite[\S 3]{brandhorst-cattaneo}.

Let $\chi\colon H \to \CC^\times$ be an effective character. We have $\M_H = \M^\chi_H \cup \M^{\bar\chi}_H$ where $\M_H^\chi$ parametrizes isomorphism classes of $H$-marked K3 surfaces $(X,\eta, G)$ with $\chi(\eta \rho_X(g)\eta^{-1}) \cdot \omega_X = (g^*)^{-1} \omega_X$ for all $g \in G$.
By \cite[Prop. 3.9]{brandhorst-cattaneo} the forgetful map $\M_H \to \M_L$, $(X,\eta,G) \mapsto (X,\eta)$ into the moduli space of $L$-marked K3 surfaces is a closed embedding.

\begin{definition}
 Let $H\leq O(L)$ be effective and $\chi\colon H \to \CC^\times$ be a non-trivial effective character.
 We denote by
 \[\DD^\chi= \{\CC \omega \in \PP((L\otimes \CC)^\chi) \mid \sprodb{\omega}{\omega}=0, \sprodb{\omega}{\bar\omega}>0\}\]
 the corresponding period domain and its period map by
 \[\P \colon \M_H^\chi \to \DD^\chi.\]
It is a local isomorphism (see e.g. \cite{brandhorst-cattaneo}).
The discriminant locus is $\Delta = \bigcup  \left\{\PP(\delta^\perp) \mid \delta \in L_H, \delta^2=-2\right\} \subseteq \PP(L_\CC)$.
\end{definition}

\begin{proposition}\label{perioddomain}
Let $\chi \colon H \to \CC^\times$ be an effective and non-trivial character.
 \begin{enumerate}
  \item The image of $\P$ is $\DD^\chi \setminus \Delta $.
  \item If $(X_1,\eta_1,G_1)$ and $(X_2,\eta_2,G_2)$ lie in the same fiber $\P^{-1}(\CC \omega)$, then $(X_1,G_1)$ and $(X_2,G_2)$ are conjugate.
  \item If $\chi$ is real, then $\DD^\chi \setminus \Delta $ has two connected components. They are complex conjugate. If $\chi$ is not real, then $\DD^\chi \setminus \Delta$ is connected.
 \end{enumerate}
\end{proposition}
\begin{proof}
(1) Let $\CC \omega \in \DD^\chi$.
By the surjectivity of the period map we find a marked K3 surface $(X,\eta)$ with period $\CC \omega$.
Set $N = \CC \omega^\perp \cap L$. By Lefschetz' theorem on $(1,1)$-classes we have $\eta(\NS(X))=N$.

  If $\CC\omega \not \in \Delta$, we have that $\NS(H)_H = \NS(H) \cap (L^H)^\perp$ does not contain any roots. This means that $\eta^{-1}(\NS(H)^H)\subseteq \NS(X)$ contains an ample class.
Thus $\eta^{-1} H \eta$ preserves the period and the ample cone. By the strong Torelli theorem there is a (unique) group of automorphisms $G\leq \Aut(X)$ with
$\eta \rho_X(G) \eta^{-1}=H$.

If conversely $\CC \omega \in \Delta$, then $H$ does not preserve a Weyl chamber of $N$ so that $H$ cannot come from a group of automorphisms of $X$.

(2) Let $(X_1,\eta_1,G_1)$ and $(X_2,\eta_2,G_2)$ be $H$-marked K3 surfaces in the fiber of $\CC \omega$.
Then $\varphi = \eta_1^{-1} \circ \eta_2$ is a Hodge isometry which conjugates $\rho_{X_1}(G_1)$ and $\rho_{X_2}(G_2)$. However, it may not preserve the ample cones. By \cite[Lemma 1.7 and Theorem 1.8]{oguiso-sakurai} there exists a unique element $w \in \langle \pm 1\rangle \times W(\NS(X_1))$
such that $w \circ \varphi$ preserves the ample cones and $w g^*= g^* w$ for all $g \in G_1$. Since now $w \circ \varphi$ is an effective Hodge isometry, the strong Torelli Theorem applies and provides an isomorphism $F\colon X_1 \to X_2$ with $F^*=w \circ \varphi$.
By construction we have $FG_1 F^{-1} = G_2$ as desired.

  (3) By \cite[\S9 \& \S 11]{dolgachev-kondo2005} the period domain $\DD^\chi$ has $2$ connected components if $\chi$ is real and one else. The discriminant locus is a locally finite union of real codimension two hyperplanes. Removing it does not affect the number of connected components.
\end{proof}

\begin{remark}
If $\omega \in L_\CC^\chi$, then we have $\omega^2 = \chi(g)^2 \omega^2$ for any $g \in G$. Thus $(1-\chi(g)^2)\omega^2=0$.
Let $n = [H:H_s]=\lvert \image \chi \rvert$. For $n>2$ this condition implies $\omega^2=0$ and hence the dimension of $\DD^\chi$ is $\dim_\CC L^\chi_\CC -1$.
While for $n=2$ it is $\dim L^\chi_\CC -2$.
\end{remark}

Denote by $N(H)=\{f \in O(L) \mid f H = H f\}$ the normalizer of $H$ in $O(L)$.
If $(X,\eta,G)$ is an $H$-marked K3 surface and $f \in N(H)$, then $(X,f \circ \eta, G)$ is an $H$-marked K3 surface as well. In fact all $H$-markings of $(X,G)$ arise in this way.
So $N(H)$ is the group of changes of marking.

Set $\DD_H = \DD^\chi \cup \DD^{\bar \chi}$.
The group $N(H)$ acts on $\DD_H$ via an arithmetic subgroup of $O(T(H))$, respectively $U(T(H))$.
Therefore by \cite{baily-borel1966} the space $\DD_H/N(H)$ is a quasi-projective variety with only finite quotient singularities.

\begin{theorem}
 The coarse moduli space $\F_H:=\M_H/N(H)$ of $H$-markable K3 surfaces admits a bijective period map
 $\F_H \to (\DD_H \setminus \Delta)/N(H)$.
\end{theorem}
\begin{proof}
We can use the action of the normalizer to forget the marking and thus obtain a period map $\F_H =\M_H/N(H) \to (\DD_H\setminus \Delta)/N(H)$.
That it is bijective follows from \Cref{perioddomain}. Part (1) gives surjectivity and part (2) injectivity. Indeed, if two $H$ polarizable K3 surfaces have the same image, then we can find markings on them such that they lie in the same fiber of the period map $\M_H \to \DD_H$. Then they are conjugate by \Cref{perioddomain} (2).
\end{proof}
See \cite{alexeev-engel,alexeev-engel-han} for more on moduli of K3 surfaces and their compactifications.

\subsection{Connected components}

We next show that deformation classes of K3 surfaces with
finite groups of automorphisms are precisely the connected components of the coarse moduli space
of $H$-markable K3 surfaces $\F_H$. In the following, for a topological space $Y$ we denote by $\pi_0(Y)$ the set of (path) connected components of $Y$.

\begin{theorem}\label{deformation class vs F_H}
  Let $(X,G)$ and $(X',G')$ be two pairs consisting of K3 surfaces $X, X'$ and finite groups of automorphisms $G \leq \Aut(X)$, $G' \leq \Aut(X')$. Then $(X,G)$ and $(X',G')$ are deformation equivalent if and only if they are markable by the same effective subgroup $H\leq O(L)$ and they lie in the same connected component of $\F_H$.
\end{theorem}

\begin{proof}
Every pair $(X,G)$ is $H$-markable for some effective subgroup $H\leq O(L)$.

Let $(X',G')$ be deformation equivalent to $(X,G)$. This means that we find a connected family
$\X \to B$, a group of automorphisms $\G\leq \Aut(\X/B)$ and points $b,b' \in B$ such that the fibers above $b$ and $b'$ are conjugate to $(X,G)$ and $(X',G')$. The $H$-marking of $(X,G)$ induces an $H$-marking of the fiber above $b$. By parallel transport in the local system $R^2 \pi_*\underline{\ZZ}$ we move the marking from $(X,G)$ to $(X',G')$ along some continuous path $\gamma$ in $B$ connecting $b$ and $b'$. Therefore the fiber above $b'$ is $H$-markable. The isomorphism of the fiber with $(X',G')$ allows to transport this marking to $(X',G')$.  Therefore $(X',G')$ is $H$-markable. Its point in the moduli space $\F_H$ of $H$-markable K3-surfaces coincides with that of the fiber above $b'$. Likewise $(X,G)$ gives the same point in $\F_H$ as the fiber above $b$. Since the two fibers lie in the same connected component of $\M_H$ their images lie in the same connected component of $\F_H$.

Conversely let $(X,G)$ and $(X',G')$ be $H$-markable and in the same connected component of $\F_H$.
Then we can find markings $\eta$, and $\eta'$ such that $(X,\eta,G)$ and $(X,\eta',G')$ are $H$-marked.
Then $\pi_0(\F_H) = \pi_0(\M_H/N(H)) \cong \pi_0(\M_H)/N(H)$. Since $(X,G)$ and $(X',G')$ lie in the same connected component of $\F_H$, we find $n \in N(H)$ such that $(X,\eta,G)$ and $(X,n \circ \eta',G)$ lie in the same connected component of $\M_H$. Since $\M_H$ is a fine moduli space it has a universal family and this gives a deformation of $(X,\eta,G)$ and $(X,n \circ \eta',G)$ as $H$-marked K3 surfaces. By forgetting the markings we obtain a deformation of $(X,G)$ with $(X',G')$.
\end{proof}

\begin{corollary}
  The set of deformation classes of pairs $(X,G)$ with $X$ a K3 surface and $G \leq \Aut(X)$ with $G \neq G_s$ is in bijection with the set $\bigcup_{H \in T} \pi_0(\F_H)$, where $T$ is a transversal of the set of conjugacy classes effective, non-symplectic subgroups of $O(L)$.
\end{corollary}

Let $L$ be a K3 lattice and $H\leq O(L)$ an effective subgroup. Our next goal is to determine the connected components of the coarse moduli space $\F_H$ of $H$-markable K3 surfaces.
Since the period domain $\DD_H \setminus \Delta$ has exactly two connected components (\Cref{perioddomain}~(3)), $\F_H$ has at most two components as well. It has only one connected component if and only if the action of $N(H)$ on $\DD_H$ exchanges the two components.

Let $\chi: H \to \CC^\times$ be an effective character. For $n \in N(H)$ denote by
$\chi^n$ the character defined by $\chi^n(h) = \chi(n^{-1} h n)$.
Denote by $N(\chi)$ the stabilizer of $\chi$ in $N(H)$. For completeness sake we mention the following proposition.

\begin{proposition}
 Let $\chi \colon H \to \CC^\times$ be an effective character and $[H:H_s]>2$.
 Then the number of connected components of $\F_H$ is $2/[N(H):N(\chi)]$.
\end{proposition}
\begin{proof}
 Since $[H:H_s]>2$, the connected components of $\M_H$ are $\M_H^\chi$ and $\M_H^{\bar \chi}$. If $(X,\eta,G) \in \M_H^\chi$ and $n \in N(H)$ then $(X,n \circ \eta,G) \in \M_H^{\chi^n}$.
\end{proof}

Let now $[H:H_s]=2$. Then we have seen that $\chi=\bar \chi$ and $\DD^\chi$ has two connected components.
This can be dealt with by introducing \emph{positive sign structures}. Our account follows \cite{shimada2018}.
A period $\CC \omega \in \DD^\chi$ can be seen as an oriented, positive definite real $2$-plane. Indeed, the two real vectors $\re \omega, \im \omega \in L_\RR$ give an ordered (and thus oriented), orthogonal basis of a positive definite plane in $L$.

\begin{definition}
 Let $L$ be a $\ZZ$-lattice of signature $(s_+,s_-)$. Then a \emph{sign structure} on $L$ is defined as a choice of one of the connected components $\theta$ of the manifold parametrizing oriented, $s_+$-dimensional, positive definite, real subspaces $S$ of $L_\RR$. Unless $L$ is negative definite it has exactly two positive sign structures.
\end{definition}

For $[H:H_s]>2$, the periods in $\DD^\chi$ all give the same sign structure. But for $[H:H_s]=2$ there are two sign structures which give the two connected components of the period domain.
Whether or not $N(H)$ preserves the sign structure is encoded by a certain character. See \cref{sect:lattices} for our conventions on the spinor norm.
\begin{proposition}[\cite{looijenga-wahl1986,miranda-morrisonI}]
 Let $L$ be an $\RR$-lattice, so that the spinor norm takes values in $\{\pm 1\}\cong \RR^\times/\RR^{\times 2}$.
 The action of an isometry $g \in O(L)$ on the set of positive sign structures of $L$ is trivial if and only if $\det(g) \cdot \spin(g) > 0$.
\end{proposition}
Let $L$ be an $\RR$-lattice. We denote by $O^+(L)=\ker (\det \cdot \spin)$ the subgroup of orientation preserving isometries. If $G \leq O(L)$ is a subgroup, we denote by $G^+ = G \cap O^+(L)$ its normal subgroup of orientation preserving elements.

\begin{proposition}\label{quadruplewithtriple}
 Let $\chi \colon H \to \CC^\times$ be an effective character and $[H:H_s]=2$.
  Set $T = T(\chi)$, let $\pi \colon N(H) \to O(T)$ be the restriction and $N_T= \pi(N(H))$.
 Then the number of connected components of $\F_H$ is $2/[N_T:N_T^+]$.
\end{proposition}
\begin{proof}
  The subgroup of $N(H)$ fixing the sign structures of $T(\chi)$ is by the definition $\pi^{-1}(N_T^+)$.
Therefore $2/[\pi(N(H)):\pi(N(H))^+]$ is the number of connected components of $N(H)$.
\end{proof}

\begin{lemma}
 Keep the notation of \Cref{quadruplewithtriple}.
 Let $\disc{} \colon O(T) \to O(\disc{T})$ be the natural map, $J =D(N_T)$, $J^+=D(N_T^+)$ and
 $K = \ker \disc{}$.
 Then $[N_T:N_T^+] = [K:K^+][J:J^+]$.
\end{lemma}
\begin{proof}
We claim that that $K \subseteq N_T$. By definition any element $g\in K$ acts trivially on the discriminant group $\disc{T}$. Therefore $g\oplus \id_{\NS(\chi)}$ extends to $L$. Since the restriction of any $h\in H$ to $T$ is given by $\pm \id_{T}$, $g \oplus \id_{\NS(\chi)}$ commutes with $h$. The claim is proven.
As a consequence we have $N_T = \disc{}^{-1}(J)$.
Thus we have a commutative diagram with exact rows
\[
\begin{tikzcd}
  1\arrow[r]& K^+ \arrow[r]\arrow[d] &N_T^+ \arrow[r]\arrow[d]& J^+\arrow[r]\arrow[d] & 1\\
1\arrow[r]& K\arrow[r] &N_T\arrow[r]{}{D} & J \arrow[r] &1
\end{tikzcd}\]
where the vertical arrows are inclusions of normal subgroups.
Hence the cokernels exist and so
we obtain the corresponding exact sequence
\[1 \to K/K^+ \to N_T/N_T^+ \to J/J^+ \to 1\]
of the cokernels. To see this follow the proof of the snake lemma, which indeed is valid in this situation.
\end{proof}

\begin{remark}\label{rem:mmapplication}
We will see in \Cref{plusindexcompute}
how to compute the number of components $[N_T:N_T^+]=[K:K^+][J:J^+]$ using Miranda--Morrison theory.
\end{remark}

\subsection{Saturated effective subgroups}\label{HpolEnum}
By \Cref{deformation class vs F_H}, the set of connected components of $\F_H$ is in bijection with the deformation classes of $H$-markable K3 surfaces.
Our next goal is to enumerate all possible effective groups $H$ up to conjugacy in $O(L)$.
The symplectic fixed and cofixed lattices $L^{H_s}$ and $L_{H_s}$ turn out to be the crucial invariants for this task.
Let $L$ be a $\ZZ$-lattice and $M$ a subset of $L \otimes \CC$. We denote by $O(L,M)=\{f \in O(L) \mid f_\CC(M)=M\}$ the stabilizer of $M$.

\begin{definition}
 Let $L$ be a K3 lattice and $H \leq O(L)$ an effective subgroup. Then we call the kernel $S$ of $O(L,L^{H_s}) \to O(L^{H_s})$ the saturation of $H_s$.
 The group generated by $H$ and $S$ is called the saturation of $H$.
 For a K3 surface $X$ we call a finite subgroup $G\leq \Aut(X)$ saturated if its image $\rho_X(G_s)$ is saturated.
 We call a saturated symplectic group $H_s\leq O(L)$ a \emph{heart}.
\end{definition}

\begin{remark}
Note that the saturation $S$ of $H_s$ is the largest subgroup $S\leq O(L)$ with $L^S=L^{H_s}$. Further the saturation of an effective group $H\leq O(L)$ is effective.
Indeed, if $L = \HH^2(X,\ZZ)$ for some K3 surface $X$ and $G_s$ a finite group of symplectic automorphisms, then the strong Torelli theorem implies that the saturation of $\rho_X(G_s)=H_s$ is in the image of $\rho_X$ by a finite group of symplectic automorphisms containing $G_s$ and with the same fixed lattice.
We conclude that if a pair $(X,G)$ is markable by $H$ then it is also markable by the saturation of $H$.
Therefore it is enough to enumerate the saturated effective subgroups of $O(L)$.
\end{remark}

\begin{remark}\label{rem:hashimoto}
For the symplectic groups, the saturated subgroups are known:
By a theorem of Hashimoto \cite{Hashimoto2012}, there are exactly $44$ conjugacy classes of effective, saturated subgroups $H_s \leq O(L)$.
They are determined by the isometry classes of their fixed and cofixed lattices.
The fixed lattices are listed by Hashimoto while the cofixed lattices can be obtained from the permutation representation of the Mathieu group $M_{24}$ on the type $24 A_1$ Niemeier lattice.
Alternatively, one may obtain them from isometries of the Leech lattice and the tables enumerated in \cite{mason2016}.
\end{remark}

\subsection{Enumerating effective characters}
Let $H_s \leq O(L)$ be a symplectic effective subgroup.
We would like to enumerate conjugacy classes of effective characters $\chi \colon H \to \CC^\times$
with the given heart $H_s$.

\begin{definition}
Let $\chi\colon H \to \CC^\times$ be an effective character and $n=[H:H_s]$.
The distinguished generator of $H/H_s$ is $g H_s$ with $\chi(g) = \zeta_n:= \exp(2\pi i /n)$.
Set $F = L^{H_s}$ and $C= L_{H_s}$. The distinguished generator $gH_s$ restricts to an isometry $f=g|_F$ of $F$ of order $n$. We call the lattice with isometry $(F,f)$ the \emph{head} of $H$ and $H_s$ its \emph{heart}.
The \emph{spine} of $\chi$ is the glue map $\phi\colon \disc{F} \to \disc{C}$
with $L_\phi = L$.
\end{definition}

Our next goal is to see how heart, head and spine determine the character.
The first step is to make the definition of heart and head independent of a character.

\begin{definition}\label{heads and hearts}
 Let $H_s \leq O(L)$ be a heart, $F=L^{H_s}$ its fixed lattice and
 $f \in O(F)$ of order $n$. We call the lattice with isometry $(F,f)$ a head (of $H_s$) if the following hold:

 \begin{enumerate}
  \item $\ker(f+f^{-1} - \zeta_n - \bar \zeta_n)$ is of signature $(2,*)$,
  \item $(\ker \Phi_n(f)\Phi_1(f))^\perp$ does not contain any vector of square $(-2)$.
 \end{enumerate}

\end{definition}
By abuse of notation we will identify $O(C)$ with $O(C)\times \{\id_F\} \subseteq O(L\otimes \QQ)$.
Recall that for a glue map $\phi \colon \disc{F} \to \disc{C}$, we denote by $F \perp C \subseteq L_\phi$ the corresponding primitive extension.
Note that $L_\phi\cong L$ is a K3 lattice as well and $H_s$ preserves $L_\phi$ because all its elements act trivially on the discriminant group of $C$.
Suppose that $\phi D_f \phi^{-1} = D_c$ for some $c \in O(C)$.
Then $g = f \oplus c$ preserves $L_\phi$. Set  $H_\phi = \langle g , H_s \rangle$.
Since $H_s$ is saturated, any other choice of $c$ is in $cH_s$, so $H_\phi$ is independent of this choice. Let $\chi_\phi\colon H_\phi \to \CC^\times$ be defined by $\chi(g)=\zeta_n$.

\begin{definition}\label{spines}
Let $H_s$ be a heart and $(F,f)$ a head.
A glue map $\phi \colon \disc{F} \to \disc{C}$ is called a spine if
\begin{enumerate}
 \item  $\phi \circ  D_f \circ \phi^{-1}$ is in the image of $O(C) \to O(\disc{C})$ and
 \item  $\NS(\chi_\phi)_{H_\phi}$ does not contain any vector of square $(-2)$.
\end{enumerate}
\end{definition}

\begin{proposition}\label{effectivespine}
 Let $H_s$ be a heart, $(F,f)$ a head and $\phi\colon \disc{F} \to \disc{C}$ a spine.
 Then the corresponding character $\chi_\phi \colon H_\phi \to \CC^\times$ is effective.
\end{proposition}
\begin{proof}
This follows immediately from \Cref{iseffective} and the definitions.
\end{proof}
\begin{definition}
Let $i=1,2$, $H_i \leq O(L_i)$ two effective subgroups and
$\chi_i\colon H_i \to \CC^\times$ two effective characters.
We say that $\chi_1$ is isomorphic to $\chi_2$ if and only if there is an isometry $\psi \colon L_1 \to L_2$ with $H_2 = \psi H_1 \psi^{-1}$ and $\chi_1(h) = \chi_2(\psi \circ h \circ \psi^{-1})$ for all $h \in H_1$.
\end{definition}
If $L_1=L_2$, then $\chi_1$ and $\chi_2$ are isomorphic if and only if they are conjugate.
If moreover $H_1=H_2$ then they are isomorphic if and only if they are conjugate by an element of $N(H)$.
Recall that $U(F,f)$ denotes the centralizer of $f$ in $O(F)$.
\begin{theorem}\label{fiber}
Let $H_s$ be a heart, $(F,f)$ a head of $H_s$ and $S$ the set of spines
$\phi \colon \disc{F} \to \disc{C}$.
Then the double coset
\[O(C) \backslash S / U(F,f)\]
is in bijection with the set of isomorphism classes of effective characters $\chi$ with the given heart and head.
\end{theorem}
\begin{proof}
By \Cref{effectivespine} any spine $\phi \in S$ determines an effective character $\chi_\phi\colon H_\phi \to \CC^\times$ with $L_\phi$ a K3 lattice.
Conversely any effective character with the given heart and head arises in this fashion.

Let $\phi$ and $\phi'$ be two spines. We have to show that $\chi_\phi$ and $\chi_{\phi'}$ are conjugate if and only if $\phi' \in O(C) \phi U(F,f)$.

Suppose $\phi' = D_a \phi D_b$ with $a \in O(C)$ and $b \in U(F,f)$. Then $a \oplus b\colon L_\phi \to L_{\phi'}$ gives the desired isomorphism of the characters.

Conversely let $\psi\colon L_\phi \to L_{\phi'}$ be an isomorphism of $\chi$ and $\chi'$.
By construction $V:=L_\phi \otimes \QQ=L_{\phi'}\otimes \QQ$. We may view $\psi$ as an element of $O(V)$. Note that $F \perp C \subseteq V$. Since $\psi$ preserves the common heart $H_s$ and head $(F,f)$ of $\chi_1$ and $\chi_2$, we can write it as $\psi = a \oplus b$ with $a \in O(C)$ and $b \in U(F,f)$. Then $ D_a \circ \phi \circ D_b = \phi'$.
\end{proof}

\begin{remark}
The previous results yield the following procedure for determining a transversal of the set of isomorphism classes of effective characters.
  \begin{enumerate}
    \item
      Let $\mathcal{H}$ be the set of possible hearts up to conjugacy, which have been determined by Hashimoto \cite{Hashimoto2012} (see \Cref{rem:hashimoto}).
    \item
      For each $H_s \in \mathcal{H}$, determine a transversal of the isomorphism classes of the heads $(F, f)$, which amounts to classifying conjugacy classes of isometries of finite order $n$ of a given $\ZZ$-lattice.
      This is explained in \Cref{ConjAlg}, see in particular \Cref{rem:latticeK3} for the possible values of $n$.
    \item
      For each heart $H_s \in \mathcal{H}$ and possible head $(L, f)$ apply \Cref{fiber} to determine
      a transversal of the double coset of spines and therefore a transversal of the isomorphism classes of the effective characters with the given heart and head.
  \end{enumerate}
  Altogether we obtain a transversal of the isomorphism classes of effective characters.
Each is represented by some K3 lattice $L$ depending on the character, a finite subgroup $H \leq O(L)$, a normal subgroup $H_s \leq H$ and a distinguished generator of $H/H_s$.
\end{remark}

\section{Conjugacy classes of isometries}\label{ConjAlg}
We have seen that to classify finite subgroups of automorphisms of K3 surfaces up to deformation equivalence we need to classify conjugacy classes of isometries of finite order of a given $\ZZ$-lattice.
So given a polynomial $\mu(x) \in \ZZ[x]$ and a $\ZZ$-lattice $L$
we seek to classify all conjugacy classes of isometries of $L$ with the given characteristic polynomial $p$.
By general results of Grunewald and Segal \cite[Cor. 1]{grunewald:conjugacy} from the theory of arithmetic groups, we know that the number of such conjugacy classes is finite and (in theory) computable.
Necessary and sufficient conditions for the existence of an isometry of some unimodular lattice $L$ with a given characteristic polynomial have been worked out by Bayer-Fluckiger and Taelman in \cite{fluckiger20,fluckiger21,fluckiger21b}.

\subsection{Hermitian lattices and transfer}\label{transfer}
In this subsection $E$ is an \'etale $\QQ$-algebra with a $\QQ$-linear involution $\overline{\phantom{x}}\colon E \to E$.

\begin{definition}
Let $V$ be an $E$-module. A hermitian form on $V$ is a sesquilinear form $h \colon V \times V \to E$ with $h(x,y) = \overline{h(y,x)}$. We call $(V,h)$ a hermitian space over $E$.
For a $\ZZ$-order $\Lambda \subseteq E$ we call a $\Lambda$-module $L \subseteq V$ a hermitian $\Lambda$-lattice. All hermitian forms and lattices are assumed non-degenerate.
\end{definition}
Let  $\Tr \colon E \to \QQ$ be the trace.
Note that $E$ being \'etale is equivalent to the trace form $E\times E \to \QQ, (x,y) \mapsto \Tr(xy)$ being non-degenerate.

\begin{definition}[Transfer]
A $\QQ$-bilinear form $b \colon V \times V \to \QQ$ on an $E$-module $V$ is said to be an $E$-bilinear form if $b(ev,w) = b(v,\bar e w)$ for all $v,w \in V$ and $e \in E$.
Let $(V,h)$ be a hermitian space over $E$. Then $b = \Tr \circ h$ is an $E$-bilinear form. It is called the trace form of $h$.
\end{definition}

\begin{proposition}[\cite{milnor1969}]
Every $E$-bilinear form on $V$ is the trace form of a unique hermitian form on $V$.
\end{proposition}

Note that an $E$-linear map preserves a hermitian form if and only if it preserves the respective trace form. In view of these facts we may work with $E$-bilinear forms and hermitian forms over $E$ interchangeably.

Given $x^n-1=\mu(x)\in \ZZ[x]$ and a lattice $L$ we seek to classify the conjugacy classes of isometries $f\in O(L)$ satisfying $\mu(f) = 0$.
To this aim we put
$E:=\QQ[x,x^{-1}]/(\mu)\cong \QQ[x]/(\mu)$. This is an \'etale algebra and
$x \mapsto  x^{-1}$ defines an involution
$\overline{\phantom{x}} $ on $E$.
For $i \in I =: \{i \in \NN \mid i \text{ divides } n\}$ set $E_i=\QQ[x]/(\Phi_i(x))$.
The algebra $E$ splits as a direct product $E= \prod_{i \in I} E_i$
of cyclotomic fields with the induced involution.
Let $(e_i)_{i \in I}$ be the corresponding system of primitive idempotents in $E$, such that $\overline{e_i} = e_i$ and $E_i = E e_i$.

Let $\Lambda = \ZZ[x]/(\mu)$ and $\Gamma = \prod_{i \in I} \Lambda e_i$.
The conductor of $\Gamma$ in $\Lambda$ is
\[\mathfrak{f}_\Gamma= \{x \in \Lambda \mid \Gamma x \subseteq \Lambda\}.\]
It is the largest $\Gamma$-ideal contained in $\Lambda$.
We obtain the following series of inclusions:

 \begin{equation}\label{eqn:sandwich_ring}
 \mathfrak{f}_{\Gamma} \subseteq \Lambda \subseteq \Gamma = \prod_{i\in I} \Lambda e_i
 \end{equation}

\begin{lemma}
 The conductor satisfies $\mathfrak{f}_\Gamma = \bigoplus_{i \in I} (\Lambda \cap e_i \Lambda)$.
\end{lemma}

\begin{proof}
  For $x \in \Lambda$ we have $x = 1 x = \sum_i e_i x$ and so $x \Gamma = \bigoplus_{i \in I} \Lambda e_i x$ is contained in $\Lambda$ if and only if $e_i x \in \Lambda$ for all $i\in I$. This means that $e_i x \in \Lambda \cap e_i \Lambda$.
\end{proof}

\begin{example}\label{ex:conductor}
 Let $p \in \ZZ$ be a prime. For $\mu(x) = x^p-1 = (x-1)\Phi_p(x)$
 we obtain
 \[E = \QQ[x] / (x^p-1) \cong \QQ[x]/(x-1) \times \QQ[x]/\Phi_p(x) \cong \QQ \times \QQ[\zeta_p].\]
 Set $g(x) = \sum_{i=0}^{p-2}(i+1-p)x^i$.
 One finds that $p = (x-1)g(x) + \Phi_p(x)$.
 Hence $e_1=\Phi_p(x)/p$ and $e_p = (x-1)g(x)/p$.
 The conductor ideal is $\mathfrak{f}_\Gamma=p e_1 \Lambda + p e_p \Lambda$. It contains $p$.
\end{example}
Let $b$ be the bilinear form of the lattice $L$.
A given isometry $f\in O(L,b)$ with minimal polynomial $\mu$
turns $(L,b)$ into a hermitian $\Lambda$-lattice $(L,h)$ by letting the class of $x$ act as $f$.
Note that for $x,y \in L\otimes E$ we have
\[h(e_i x, e_j x) = e_i \bar e_j h(x,y) = e_i e_j h(x,y) = \delta_{ij} h(x,y).\]
Thus $e_i L$ is orthogonal to $e_j L$ for $i\neq j$.
\Cref{eqn:sandwich_ring}
yields the corresponding chain of finite index inclusions
\begin{equation}\label{eqn:sandwich}
\mathfrak{f}_{\Gamma}L \subseteq L \subseteq \Gamma L.
\end{equation}

Setting $L_i = e_i \mathfrak{f}_{\Gamma}  L = L \cap e_i L = \ker \Phi_i(f)$ and $L_i' = e_i  \Gamma L$ the outermost lattices are
\[(\mathfrak{f}_{\Gamma}L,h) =\bigperp_{i\in I}(L_i,h_i)\quad \text{ and } \quad
(\Gamma L,h) = \bigperp_{i\in I} (L_i',h_i).\]
Since $\Lambda e_i=\ZZ[x]/\Phi_i(x) = \ZZ_{E_i}$ is the maximal order in $E_i$, $L_i$ and $L_i'$ are hermitian lattices over the
ring of integers of a number field. Such lattices are well understood.
We use the outermost lattices of the sandwich to study the $\Lambda$-lattice $(L,h)$.
\begin{example}\label{ex:porder}
 Let $L$ be a $\ZZ$-lattice and $f\in O(L)$ an isometry of prime order $p$.
 Then $L_1 = \ker \Phi_1(f)$ and $L_p = \ker \Phi_p(f)$. Since, by \Cref{ex:conductor}, $p \in \mathfrak{f}_\Gamma$ we have
 \begin{equation}\label{index-p}
p L \subseteq \mathfrak{f}_\Gamma L= L_1 \perp L_p \subseteq L
 \end{equation}
\end{example}

The idea for the classification is as follows:
Given $\mu(x)$ and the $\ZZ$-lattice
$(L,b)$,
we get restrictions on the possible genera of the lattices
$(L_i,h_i)$ from $(L,b)$ and the conductor.
We take $L$ as an overlattice of the orthogonal direct sum $\bigperp_{i\in I} (L_i,h_i)$
up to the action of the product of unitary groups $\prod_{i \in I}U(L_i,h_i)$.
Thanks to \cref{eqn:sandwich}
this is a finite problem. In practice, we successively take equivariant primitive extensions.

\subsection{Glue estimates}\label{glue_estim}
The caveat of dealing with primitive extensions $A \perp B \subseteq C$ is that we do not know how to predict the genus of $C$.
Or more precisely, how to enumerate all glue maps such that $C$ lies in a given genus.
So we have to resort to check this in line \ref{keepC?} of \Cref{alg:extensions} only after constructing $C$.
To reduce the number of glue maps that have to be checked, in this section we prove various necessary conditions.

\begin{proposition}\label{glue:det-estimate}
Let $C$ be an integral $\ZZ$-lattice and $f \in O(C)$ an isometry of prime order $p$ with characteristic polynomial
$\Phi_1^{e_1}\Phi_p^{e_p}$.
Set $A= \ker \Phi_1(f)$, $B = \ker \Phi_p(f)$ and $m = \min\{e_1,e_p,l(D_A),l(D_B)\}$.
Then $pC \subseteq A \perp B$ and $[C: A\perp B] \mid p^m$.
\end{proposition}

\begin{proof}
By \Cref{index-p} $p C \subseteq A \perp B$.
Let $D_A \geq H_A \cong C/(A\perp B) =H\cong H_B \leq D_B$ be the glue between $A$ and $B$. Note that these are isomorphisms as $\ZZ[x]$-modules and $p H =0$ by \Cref{ex:porder}.

The polynomial $\Phi_p$ annihilates $B$; hence it annihilates $B^\vee$ and $H_B \leq D_B=B^\vee/B$.
Let $B' \leq p^{-1}B$ be defined by $B'/B = H_B$.
The $\ZZ[\zeta_p]$-module $B'$ is a finitely generated torsion-free module of rank $e_p$.
  By the invariant factor theorem over Dedekind domains \cite[(22.13)]{curtis2006} any torsion quotient module of $B'$ is generated by at most $e_p$ elements.
On the other hand, since $H_B \leq D_B$, it is generated by at most $l(D_B)$ elements as a $\ZZ$-module, in particular as $\ZZ[x]$-module.

Similarly $H_A$ is annihilated by $\Phi_1$ and generated by at most $\min\{e_1,l(D_A)\}$ elements.
Since the glue map is equivariant, the minimal number $n$ of generators of $H$ as $\ZZ[x]$-module satisfies $n \leq m$.
As $H=C/(A\perp B)$, viewed as a $\ZZ[\zeta_p]$-module, is annihilated by the prime ideal $P$ generated by $\Phi_1(\zeta_p)$, we have $H\cong (\ZZ[\zeta_p]/P)^n$.
Since $p$ is totally ramified in $\ZZ[\zeta_p]$, the ideal $P$ has norm $p$ and thus
\[[C:A\perp B]=|\ZZ[\zeta_p]/P|^n = p^n \mid p^m. \qedhere\]
\end{proof}

We call a torsion bilinear form $b \colon A \times A \to \QQ_2/\ZZ_2$ \emph{even }if $b(x,x)=0$ for all $x \in A$, otherwise we call it \emph{odd}.
Imitating \cite[II \S 2]{miranda-morrison}, we define the functors $\rho_k$.

\begin{definition}
  Let $N$ be an integral lattice over $\ZZ_p$. Set $G_k = G_k(N) = p^{-k}N \cap N^\vee$ and define
 $\rho_k(N) = G_k/(G_{k-1}+pG_{k+1})$. It is equipped with the non-degenerate torsion bilinear form $b_k(\bar x,\bar y) = p^{k-1}xy \bmod \ZZ_p$.
 If $p=2$ and both $\rho_{k-1}(N)$ and $\rho_{k+1}(N)$ are even, then we call $\rho_{k}(N)$ \emph{free}. Otherwise we call it \emph{bound}. If it is free, $\rho_{k}(N)$ carries the torsion quadratic form $q_k(\bar x) = 2^{k-1}x^2 \bmod 2\ZZ_2$.
\end{definition}
Let $L=\bigperp_{j=0}^l (L_j,p^jf_{j})$ be a Jordan decomposition with $f_{j}$ a unimodular bilinear form. Then one checks that
$\rho_j(L) \cong (L_j/pL_j, \bar f_i)$ where $\bar f_i$ is the composition of $p^{-1}f_i$ and the natural map $\QQ_p \to \QQ_p/\ZZ_p$.

\begin{remark}
Note that $\bar f_i$ determines the rank of $f_i$, its parity and its determinant modulo $p$. Thus, if $p$ is odd, it determines $f_j$ up to isometry. For $p=2$ this is not the case.
\end{remark}

Let $N$ be an integral lattice over $\ZZ_p$ and $l \in \ZZ$ such that
$p^{l+1} N^\vee \subseteq N$. Then $G_{l+1}(N)=G_{l+2}(N) = N^\vee$ and $G_{l}(N)=p^{-l}N \cap N^\vee$.
Using $pN^\vee \subseteq p^{-l}N$ we obtain
\[\rho_{l+1}(N)
= N^\vee /(p^{-l}N \cap N^\vee). \]
Multiplication by $p^l$ gives the isomorphism
\[\rho_{l+1}(N) \cong
p^l N^\vee/ (N \cap p^l N^\vee) \cong
p^lD_N\]

\begin{proposition}\label{glue:estimate-level}
 Let $N_1 \perp N_2 \subseteq L$ be a primitive extension of $\ZZ_p$-lattices with corresponding glue map
 $D_1 \supseteq H_1 \xrightarrow{\phi} H_2 \subseteq D_2$ where $D_i=N_i^\vee/N_i$ is the discriminant group of $N_i$, $i\in \{1,2\}$. Suppose that $p L \subseteq N_1 \perp N_2$.

 Then $p^l L^\vee \subseteq L$ if and only if the following four conditions are met.
 \begin{enumerate}
  \item $p^{l+1} D_i = 0$, i.e. $p^{l+1}N_i^\vee\subseteq N_i$,
  \item $p^l D_i \subseteq H_i$,
  \item $\phi(p^l D_1) = p^l D_2$,
  \item $\hat \phi \colon \rho_{l+1}(N_1)\cong p^lD_1 \to p^lD_2 \cong \rho_{l+1}(N_i)$ is an anti-isometry with respect to the bilinear forms $b_{l+1}$.
 \end{enumerate}
If moreover both $\rho_{l+1}(N_i)$ are free, then $\rho_l(L)$ is even if and only if
 \begin{enumerate}
\item [(4')] $\hat\phi$ is an anti-isometry with respect to the quadratic forms $q_{l+1}$.
 \end{enumerate}
\end{proposition}

\begin{proof}
Suppose that $p^l L^\vee \subseteq L$. We prove (1-4).

Let $i \in\{1,2\}$. By the assumptions
\[p^{l+1} L^\vee \subseteq p L \subseteq N_1 \perp N_2.\]
Since the extension is primitive, the orthogonal projection $\pi_i \colon L^\vee \rightarrow N_i^\vee$ is surjective. Applying $\pi_i$ to the chain of inclusions yields
\[p^{l+1} N_i^\vee \subseteq p \pi_i(L) \subseteq N_i\]
which proves (1).

For (2) consider the inclusion $p^lL^\vee \subseteq L$.
A projection yields $p^l N_i^\vee \subseteq \pi_i(L)$. Now (2) follows with $D_i = N_i^\vee / N_i$ and  $H_i = \pi_i(L)/N_i$.

(3) We have $\pi_i(p^l L^\vee) =  p^l N_i^\vee$.
Recall that the glue map $\phi$ is defined by its graph $L/(N_1 \perp N_2)$.
Its subset
\[p^l L^\vee / ((N_1 \perp N_2) \cap p^lL^\vee)\]
projects onto both $p^l D_1$ and $p^l D_2$. This proves the claim.

(4) Let $i\in{1,2}$, $x_i \in N_1^\vee$ and $y_i \in N_2^\vee$ with $\hat \phi(\bar x_i) = \bar y_i$, i.e. $\phi(p^l x_i + N_1) = p^l y_i + N_2$, i.e.
\[p^l(x_i+y_i) \in L.\]
In fact, from the proof of (3), we know a little more, namely that $p^l(x_i+y_i) \in p^l L^\vee$, so that $x_i+y_i \in L^\vee$.
This implies that $\sprodb{p^l(x_1+y_1)}{x_2+y_2} \in \ZZ_p$ which results in
\[b_{l+1}(\bar x_1, \bar x_2) \equiv p^{l}\sprodb{x_1}{x_2} \equiv - p^{l}\sprodb{y_1}{y_2} \equiv - b_{l+1}(\bar y_1, \bar y_2) \mod \ZZ_p.\]

(4') Suppose furthermore that both $\rho_{l+1}(N_i)$ are free and that $\rho_l(L)$ is even. Take $x=x_1=x_2 \in N_1^\vee$ and $y=y_1=y_2 \in N_2^\vee$.
Then $2^{l-1}\sprodbq{x+y}{} \in  \ZZ_2$ since $\rho_{l}(L)$ is even.
Therefore $2^{l}\sprodbq{x}{} \equiv 2^l\sprodbq{y}{} \mod 2 \ZZ$.

Now suppose that (1--4) hold for the glue map $\phi$.
Let $x + y \in L^\vee$. We have to show that $p^l(x+y) \in L$.

Let $w \in N_1^\vee$, and $z \in N_2^\vee$ with $\hat \phi( \bar w )=\bar z$.
By the definition of $\hat \phi$ and $\phi$, this implies that $p^l(w+z) \in L$. Therefore
\begin{eqnarray*}
b_{l+1}(\bar y - \hat \phi(\bar x),\bar z)
&= &b_{l+1}( \bar y, \bar z) - b_{l+1}(\hat \phi (\bar x), \hat \phi(\bar w))\\
&=&b_{l+1}( \bar y, \bar z) + b_{l+1}(\bar x, \bar w)\\
&\equiv& p^l\sprodb{y}{ z} + p^{l}\sprodb{x}{w} \\
&\equiv & \sprodb{x+y}{p^l(w+z)}\\
&\equiv& 0 \mod \ZZ_p.
\end{eqnarray*}
Since the bilinear form on $\rho_l(N_2)$ is non-degenerate, this shows that
$\hat \phi(\bar x)=\bar y$. By the definition of a glue map we obtain $p^l(x+y) \in L$.

Suppose furthermore that (4') holds, so $p=2$, $\rho_{l+1}(N_i)$ is free and $\hat \phi$ preserves the induced quadratic forms. Let $x+y \in L^\vee$ we have to show that $q_l(\overline{x+y})\equiv 0 \mod \ZZ$.
Indeed,
\begin{eqnarray*}
2q_l(\overline{x+y})
&=& 2^{l} \sprodbq{x+y}{}\\
&\equiv& q_{l+1}(\bar x)+q_{l+1}(\bar y)\\
&\equiv&
q_{l+1}(\bar x)+q_{l+1}(\hat \phi(\bar x)) \\
&\equiv& 0 \mod 2 \ZZ.
\qedhere
\end{eqnarray*}
\end{proof}

\begin{definition}\label{def:admissible-glue-map}
 We call a glue map \emph{admissible} if it satisfies \Cref{glue:estimate-level}
(1--3) and (4) respectively (4').
\end{definition}

\begin{example}
 In the special case that $l=0$ we recover the result that the discriminant bilinear forms of $N_1$ and $N_2$ are anti-isometric.
 And if further $L$ is even, that the discriminant quadratic forms are anti-isometric.
\end{example}

\begin{definition}
 Let $p$ be a prime number and $A,B,C$ be integral $\ZZ$-lattices. Let $p^{-l}\ZZ =\scale(C^\vee)$.
 We say that the triple $(A,B,C)$ is
 \emph{$p$-admissible} if the following hold:
 \begin{enumerate}\setlength{\itemsep3pt}
  \item $(A  \perp B) \otimes \ZZ_q \cong C \otimes \ZZ_q$ for all primes $q\neq p$,
  \item $\det A\cdot \det B=p^{2g}\det C$ where $g \leq l(D_A),(\rk B)/(p-1),l(\disc{B})$,
  \item $\scale(A \perp B)\subseteq \scale(C)$ and $p \scale(A^\vee \perp B^\vee) \subseteq \scale(C^\vee)$
  \item $\rho_{l+1}(A\otimes \ZZ_p)$ and $\rho_{l+1}(B\otimes \ZZ_p)$ are anti-isometric as torsion bilinear modules. If further $p=2$, both are free and $\rho_l(C\otimes \ZZ_2)$ is even, then they are anti-isometric as torsion quadratic modules,
  \item there exist embeddings $pC\otimes \ZZ_p \hookrightarrow (A \perp B) \otimes \ZZ_p \hookrightarrow C\otimes \ZZ_p$,
  \item $\dim \rho_{l+1}(A\otimes \ZZ_p)\oplus \rho_{l+1}(B \otimes \ZZ_p) \leq \dim \rho_l(C \otimes \ZZ_p)$.
 \end{enumerate}
\end{definition}

Note that $(C,0,C)$ and $(0,C,C)$ are $p$-admissible for all $p$.
We call a triple $(\A,\B,\C)$ of genera of $\ZZ$-lattices $p$-admissible if for any representatives $A$ of $\A$, $B$ of $\B$ and $C$ of $\C$ the triple $(A,B,C)$ is $p$-admissible.

\begin{remark}
 For the existence of the (not necessarily primitive!) embeddings in (6) there is a necessary and sufficient criterion found in \cite[Theorem 3]{omeara1958}. Note that condition (V) in said theorem is wrong.
 The correct condition is
 \[(\mathrm{V}) \qquad 2^i (1+4 \omega) \to (2^i \oplus \mathfrak{L}_{i+1})/ \mathfrak{l}_{[i]}.\]
 Thus being $p$-admissible is a condition that can be checked easily algorithmically.
\end{remark}

\begin{lemma}\label{ispadmissible}
 Let $C$ be a $\ZZ$-lattice and $f \in O(C)$ an isometry of order $p$.
 Let $A = \ker \Phi_1(f)$ and $B = \ker \Phi_p(f)$.
 Then $(A,B,C)$ is $p$-admissible.
\end{lemma}
\begin{proof}
  Let $p^{-l}\ZZ = \scale(C^\vee)$.
  First note that \Cref{glue:estimate-level} is applicable to $L = C$, since $p^{l}C^\vee \subseteq C$ by the definition of $l$.

(1) From \Cref{glue:det-estimate} we obtain $pC \subseteq A\perp B \subseteq C$. After tensoring with $\ZZ_q$ for $q\neq p$ we obtain (since $p$ is a unit in $\ZZ_q$) that $C \otimes \ZZ_q = (A\perp B)\otimes \ZZ_q$.

(2) This is \Cref{glue:det-estimate}.

(3) $A \perp B \subseteq C$ gives $\scale(A\perp B) \subseteq \scale(C)$.
Dualizing $pC \subseteq A\perp B$ yields $p (A^\vee \perp B^\vee) \subseteq C^\vee$. Now take the scales.

(4) This is \Cref{glue:estimate-level}~(4) and~(4').

(5) We know that $pC \subseteq A\perp B \subseteq C$.

(6) Let $x \in A^\vee$ and $y \in B^\vee$ with $\hat \phi(\bar x) = \bar y$, i.e. $p^l(x+y)\in C$.
By the proof of \Cref{glue:estimate-level} (4), $x+y \in C^\vee$.
Suppose that $x+y$ is zero in $\rho_l(C)$, i.e. $x+y \in p^{-l+1}C \cap C^\vee$. Then $p^{l-1} (x+y) \in C\subseteq p^{-1}(A\perp B)$, therefore
$x \in p^{-l} A$. Thus $\bar x = 0$ in $\rho_{l+1}(A)$. Similarly $\bar y = 0$ in $\rho_{l+1}(B)$. This shows that
the graph $\Gamma$ of $\hat \phi$ injects naturally into $\rho_l(C)$.
Note that $p C \subseteq A^\vee \perp B^\vee$ gives $p A^\vee \subseteq C^\vee$.
Suppose that $\bar x \neq 0$.
Since $b_{l+1}$ is non-degenerate, we find $a \in A^\vee$ with
\[1/p = b_{l+1}(\bar x, \bar a)\equiv  p^l \sprodb{x}{a} = p^{l-1}\sprodb{x}{pa} \equiv b_{l}(\overline{x+y},p a) \mod \ZZ_p.\]
This shows that the span of $p A^\vee$ and
$\Gamma$ in $\rho_l(C)=C^\vee/(p^{-l}C + C^\vee)$ is a non-degenerate subspace of dimension $2 \dim \rho_{l+1}(A)$.
\end{proof}

\begin{definition}
  Let $L$ be a $\ZZ$-lattice with $p^{l+1} L^\vee \subseteq L$ and let $H\leq \disc{L}$ with $p^{l} \disc{L} \leq H$. We denote by
 $O(H,\rho_l(L))$ the set of isometries $g$ of $H$ which preserve $p^l\disc{L}$ and such that the map $\hat g$ induced by $g$ on $\rho_L(L)$ preserves the torsion bilinear (respectively quadratic) form of $\rho_l(L)$.
\end{definition}

\begingroup
\captionof{algorithm}{AdmissibleTriples}\label{alg:admissible}
\endgroup
\begin{algorithmic}[1]
\REQUIRE A prime $p$ and a $\ZZ$-lattice $C$.
\ENSURE All tuples $(\A ,\B)$ of genera of $\ZZ$-lattices such that $(\A,\B,\C)$ is $p$-admissible with $\C$ the genus of $C$ and $\rk \B$ divisible by $p-1$.
\STATE $n \gets \rk C$
\STATE $d \gets \det(C)$
\STATE Initialize the empty list $L = [\;]$.
\FOR{ $e_p\in \{r \in \ZZ \mid 0 \leq r \leq n/(p-1)\}$}
  \STATE $r_p \gets (p-1)e_p$
  \STATE $r_1 \gets n - r_p$
  \STATE $m \gets \min\{e_p,r_1\}$
  \STATE Form the set
  \[D = \left\{ (d_1,d_p) \in \NN^2 \;\middle|\; \exists g \mid\gcd\left(d_1,d_p,p^m\right): dg^2=d_1 d_p\right\}.\]
  \FOR{$(d_1,d_p) \in D$}
    \STATE Form the set $\L_1$ consisting of all genera of $\ZZ$-lattices $A$ with
    \[\rk A = r_1, \quad \det A = d_1,\quad  \scale(A) \subseteq \scale(C),\quad  \scale(A^\vee)\subseteq p\scale(A^\vee).\]
    \\[-10pt]
    \STATE Form the set $\L_p$ consisting of all genera of $\ZZ$-lattices $B$ with
    \[\rk B = r_p, \quad \det B = d_p,\quad  \scale(B) \subseteq \scale(C),\quad  \scale(B^\vee)\subseteq p\scale(C^\vee).\]
    \\[-10pt]
    \FOR{$(\A,\B) \in \L_1 \times \L_p$}
      \IF {$(\A,\B,\C)$  is  $p$-admissible}
        \STATE Append $(\A,\B)$ to $L$.
      \ENDIF
    \ENDFOR
  \ENDFOR
\ENDFOR
\RETURN L
\end{algorithmic}\hrulefill\\

\begin{remark}
  Genera of $\ZZ$-lattices can be described by the Conway--Sloane genus symbol \cite[15 \S 7]{splag}. We have implemented an enumeration of all such genus symbols with a given signature and bounds on the scales of the Jordan components in \textsc{SageMath}~\cite{sagemath} and \textsc{Hecke}/\textsc{Oscar}~\cite{hecke}.
\end{remark}

\subsection{Enumeration of conjugacy classes of isometries.}\label{enumconj}
Let $p\neq q$ be prime numbers.
In this subsection we give an algorithm which, for a given genus of $\ZZ$-lattices $\L$, computes a complete set of representatives for the isomorphism classes of lattices with isometry $(L,f)$ of order $p^i q^j$ such that $L$ is in $\L$.

Let $(L,f)$ be a lattice with isometry. As before, we will drop $f$ from the notation and simply denote it by $L$ and the corresponding isometry by $f_L$.
If $N \leq L$ is an $f$-invariant sublattice we view it as a lattice with isometry $f_N = f|_N$.

The data structure we use for lattices with isometry is a triple
$(L,f_L,G_L)$, where $G_L$ is the image of $U(L) \to O(\disc{L})$ and $U(L)$ denotes the centralizer of $f_L$ in $O(L)$.

By abuse of terminology we call such a triple a lattice with isometry as well.
So every algorithm in this section which returns lattices with isometry actually returns such triples $(L,f_L,G_L)$ (or at least a function which is able to compute $G_L$ when needed). We omit $f_L$ and $G_L$ from notation and denote the triple simply by $L$.

\begin{definition}
 Let $A$ be a lattice with an isometry of finite order $m$.
 For a divisor $l$ of $m$ denote by $H_l$ the sublattice $\ker \Phi_{l}(f_A)$ viewed as a hermitian $\ZZ[\zeta_l]$-lattice with $\zeta_l$ acting as $f_A|H_l$. Denote by $\H_l$ its genus as hermitian lattice.
 For a divisor $l \mid m$ let $A_l = \ker (f_A^l-1)$ and denote by $\A_l$ its genus as $\ZZ$-lattice. The \emph{type} of $A$ is the collection $(\A_l,\H_l)_{l \mid m}$ and will be denoted by $t(A) = t(A, f_A)$.
\end{definition}

Since we can encode a genus in terms of its symbol and can check for equivalence of two given symbols efficiently, the type is an effectively computable invariant.

\begingroup
\captionof{algorithm}{PrimitiveExtensions}\label{alg:extensions}
\endgroup
\begin{algorithmic}[1]
\REQUIRE Lattices $A,B,C'$ with isometry such that $(A,B,C')$ is $p$-admissible.
\ENSURE
  A set of representatives of the double coset $G_B \backslash S {/}G_A$ where
$S$ is the set of all primitive extensions $A\perp B\subseteq C$ with $pC \subseteq A \perp B$ and $t(C') = t(C,f_C^p)$.
\STATE Initialize the empty list $L = [\,]$.
\STATE Let $g \in \NN$ be such that $p^{2g}  \det A  \det B = \det C'$. \label{alg:extensions_g}
\STATE Let $\mu_A$ (resp. $\mu_B$)
be the minimal polynomial of $f_A$ (resp. $f_B$).
\STATE $V_A \gets \ker \mu_B(D_{f_A}|_{(p^{-1} A \cap A^\vee)/A})$
\STATE $V_B \gets \ker \mu_A(D_{f_B}|_{(p^{-1} B \cap B^\vee)/B})$
\STATE Let $\Gr_A$ be the set of $f_A$-stable subspaces of dimension $g$ of $V_A$ containing $p^l D_A$. Define $\Gr_B$ analogously.
Form the set $R$ consisting of anti-isometric pairs $(H_A,H_B)$, i.e.
$(H_A,q_A|H_A) \cong (H_B,-q_B|H_B)$,
as $(H_A,H_B)$ runs through a set of representatives of
  $\Gr_A/G_A$ and $\Gr_B/G_B$ respectively.
\FOR{ $(H_A,H_B) \in R$}\label{HAHB}
  \STATE Compute an admissible glue map $\psi_0\colon H_A \to H_B$, see \Cref{def:admissible-glue-map}. \label{discard_psi1}
  \STATE Let $S_A\leq G_A$ (resp. $S_B \leq G_B$) be the stabilizer of $H_A$ (resp. of $H_B$).
  \STATE $S_A^{\psi_0}\gets \psi_0 \im(S_A\to O(H_B)) \psi_0^{-1}$
  \STATE $\bar f_A^{\psi_0} \gets \psi_0 (f_A|H_A)\psi_0^{-1}$
  \STATE Compute an element $g \in O(H_B,\rho_l(B))$ such that $g \bar f_A^{\psi_0} g^{-1} = f_B|H_B$. If such an element does not exist, discard $\psi_0$ and continue the for loop in line 7. \label{discard_psi2}
  \STATE $\psi \gets g \circ \psi_0$, $S_A^\psi \gets gS_A^{\psi_0}g^{-1}$
  \STATE Let $O(H_B,\rho_{l+1}(B),f_B)$ be the group of isometries of $H_B$ preserving $\rho_{l+1}(B)$ and  commuting with the action of $f_B$.
  \FOR{$S_BhS_A^\psi \in S_B \backslash O(H_B,\rho_l(B),f_B) /S_A^\psi$}
    \STATE Let $\Gamma_{h\psi}$ be the graph of $h\psi$.
    \STATE Define $C$ by $C/(A\perp B) = \Gamma_{h \psi}$.
    \STATE $f_C \gets f_A \oplus f_B$
    \IF{$t(C,f_C^p) \neq t(C')$}  \label{keepC?}
      \STATE Discard $C$.
    \ENDIF
    \STATE $S_C \gets\{(a,b) \in S_A \times S_B \mid b|_{H_B} \circ h\psi = h\psi \circ a|_{H_A}\}$ \label{preserveC}
    \STATE $G_C \gets \image\left(S_C \rightarrow O(\disc{C})\right)$
    \STATE Append $(C,f_C,G_C)$ to $L$.
  \ENDFOR
\ENDFOR
\RETURN L
\end{algorithmic}\hrulefill\\

\begin{lemma}
 Algorithm \ref{alg:extensions} is correct.
\end{lemma}

\begin{proof}
 Suppose that $A \perp B \subseteq C$ is an equivariant
 primitive extension with $pC \subseteq A \perp B$ and $t(C,f_C^p)=t(C')$. Let $\phi \colon H_A \to H_B$ be the corresponding glue map. It is admissible by \Cref{glue:estimate-level}.
 The existence of $g$ in line \ref{alg:extensions_g} follows from $(A,B,C')$ being $p$-admissible.
 Further, $p C \subseteq A \perp B$, gives $H_A \subseteq (p^{-1}A\cap A^\vee)/A$.

 Since $\phi$ is equivariant, we get that $\mu_B(G_{f_A})$ vanishes on $H_A$.
 Hence $H_A$ is a $g$-dimensional subspace of the $\FF_p$-vector space $V_A$.
 It is stable under $f_A$. Further, by \Cref{glue:estimate-level}, it contains $p^l D_A$.
 Similarly $H_B$ is preserved by $f_B$, contains $p^l D_B$ and is contained in $V_B$. Therefore $(H_A,H_B)$ appears in the for loop in line \ref{HAHB}.

 Since $(A,B,C')$ is $p$-admissible,
 there exists an admissible glue map
 $\psi_0\colon H_A \to H_B$. It can be computed using normal forms of quadratic or bilinear forms over finite fields.
 The set of admissible glue maps from $H_A$ to $H_B$ is given by $O(H_B, \rho_l(B))\psi_0$.
 There exists an admissible equivariant glue map from $H_A$ to $H_B$ if and only if we find $g \in O(H_B,\rho_l(B))$ with
 \[g \psi_0  (f_A|H_A) = (f_B|H_B) g \psi_0.\]
 Reordering we find $g  \psi_0  (f_A|H_A) \psi_0^{-1}  g^{-1} = f_B|H_B$.
 This justifies lines \ref{discard_psi1} to \ref{discard_psi2} of the algorithm. So we continue with $\psi$ an equivariant admissible glue map.
 Now the set of equivariant admissible glue maps is $O(H_B,\rho_l(B),f_B)\psi$.
 Let $h \psi$, $h \in O(H_B,\rho_l(B),f_B)$ be an equivariant admissible glue map and let $a \in S_A$ and $b \in S_B$.
 Then
 \[b h\psi a = (b|_{H_B}) h (\psi a \psi^{-1}) \psi = h' \psi.\]
 Therefore $S_B h\psi S_A \mapsto S_B h S_A^{\psi} \psi$ defines a bijection of
 \[S_B \backslash \{\mbox{equivariant admissible glue maps } \psi\colon H_A \to  H_B\}/S_A \]
 with the double coset
\[ S_B \backslash O(H_B,\rho_l(B),f_B)/S_A^\psi.\]

 Finally, the condition on $(a,b)$ in the equation for $S_C$ in line \ref{preserveC} of the algorithm
 is indeed the one to preserve $C/(A\perp B) \leq \disc{A \perp B}$.
 Thus $S_C$ is the stabilizer of $C/(A \perp B)$ in $G_A \times G_B$.
\end{proof}

\begin{remark}
  The computation of representatives and their stabilizers in Steps~6 and~9 of \Cref{alg:extensions} can be very costly.
  In \textsc{Magma}~\cite{magma}, based on the algorithms in \cite{obrien1990}, a specialized method \textsc{OrbitsOfSpaces}
for linear actions on Grassmannians is provided.
\end{remark}

\begingroup
\captionof{algorithm}{Representatives}\label{reps}
\endgroup
 \begin{algorithmic}
  \REQUIRE A lattice with isometry $A$ such that $\Phi_n(f_A)=0$ and an integer $m$,
  or just its type $t(A)$.
  \ENSURE Representatives of isomorphism classes of lattices with isometry $B$ of order $m \cdot n$ and minimal polynomial $\Phi_{mn}$ such that $t(B,f_B^m)=t(A)$.
 \end{algorithmic}\hrulefill\\

\Cref{reps} relies on an enumeration of genera of hermitian lattices over maximal orders of number fields with bounds on the determinant and level.
Then for each genus a single representative is computed \cite[Algorithm 3.5.6]{KirschmerHabil}
and its type is compared with that of $A$.
Finally, Kneser's neighbor method \cite[\S 5]{KirschmerHabil} is used to compute representatives for the isometry classes of the genus.
\begingroup
\captionof{algorithm}{Split}\label{split}
\endgroup
\begin{algorithmic}[1]
\REQUIRE A lattice with isometry $C$ such that $\Phi_{q^d}(f_C)=0$.
\ENSURE Representatives of the isomorphism classes of lattices with isometry $M$ such that $(M,f_M^p)$ is of the same type as $C$.
\STATE Initialize an empty list $L = [\ ]$.
\FOR{ $(\A_0,\B_0) \in $ AdmissibleTriples$(p,C)$} \label{splitfor1}
  \STATE $R_1\gets$ Representatives($A_0$, $\id_{A_0}$, $q^d$) where $A_0$ is any representative of $\A_0$
  \STATE $R_2\gets$ Representatives($B_0$, $\id_{B_0}$, $pq^d$) where $B_0$ is any representative of $\B_0$
  \FOR{$(A,B) \in R_1 \times R_2$} \label{split_for} \label{splitfor2}
    \STATE $E \gets \mbox{PrimitiveExtensions}(A,B,C,p)$
    \STATE Append the elements of $E$ to $L$.
    \ENDFOR
\ENDFOR
\RETURN L
\end{algorithmic}\hrulefill\\

\begin{lemma}
 \Cref{split} is correct.
\end{lemma}
\begin{proof}
Let $M$ be a lattice with isometry such that $(M,f^p_M)$ is of the same type as $C$. Then the minimal polynomial of $f_M$ is a divisor of
$\Phi_{pq^d}\Phi_{q^d}$. Let $M_{p^dq}=\ker \Phi_{pq^d}(f_M)$ and $M_{p^d}=\ker \Phi_{q^d}(f_M)$ be the corresponding sublattices. Then $(M_{pq^d},M_{q^d},C)$ is $p$-admissible by \Cref{ispadmissible} applied to $f_M^q$.
Hence their $\ZZ$-genera appear at some point in the for loop in line \ref{splitfor1}.
Similarly, at some point in the for loop in line \ref{splitfor2}, $A \cong M_{p^dq}$ and $B\cong M_{p^d}$ as hermitian lattices. Then some lattice with isometry isomorphic to $M$ is a member of $E$ by \Cref{extensions}.

  Conversely, only lattices with isometry $(M,f_M)$ with $(M,f_M^p)$ of the same type as $C$ are contained in $E$.
No two pairs $(A,B)$ in line \ref{split_for} are isomorphic and for a given pair the extensions computed  are mutually non-isomorphic by the correctness of \Cref{alg:extensions}.
Thus no two elements of $L$ can be isomorphic.
\end{proof}

\begingroup
\captionof{algorithm}{FirstP}\label{FirstP}
\endgroup
\begin{algorithmic}[1]
\REQUIRE A lattice with isometry $C$ of order $q^e$ and $b \in \{0,1\}$.
  \ENSURE Representatives of the isomorphism classes of lattices with isometry $M$ such that $(M,f_M^p)$ is of the same type as $C$. If $b=1$, return only $M$ such that $f_M$ is of order $pq^e$.
\STATE Initialize an empty list $L = [\ ]$.
\IF {$e=0$}
\RETURN Split$(C, p)$, where in case $b = 1$ we return only those lattices $M$ with $f_M$ of order $pq^e$. \label{firstp_e0}
\ENDIF
\STATE $A_0 \gets \ker(\Phi_{q^e}(f_C))$ \label{firstp1}
\STATE $B_0 \gets \ker(f_C^{q^{e-1}}-1)$ \label{firstp2}
\STATE $\A \gets  \mbox{Split}(A_0, p)$
  \STATE $\B \gets  \mbox{FirstP}(B_0,p, 0)$
\FOR{$(A,B) \in \A \times \B$} \label{firstpAB}
    \IF {$b=1$ and $p \nmid \ord(f_A)$ and $p \nmid \ord(f_B)$}
        \STATE Discard $(A,B)$ and continue the for loop with the next pair.
    \ENDIF
    \STATE $E \gets \mbox{PrimitiveExtensions}(A,B,C,q)$
    \STATE Append the elements of $E$ to $L$.
\ENDFOR
\RETURN $L$
\end{algorithmic}\hrulefill\\

\begin{lemma}
 \Cref{FirstP} is correct.
\end{lemma}

\begin{proof}
Let $M$ be a lattice with isometry such that $t(M,f_M^p)=t(C)$. Then $f_M^{pq^e}-1=0$.
Set $f=f_M^{pq^{e-1}}$. We see that $A_0= \ker(\Phi_q(f))$
and $B_0 = \ker(\Phi_1(f))$. Therefore $A_0 \perp B_0 \subseteq M$ is a primitive extension
and $(A_0,B_0,M)$ is $q$-admissible.

If $e=0$, then $f_C$ is the identity. So the call of \Cref{split} in line \ref{firstp_e0} returns the correct result.
Otherwise we proceed by induction on $e$. Note that $\Phi_{q^e}(f_{A_0})= 0$ so
the input to Split is valid. Further the order of $f_{B_0}$ is a divisor of $q^{e-1}$. Thus in line \ref{firstpAB} we have $(\Phi_{pq^e}\Phi_{q^e})(f_A)=0$
and $f_B^{pq^{e-1}} = 1$ (possibly $p \nmid \ord{f_B}$).
  Note that $(A,B,C)$ in line \ref{firstpAB} is indeed $q$-admissible, because $(A_0,B_0,C)$ is.
\end{proof}

\begingroup
\captionof{algorithm}{PureUp}\label{PureP}
\endgroup
\begin{algorithmic}[1]
\REQUIRE A lattice with isometry $C$ such that $\prod_{i=0}^e\Phi_{p^dq^i}(f_C)=0$ for $d>0$, $e\geq 0$.
\ENSURE Representatives of the isomorphism classes of lattices with isometry $M$ such that $(M,f_M^p)$ is of the same type as $C$.
\STATE initialize an empty list $L$
\IF {$e=0$}
\RETURN Representatives$(C,p)$
\ENDIF
\STATE $A_0 \gets \ker(\Phi_{p^dq^e}(f_C))$
\STATE $B_0 \gets A_0^\perp$,
\STATE $\A\gets \mbox{Representatives}(A_0,p)$
\STATE $\B = $ PureUp$(B_0,p)$
\FOR{$(A,B) \in \A \times \B$}
    \STATE $E \gets \mbox{Extensions}(A,B,C,q)$
    \STATE append the elements of $E$ to $L$.
\ENDFOR
\RETURN $L$
\end{algorithmic}\hrulefill\\

\begin{lemma}
\Cref{PureP} is correct.
\end{lemma}
\begin{proof}
If $e=0$, then $\Phi_{p^d}(f_C)=0$, so Representatives does the job.
Let $M$ be in the output of PureUp. Since $d>0$, we have $\prod_{i=0}^e\Phi_{p^{d+1}q^i}(f_M)=0$. Therefore $M$ is a valid input to PureUp and we can proceed by induction on $e$. The details are
similar to the proof of \Cref{FirstP}.
\end{proof}

\begingroup
\captionof{algorithm}{NextP}\label{NextP}
\endgroup
\begin{algorithmic}[1]
\REQUIRE A lattice with isometry $C$ of order $p^dq^e$ where $q \neq p$ are primes.
\ENSURE Representatives of the isomorphism classes of lattices with isometry $M$ of order $p^{d+1}q^e$ such that $(M,f_M^p)$ is of the same type as $C$.
\STATE initialize an empty list $L$
\IF {$d=0$}
    \RETURN FirstP($C,p,1)$
\ENDIF
\STATE $B_0 \gets \ker( f_C^{p^{d-1}q^e}-1)$
\STATE $A_0 \gets B_0^\perp = \ker \prod_{i=0}^e \Phi_{p^dq^i}(f_C)$
\STATE $\A \gets \mbox{PureUp}(A_0, p)$
\STATE $\B \gets \mbox{NextP}(B_0, p)$
\FOR{$(A,B) \in \A \times \B$}
    \STATE $E \gets \mbox{Extensions}(A,B,C,p)$
    \STATE append the elements of $E$ to $L$
\ENDFOR
\RETURN $L$
\end{algorithmic}\hrulefill\\
\begin{lemma}
\Cref{NextP} is correct.
\end{lemma}
\begin{proof}
 If $d=0$, then $f_C$ has order $q^e$. Hence we can call FirstP. For $d>0$, $A_0$ is a valid input for PureUp.
The proof proceeds by induction on $d$ since $f_{B_0}$ has order at most $p^{d-1}q^e$
\end{proof}

By calling NextP on a complete set of representatives of the types of lattices with isometry of order $p^dq^e$, we can obtain a complete set of representatives for the isomorphism classes of lattices with isometry of order $p^{d+1}q^e$.
By iterating this process we have an algorithm to enumerate representatives for all isomorphism classes of lattices with isometries of a given order $p^d q^e$.

\begin{remark}\label{rem:latticeK3}
  For the application to classifying finite groups of automorphisms of K3 surfaces we note the following:
  \begin{enumerate}
    \item
     We only enumerate those lattices with the correct signatures and discard at each stage the lattices which are negative definite and contain $(-2)$-vectors since they do not lead to isometries preserving the ample cone.
    \item
Let $G$ be a finite subgroup of automorphisms of a complex K3 surface.
Recall that $[G:G_s]=n$ satisfies $\varphi(n) \leq 20$ and $n \neq 60$.
The integers with $\varphi(n) \leq 20$ and three prime factors are
$30$, $42$ $60$, $66$ with $\varphi(n) = 8 , 12 , 16, 20$.
      Suppose $n = 66$. Since $\varphi(66)=20>12$ we have $G_s=1$. So $G$ is cyclic and we know by \cite{keum2015} that the pair is unique. We can treat $42$ with similar arguments. Finally, $30$ is treated by hand with the help of some of the algorithms described above.
  \end{enumerate}

  The actual computation was carried out using \textsc{SageMath}~\cite{sagemath}, \textsc{Pari}~\cite{PARI}, \textsc{GAP}~\cite{GAP4} and \textsc{Magma}~\cite{magma}.
\end{remark}

\subsection{Computation of the group $G_L$}\label{subsec:computeGL}
The algorithms of the previous section for enumerating isomorphism classes of lattices with isometry
require as input for each lattice with isometry $L$ the group $G_L$, which is the image of
the natural map $U(L) \to U(D_L)$. Recall that we use a recursive approach 
and for primitive extensions $C$ of lattices with isometry $A$ and $B$, the group
$G_C$ can be determined using $G_A$ and $G_B$ (see \Cref{alg:extensions}). It is therefore sufficient to explain how $G_L$ can computed for the lattices constructed in \Cref{reps}, which form the base case of the recursive strategy.
We therefore consider lattices with isometry $L$ such that $f_L$ has irreducible minimal polynomial and $\ZZ[f_L]$ is the maximal order of $\QQ[f_L]$.
We distinguish the following four cases:

\begin{enumerate}
  \item
    The lattice $L$ is definite. Then $O(L)$ is finite and can be computed using an algorithm of Plesken and Souvignier \cite{Plesken-Souvignier}.
  \item
    The lattice $L$ is indefinite of rank $2$ and $f_L = \pm 1$. For this situation the computation of $G_L$ will be explained in the remainder of this section.
  \item
    The lattice $L$ is indefinite of rank $\geq 3$ and $f_L = \pm 1$. In this case, Miranda--Morison theory \cite{miranda-morrisonI, miranda-morrisonII} along with some algorithms by Shimada \cite{shimada2018} solve the problem.
    A short account of this is given in \Cref{mmtheory}.
  \item
    The automorphism satisfies $f_L \neq \pm 1$. This will be addressed in \Cref{hermitianmm},
    where we extend the theory of Miranda--Morison to the hermitian case.
\end{enumerate}

We end this section by describing the computation of $G_L$ in case~(2).
Therefore let $L$ be be an indefinite binary lattice over $\ZZ$ and $V = L \otimes \QQ$ the ambient quadratic space of discriminant $d \in \QQ^\times/(\QQ^\times)^2$.

It follows from \cite[\S 5]{Eichler1974} that we may assume that $V$ is a two-dimensional étale $\QQ$-algebra and $L \subseteq V$ is a $\ZZ$-lattice of rank $2$.
More precisely, $V$ is isomorphic to the Clifford algebra $C^+$, which in turn is isomorphic to $\QQ(\sqrt{d})$. The $\QQ$-algebra $V$ is a quadratic extension of $\QQ$ if and only if $d$ is not a square, which is the case if and only if $V$ is anisotropic. If $d$ is a square, then $V \cong \QQ \times \QQ$.
If $\sigma \colon V \to V$ denotes the non-trivial automorphism of $V$ as a $\QQ$-algebra,
then the quadratic form $q$ on $V$ is given by $q(x) = x \sigma(x)$ for $x \in V$.
Note that $\sigma \in O(V)$ and $\det(\sigma) = -1$.

Every element $y \in V$ induces an endomorphism $\tau_y \colon V \to V, \, x \mapsto yx$ of determinant $\det(\tau_y) = y \sigma(y) = q(y)$.
For a subset $X \subseteq V$ set $X^1 = \{ x \in X \mid q(x) = 1 \}$.
The proper automorphism group $SO(V)$ of $V$ is equal to $\{ \tau_y \mid y \in V^1\}$
and $O(V) = \{ \tau_y, \sigma \tau_y \mid y \in V^1 \}$.
We call two $\ZZ$-lattices $I, J$ of $V$ equivalent, if there exists $\alpha \in V^1$ such that $I = \alpha J$.
Finally set
\[
  \Lambda = \{ x \in V \mid xL \subseteq L \},
\]
which is a $\ZZ$-order of $V$.

\begin{proposition}
  The following hold:
  \begin{enumerate}
    \item
      We have $SO(L) = \{ \tau_y \mid y \in (\Lambda^\times)^1 \}$.
    \item
      If $L$ is not equivalent to $\sigma(L)$, then $O(L) = SO(L)$.
    \item
      If $L$ is equivalent to $\sigma(L)$, say $L = \alpha \sigma(L)$, then
      $O(L) = \langle SO(L), \sigma \tau_{\sigma(\alpha)} \rangle$.
  \end{enumerate}
\end{proposition}

\begin{proof}
  First note that for $y \in V$ we have $\tau_y(L) \subseteq L$ if and only if $y \in \Lambda$.
  This shows part~(1).

  Any isometry of $L$ extends uniquely to an isometry of $V$ and is---if the determinant is not $1$---thus of the form $\sigma \tau_\alpha$ for some $\alpha \in V$.
  Hence $L = (\sigma \tau_\alpha)(L) = \sigma(\alpha) \sigma(L)$, that is, $L$ and $\sigma(L)$ are equivalent.
  This shows part~(2).

  Now assume that $L = \alpha \sigma(L)$.
  Then $(\sigma \tau_{\sigma(\alpha)})(L) = L$ and thus $\sigma \tau_{\sigma(\alpha)} \in O(L) \setminus SO(L)$.
  If $\sigma \tau_{\sigma(\beta)} \in O(L)$ is any non-proper isometry, then $\sigma(\beta) L = \sigma(L) = \sigma(\alpha) L$ and thus $\sigma(\beta \alpha^{-1}) \in (\Lambda^\times)^1$.
  This shows part~(3), since $\sigma \tau_{\sigma(\beta)} = \sigma \tau_{\sigma(\alpha)}\tau_{\sigma(\beta \alpha^{-1})}$.
\end{proof}
\begin{remark}
  We briefly describe how the previous result can be turned into an algorithm for determining generators of $O(L)$ for an indefinite binary lattice.
  We may assume that the ambient space $V$ is an étale $\QQ$-algebra of dimension two.
  The group $\Lambda^\times$ is a finitely generated abelian group and generators
  can be computed as described in~\cite{Bley2005, Kluners2005}.
  Given generators of $\Lambda^\times$, determining generators of $(\Lambda^\times)^1$ is just a kernel computation.
  Finally, testing whether two $\ZZ$-lattices of $V$ are equivalent can be accomplished using \cite{Bley2005, Marseglia2020}.
\end{remark}

\section{Quadratic Miranda--Morrison theory} \label{mmtheory}

In this section we review classical Miranda--Morrison theory for even indefinite $\ZZ$-lattices $L$ of rank at least $3$,
as introduced by Miranda and Morrison in \cite{miranda-morrisonI, miranda-morrisonII}. Akyol and Degtyarev~\cite{sextics} incorporated sign structures to study connected components of the moduli spaces of plane sextics. We follow their example.

The purpose of this is twofold.
First this allows us to sketch the computation of the image of 
\[O(L) \to O(\disc{L})\]
settling case (3) in \Cref{subsec:computeGL}.
Second by incorporating the action on the sign structure, we obtain a way to
compute the image of
\[O^+(L) \to O(\disc{L})\]
which yields the number of connected components of the moduli space $\F_H$ (see \Cref{quadruplewithtriple} and \Cref{rem:mmapplication}).

We denote by $\AA$ %
the ring of finite adeles and by $\ZZ_\AA$ %
the ring of finite integral adeles.
For a ring $R$ set $\Gamma_R= \{\pm 1\} \times R^\times/(R^\times)^2$.
We define $O^\sharp(L\otimes R)$ as the kernel of $O(L \otimes R) \to O(D_L \otimes R)$.
We note that $\Gamma_\QQ$ has a natural diagonal embedding into $\Gamma_\AA$ and $D_L \cong D_L \otimes \ZZ_\AA\cong \disc{L \otimes \ZZ_\AA}$ naturally.

The homomorphisms
\[\sigma_p \colon O(L\otimes \QQ_p) \to \Gamma_{\QQ_p}, \quad g \mapsto (\det(g),\spin(g)).\]
induce a homomorphism
\[ \sigma\colon O(L \otimes \AA) \to \Gamma_{\AA}.\]
Let $\Sigma^\sharp(L \otimes {\ZZ_p})$ be the image of $O^\sharp(L \otimes {\ZZ_p})$
under $\sigma_p$.
We set $\Sigma(L) =\sigma(O(L\otimes \ZZ_\AA))=\prod_p \Sigma(L \otimes \ZZ_p)$
and $\Sigma^\sharp(L) = \sigma(O^\sharp(L \otimes \ZZ_\AA))=\prod_p \Sigma^\sharp(L\otimes \ZZ_p)$.
By \cite[VII 12.11]{miranda-morrison} we have $\Sigma(L\otimes \ZZ_p)=\Sigma^\sharp(L \otimes {\ZZ_p})=\Gamma_{\ZZ_p}$ whenever $L\otimes {\ZZ_p}$ is unimodular.
By \cite[IV.2.14 and IV.5.9]{miranda-morrison} the natural map $O(L\otimes \AA) \to O(\disc{L})$ is surjective.
The following commutative diagram with exact rows and columns summarizes the situation (where by abuse of notation we denote restriction of $\sigma$ by $\sigma$ as well).
\[
\begin{tikzcd}
  1\arrow[r]& O^\sharp(L\otimes \ZZ_\AA) \arrow[r]\arrow[d] &O(L\otimes \ZZ_\AA) \arrow[r]\arrow[d,"\sigma"]& O(\disc{L})\arrow[r]\arrow[d] & 1\\
1\arrow[r]& \Sigma^\sharp(L)\arrow[d] \arrow[r] &\Sigma(L)\arrow[r]\arrow[d] &\Sigma(L)/\Sigma^\sharp(L)\arrow[d] \arrow[r] &1\\
 &1 &1 &1 &
\end{tikzcd}\]

If $V$ is an indefinite $\QQ$-lattice of rank $\geq 3$, then the restriction $O(V) \to \Gamma_{\QQ}$ of $\sigma$ is surjective by \cite[VIII 3.1]{miranda-morrison}.

\begin{theorem}\cite[VIII 5.1]{miranda-morrison}\label{thm:miranda-morrison}
Let $L$ be an indefinite $\ZZ$-lattice of rank at least $3$. Then we have the following exact sequence
\[1 \to O^\sharp(L) \to O(L) \to O(D_L) \xrightarrow{\bar \sigma} \Sigma(L)/(\Sigma^\sharp(L)\cdot (\Gamma_\QQ \cap \Sigma(L))) \to 1. \]
\end{theorem}

We need an analogous sequence with $O(L)$ replaced by $O^+(L)$ to compute connected components of the coarse moduli space of $H$-markable K3 surfaces.
Define $\Gamma_\QQ^+$ as the kernel of $\Gamma_\QQ \to \{\pm 1\}$, $(d,s) \mapsto \sign(ds)$. Then for any indefinite $\QQ$-lattice $V$ of rank $\geq 3$ the homomorphism $\sigma^+\colon O^+(V) \to \Gamma_\QQ^+$ is surjective.

\begin{theorem}\label{thm:miranda-morrison-signed}
 Let $L$ be an indefinite $\ZZ$-lattice of rank at least $3$. Then we have the following exact sequence
\[O^+(L) \xrightarrow{D_+} O(D_L) \xrightarrow{\bar \sigma_+} \Sigma(L)/(\Sigma^\sharp(L)\cdot (\Gamma_\QQ^+ \cap \Sigma(L))) \to 1. \]
 \end{theorem}
 \begin{proof}
We prove $\ker \bar\sigma_+ \subseteq \im D_+$:
Let $\bar g \in O(D_L)$ and suppose that $\bar\sigma_+(\bar g) = 1$.
This means that $\bar g$ lifts to an element $g \in O(L\otimes \ZZ_\AA)$ with $D_g=\bar g$ and
$\sigma(g)\in \Sigma^\sharp(L)\cdot (\Gamma_\QQ^+ \cap \Sigma(L))$.
After multiplying $g$ with an element in $O^\sharp(L\otimes \ZZ_\AA)$, we may assume that
\[\sigma(g) \in \Gamma_\QQ^+ \cap \Sigma(L).\]
Hence there exists an element $h \in O(L\otimes \QQ)$ with $\sigma(h)=\sigma(g)$.

Since $\sigma(h^{-1}g)=1$ and $h^{-1}g(L\otimes \ZZ_p)=L \otimes \ZZ_p$ for all but finitely many primes, we can use the strong approximation theorem (see e.g.\cite[5.1.3]{KirschmerHabil}, \cite{kneser1966}) to get $f \in O(L\otimes \QQ)$ with $\sigma(f)=1$ and $f(L \otimes {\ZZ_p}) = h^{-1}g(L \otimes {\ZZ_p})$ at all primes and approximating $h^{-1}g$ at the finitely many primes dividing the discriminant. This yields (once the approximation is good enough) $D_f = D_{h^{-1}g}$ (cf \cite[VIII 2.2]{miranda-morrison}).

By construction, $h f \in O(L \otimes \QQ)$ preserves $L$
and \[D_{hf} = D_{h} \circ D_f = D_{h} \circ D_{h^{-1}g} = D_g\]
as desired.
We have $\sigma(hf) = \sigma(h)  \in \Gamma_\QQ^+$. So $hf \in O^+(L)$.

We prove $\ker \bar \sigma_+ \supseteq \im D_+$: Let $ g \in O^+(L)$. Then $\sigma(g) \in \Gamma^+_\QQ$ and since $O(L) \subseteq O(L\otimes \ZZ_\AA)$ we have $\sigma(g) \in \Sigma(L)$ as well.
 \end{proof}

The group $\Sigma(L)$ appearing in \Cref{thm:miranda-morrison,thm:miranda-morrison-signed} is infinite and infinitely generated. However its quotient
by $\Sigma^\sharp(L)$ is a finite group. We explain how to write it in terms of finite groups only, so that it may be represented in a computer.
Let $T$ be the set of primes with $\Sigma^\sharp(L\otimes \ZZ_p)=\Sigma(L\otimes \ZZ_p)=\Gamma_{\ZZ_p}$ and $S$ its complement. We know that $S$ is contained in the set of primes dividing $\det L$.
We can project the quotient $\Sigma(L)/\Sigma^{\sharp}(L)$ isomorphically
to a subquotient of the finite group $ \Gamma_S' = \prod_{p \in S} \Gamma_{\QQ_p}$ and are thus reduced to a finite computation.

Indeed, denote by $\pi_S\colon \Gamma_\AA \to \Gamma_S'$ the natural projection.
Set
\[\Sigma_S(L) = \pi_S(\Sigma(L)), \quad \Sigma^\sharp_S(L)= \pi_S(\Sigma^\sharp(L)),\quad
\Gamma_T=\prod_{p \in T} \Gamma_{\ZZ_p},\]
\[\Gamma_S = \pi_S(\Gamma_\QQ\cap  \Gamma_T \times \Gamma_S') \qquad
\Gamma_S^+=\pi_S(\Gamma^+_\QQ\cap \Gamma_T \times \Gamma_S').\]

$\Gamma_S^+$ is spanned by the images of $\{(1,p) \mid p \in S\} \cup \{(-1,-1)\} \subseteq \Gamma_\QQ^+\subseteq \Gamma_\AA$ under $\pi_S$. Since $\Gamma_\QQ/\Gamma_\QQ^+$ is spanned by $(-1,1)$,
$\Gamma_S$ is spanned by the generators of $\Gamma_S^+$ together with $\pi_S((-1,1))$.

\begin{proposition}
The projection $\pi_S$ induces isomorphisms
\[\Sigma(L)/(\Sigma^\sharp(L)\cdot (\Gamma_\QQ^+ \cap \Sigma(L)))\cong
\Sigma_{S}(L)/(\Sigma^\sharp_S(L)\cdot (\Gamma_S^+ \cap \Sigma_S(L)))\]
and
\[\Sigma(L)/(\Sigma^\sharp(L)\cdot (\Gamma_\QQ \cap \Sigma(L)))\cong
\Sigma_{S}(L)/(\Sigma^\sharp_S(L)\cdot (\Gamma_S \cap \Sigma_S(L))).\]
\end{proposition}
\begin{proof}
Note that $\pi_S(\Gamma_\QQ^+ \cap \Sigma(L)) = \Gamma_S^+ \cap \Sigma_S(L)$.
Therefore
\[K:=\pi_S^{-1}(\Sigma_S^\sharp(L)\cdot(\Gamma^+_S \cap \Sigma_S(L)))=\Sigma^\sharp(L)\cdot(\Gamma_\QQ^+ \cap \Sigma(L)).\]
Hence the surjection
\[\psi \colon \Sigma(L) \to
\Sigma_{S}(L)/(\Sigma^\sharp_S(L)\cdot (\Gamma_S^+ \cap \Sigma_S(L)))\]
induced by $\pi_S$ has kernel $K$. We conclude by applying the homomorphism theorem to $\psi$.
To prove the second isomorphism remove the $+$.
\end{proof}

The groups $\Sigma_S(L)$ and $\Sigma_S^\sharp(L)$ are found in the tables in \cite[VII]{miranda-morrison} in terms of the discriminant form of $L$ and its signature pair.

\begin{proposition}\label{plusindexcompute}
 Let $L$ be an indefinite $\ZZ$-lattice of rank at least $3$ and $J$ a subgroup of the image of the natural map $D \colon O(L) \to O(\disc{L})$. Set $J^+ = D(O^+(L)) \cap J$ and let $K = \ker D=O^\sharp(L)$.
 Then $[J:J^+]=| \sigma_+(J)|$ and \[[K:K^+]= [\Gamma_\QQ \cap \Sigma^\sharp(L): \Gamma_\QQ^+ \cap \Sigma^\sharp(L)] =
 [\Gamma_S \cap \Sigma^\sharp_S(L): \Gamma_S^+ \cap \Sigma^\sharp_S(L)].\]
\end{proposition}
\begin{proof}
We have $J^+ = \ker (\bar \bar \sigma_+) \cap J = \ker (\bar \sigma_+|_J)$.
Therefore $J/J_+ \cong \sigma_+(J)$.

The strong approximation theorem implies the equality
$\sigma(O^\sharp(L)) = \Sigma^\sharp(L) \cap \Gamma_\QQ$ and that
$\sigma(O^\sharp(L)^+) = \Sigma^\sharp(L) \cap \Gamma_\QQ$.
\end{proof}

\begin{remark}
 The theorems allow us to compute the image of $O(L) \to O(D_L)$ for $L$
 an indefinite $\ZZ$-lattice of rank at least $3$. Namely, one computes generators of $O(D_L)$ and lifts them $p$-adically to elements of $L\otimes \ZZ_p$ with sufficient precision. Then one can use these lifts to compute their spinor norm. See the work of Shimada \cite{shimada2018} for further details. An Algorithm for $p$-adic lifting and generators for $O(D_L)$ are given in \cite{brandhorst-veniani2023}.
\end{remark}

\section{Hermitian Miranda--Morrison theory}\label{hermitianmm}
Let $L$ be a lattice with isometry with irreducible minimal polynomial.
In this section we use the transfer construction to compute the image of
$U(L) \to D(\disc{L})$, thus settling case (4) of \Cref{subsec:computeGL}.  To this end we develop the analogue of Miranda--Morrison theory for hermitian lattices over the ring of integers of a number field.

\subsection{Preliminaries on hermitian lattices.}
In this section we recall some basics on hermitian lattices over the ring of integers of a number field or a local field. See \cite{KirschmerHabil} for an overview of the theory.

Let $K$ be a finite extension of $F=\QQ$ (global case) or $F=\QQ_p$ (local case) and $E$ an étale $K$-algebra of dimension $2$.
Let $\O$ be the maximal order of $E$ and $\o$ be the maximal order of $K$.

\begin{definition}
 For $E$ an \'etale $K$-algebra, we denote by
 $\Tr^E_K \colon E \to K$ the trace and by
  $\FD_{E/K}^{-1}=\{ x \in E \mid \Tr^E_K(x \O) \subseteq \o \}$
 the inverse of the different.
\end{definition}

In the local case, we let $\FP\subseteq \O$ be the largest ideal invariant under the involution of $E/K$ and $\fp$ the maximal ideal of $\o$. Define $e$ by $\FP^e = \FD_{E/K}$ and $a$ by $\FP^a = \FD_{E/F}$.

\begin{definition}
 Let $(L,h)$ be a hermitian $\mathcal O$-lattice. Its scale is the ideal $\scale(L)=h(L,L)\subseteq \O$ and its norm is the ideal $\norm(L)=\sum \{ \sprodq{x}{}\o \mid x \in L \} \subseteq \o$.
\end{definition}
It is known that
\begin{equation}\label{norm-vs-scale}
 \FD_{E/K} \scale(L) \subseteq \norm(L) \subseteq \scale(L).
\end{equation}

\subsection{The trace lattice}
In this subsection $(L,h)$ is a hermitian $\O$-lattice. By transfer we obtain its trace $\ZZ_F$-lattice $(L,b)$ with $b = \Tr^E_{F} \circ h$.
Our primary interest is in even $\ZZ$-lattices. So our next goal is to give necessary and sufficient conditions for the trace lattice to be integral and even.

The hermitian dual lattice is
\[L^\sharp = \{x \in L \otimes E \mid h(x,L) \subseteq \O\}\]
and $L^\vee = (L,b)^\vee$ is the dual lattice with respect to the trace form.

\begin{proposition}\label{trace-integral}
We have
\[L^\vee = {\FD_{E/F}^{-1}} L^\sharp.\]
The trace form on $L$ is integral if and only if $\scale(L) {\FD_{E/F}} \subseteq \O$.
\end{proposition}
The proof is left to the reader.
We continue by determining the parity of the trace lattice.
To this end,
we first establish that the transfer construction behaves well with respect to completions, as expected.

For a place $\nu$ of $K$ we use the following notation:
For an $\o$-module $M$ denote by $M_\nu = M \otimes \o_\nu$ the completion of $M$ at $\nu$ and similar for $K$-vector spaces.

\begin{proposition}\label{prop:tracedecompose}
  Let $F = \QQ$, $E/K$ a degree two extension of number fields and $p$ a prime number.
  Then
  \begin{equation}\label{decomp-local}
    (L,\Tr^E_\QQ \circ h) \otimes \ZZ_p = \bigperp_{\nu \mid p} (L_\nu, \Tr^{E_\nu}_{\QQ_p} \circ h_\nu),
  \end{equation}
  where $\nu$ runs over all places extending the $p$-adic place of $\QQ$.
\end{proposition}
\begin{proof}
To this end consider the canonical isomorphism $K \otimes \QQ_p \cong \prod_{\nu \mid p} K_\nu$ where the product runs over the prolongations of the $p$-adic valuation to $K$.
By \cite[Chap. II, \S3, Prop. 4]{Serre1979} this induces a canonical isomorphism
\[\o \otimes \ZZ_p \cong \prod_{\nu \mid p} \o_\nu\]
where $\o_\nu$ is the maximal order of $K_\nu$.
We obtain a corresponding canonical decomposition (using a system of primitive idempotents)
\begin{equation}\label{decomp-localp}
(L,h) \otimes \ZZ_p = \bigperp_{\nu \mid p} (L_\nu,h_\nu).
\end{equation}
  where each summand is a hermitian $\O_\nu$-lattice.
  Note that $\O_\nu$ is indeed the maximal order of $E_\nu$.
The decomposition $E \otimes \QQ_p \cong \prod_{\nu \mid p} E_\nu$
and viewing $\Tr$ as the trace of the left multiplication endomorphism shows that
\[\Tr^{E \otimes \QQ_p}_{\QQ_p}= \sum_{\nu \mid p} \Tr^{E_\nu}_{\QQ_p}.\]
Therefore the trace commutes with the decomposition in \cref{decomp-localp}.
\end{proof}

\begin{lemma}\label{normhelp}
  Let $K$ be a non-archimedian dyadic local field of characteristic $0$.
  Assume that $B \subseteq \o$ is a $\ZZ$-module such that $\Tr^E_K(\O) \subseteq B$, $1 \in B$ and $\Nr^E_K(\O)B \subseteq B$. Then $\o \subseteq B$.
\end{lemma}

\begin{proof}
If $E/K$ is split or unramified, then $\Tr^E_K(\O)=\o$. So let $E/K$ be ramified.
  By~\cite[\S 6]{Johnson1968} there exists $u_0 \in K$ such that $\o^\times = \Nr^E_K(\O^\times) \cup (1+u_0) N^E_K(\O^\times)$ and $u_0 \o = \fp^{e-1}$. Since $\fp^{e-1} \subseteq \fp^{\lfloor\frac{e}{2}\rfloor}=\Tr^E_K(\O)$ by \cite[Ch. V,\S3, Lemma 3]{Serre1979}, we have $(1+u_0) \Nr^E_K(\O^\times) \subseteq B$.
  Thus $\o^\times \subseteq B$ and therefore $\Nr^E_K(\O) \o^\times \subseteq B$. As $E/K$ is ramified it follows that $\o= \Nr^E_K(\O) \o^\times.$
\end{proof}

The following proof is inspired by \cite[3.1.9]{michael2015}.
\begin{proposition}\label{eventracelocal}
Let $K$ be a non-archimedian local field of characteristic $0$ and $(L,h)$ a hermitian $\O$-lattice.
The trace form $\Tr^E_F \circ h$ is even if and only if
  $\norm(L) \subseteq \FD_{K/F}^{-1}$.
\end{proposition}

\begin{proof}
  Suppose that $\norm(L)  \subseteq {\FD_{K/F}^{-1}}$.
We have
  \[\Tr^E_F(h(x,x))= 2 \Tr^K_F(h(x,x)) \in 2 \Tr^K_F(\norm(L)) \subseteq 2 \Tr^K_F({\FD_{K/F}^{-1}})=2 \ZZ_F.\]

Now suppose that the trace form is even. In particular it is integral.
We may assume that $F = \QQ_2$ is dyadic.
Let $B$ be the set of all $\omega \in K$ such that
$\Tr^{E}_F(\omega h(x,x))\subseteq 2\ZZ_2$ for all $x \in L$. Then $\lambda \bar \lambda \omega h(x,x)=\omega h(\lambda x,\lambda x)$ for $\lambda \in \O$ gives $\Nr(\O) B \subseteq B$ and $1 \in B$. We calculate
  \[\Tr^E_F((\lambda +\bar \lambda) h(x,x)) = \Tr^E_F (\lambda h(x,x))+ \Tr^E_F (\overline{ \lambda h(x,x)}) = 2 \Tr^E_F(h(\lambda x,x)) \in 2\ZZ_2 \]
  for all $\lambda \in \O$.
This gives $\Tr^E_K(\O) \subseteq B$.
By \Cref{normhelp} $\o \subseteq B$.
  Therefore $\Tr^K_F(\o h(x,x))\subseteq \ZZ_2$ which means $h(x,x)\in {\FD_{K/F}^{-1}}$.
\end{proof}

We show that the same result holds in the global setting.

\begin{corollary}\label{trace}
Let $F=\QQ$ and $E/K$ a degree two extension of number fields. Let $(L,h)$ be a hermitian $\O$-lattice.
The trace form $\Tr^E_F \circ h$ is even if and only if
$\norm(L) \subseteq {\FD_{K/F}^{-1}}$.
\end{corollary}

\begin{proof}
We use \Cref{prop:tracedecompose} and 
note that the orthogonal sum is even if and only if each summand is even.
We apply \Cref{eventracelocal} to each summand and obtain the condition that
$\norm(L_\nu) \subseteq \FD^{K_\nu}_{\QQ_2}$ for all $\nu \mid 2$.
We conclude with $\nu(\norm(L)) = \nu(\norm(L_\nu))$.
\end{proof}

Note that even implies integral and similarly ${\FD_{E/K}} \scale(L) \subseteq \norm(L)\O \subseteq {\FD_{K/F}^{-1}}\O$ implies that $\scale(L) \subseteq {\FD_{E/F}^{-1}}$ which matches up perfectly with \Cref{trace-integral} and~\Cref{trace}.

\subsection{Discriminant form and transfer}
Suppose that $L \subseteq L^\vee := {\FD_{E/F}^{-1}} L^\sharp$. Define a torsion hermitian form on $D_L=L^\vee/L$ as follows
\[\bar h\colon D_L \times D_L \to E/{\FD_{E/F}^{-1}}, \quad ([x],[y]) \mapsto [h(x,y)].\]
Suppose further that $\norm(L) \subseteq {\FD_{K/F}^{-1}}$, that is, the trace form on $L$ is even. Then we define the
torsion quadratic form
\[\bar q\colon D_L \to K{/}{\FD_{K/F}^{-1}},\quad [x] \mapsto [h(x,x)].\]
Note that $\Tr^E_K({\FD_{E/F}^{-1}})={\FD_{K/F}^{-1}}$.
For a lattice $L$ with even trace form we set $U(D_L)$ to be the group of $\O$-linear automorphisms of $D_L$ preserving $\bar q$.

\begin{proposition}
Let $(L,h)$ be a hermitian $\O$-lattice with even trace form and
 $g \in U(D_L)$. Then $g$ preserves $\bar h$.
\end{proposition}

\begin{proof}
 For $x \in L^\vee$ let $gx$ denote a representative of $g(x+L)$.
  Let $x,y \in L^\vee$. Then
  \[\Tr^E_K h(x,y) = h(x+y,x+y) - h(x,x) - h(y,y).\]
 Set $\delta(x,y) = h(x,y) - h(gx,gy)$. We have to prove that $\delta(x,y) \in {\FD_{E/F}^{-1}}$.
  Since $g$ preserves $\bar q$,
  \[\Tr^E_K( \delta(x,y)) \in {\FD_{K/F}^{-1}} =\Tr^E_K({\FD_{E/F}^{-1}}). \]
 By the $\O$-linearity of $g$ we have $\alpha gx - g\alpha x \in L$ for all $\alpha \in \O$.
  Therefore, using $L^\vee = {\FD_{E/F}^{-1}}L^\sharp$, we have for any $\alpha \in \O$ that $\alpha \delta(x,y) \equiv \delta(\alpha x,y) \mod {\FD_{E/F}^{-1}}$. Hence
  \[\Tr^E_K(\O \delta(x,y)) \subseteq {\FD_{K/F}^{-1}}.\]
  This means that $\delta(x,y) \in {\FD_{K/F}^{-1}}{\FD_{E/K}^{-1}}={\FD_{E/F}^{-1}}$.
\end{proof}

Recall that for an even lattice with isometry $(L,b,f)$ we have defined $U(D_L)$ as the centralizer of $D_f$ in $O(D_L)$. By the transfer construction in \Cref{transfer} we may view $L$ as a hermitian lattice as well.
The following proposition reconciles the two definitions of $U(D_L)$.

\begin{proposition}\label{UDLtransfer}
  Let $(L,b,f)$ be an even lattice with isometry with irreducible minimal polynomial and $(L,h)$ the corresponding hermitian $\ZZ[f]$-lattice.
  Let $E \cong \QQ[f]$, $K\cong \QQ[f+f^{-1}]$ and $\O$ be the maximal order of $E$. Suppose that $(L,h)$ is invariant under $\O$, that is, $(L, h)$ is a hermitian $\O$-lattice (this is true if $\ZZ[f] = \O$).
 Then $U(D_L)$ is the centralizer of $D_f$ in $O(D_L)$.
\end{proposition}

\begin{proof}
It is clear that $U(D_L)$ centralizes $D_f$. So let $g \in O(D_L)$ centralize $D_f$. This implies that $g$ is $\O$-linear.
It remains to show that $g$ preserves $\bar q$.
Since $D_L = \bigperp_{\nu} D_{L_\nu}$ and $U(\disc{L})=\prod_{\nu}U(\disc{L_\nu})$, we may assume that $K$ is complete.

For $x + L \in D_L$ write $gx$ for a representative of $g(x+L)$.
  Set $\delta = h(x,x)-h(gx,gx)$. We have to show that $\delta \in \Tr^E_K({\FD_{E/F}^{-1}})=\FD_{K/F}^{-1}$.
Since $g$ preserves the discriminant form $q_{(L,b)}$, we have
$2\Tr^K_F(\delta)=\Tr^E_F(\delta) \in  2 \ZZ_F$.
Let $B$ be the set of all $\omega \in \o$ such that $ \Tr^K_F(\omega \delta) \in   \ZZ_F$. As in the proof of \Cref{eventracelocal} one sees that $\Nr(\O) B \subseteq B$, $1 \in B$ and $\Tr^E_K(\O) \subseteq B$. Then \Cref{normhelp} provides $B = \o$.
Thus $\Tr^K_F(\O \delta) \subseteq \ZZ_F$, i.e. $\delta \in {\FD_{K/F}}$.
\end{proof}

\begin{remark}
 \Cref{UDLtransfer} provides a practical way to compute $U(D_L)$. We can write down a system of generators for $O(D_L)$. Then
 the computation of a centralizer is a standard task in computational group theory.
\end{remark}

\subsection{Local surjectivity of $\mathbf{U(L) \to U(D_L)}$}\label{subsec:localsurjectivity}
In this subsection we assume that $K$ is a non-archimedian local field of characteristic $0$, $\pi$ a prime element of $\O$, $p = \pi \bar \pi$,
and $L$ a hermitian $\O$-lattice with $\norm(L)\subseteq \FD_{K/F}^{-1}$, that is, its trace lattice is even.
Recall that $\FP\subseteq \O$ is the largest ideal invariant under the involution of $E/K$, $\fp$ the maximal ideal of $\o$ and the integers $a, e$ satisfy $\FP^e = \FD_{E/K}$ and $\FP^a = \FD_{E/F}$.

If $E/K$ is a ramified field extension then by \cite[Ch. V,\S3, Lemma 3]{Serre1979} we have for all $i \in \ZZ$ that $\Tr(\FP^i) = \fp^{\lfloor \frac{i+e}{2}\rfloor}$. Therefore $\Tr(\FP^{1-e})=\o$ and $\Tr(\FP^{2-e})=\fp$. So there exists $\rho \in E$ with $\rho \O = \FP^{1-e}$ and $\Tr(\rho)=1$. If $E/K$ is an unramified field extension, then we find $\rho \in \O^\times$ with $\Tr(\rho)=1$. If $E = K \times K$, then we can take $\rho = (1, 0) \in \o\times \o=\mathcal{O}$ which satisfies $\Tr(\rho) = 1$ as well.

For a hermitian matrix $G \in E^{n \times n}$ set $\scale(G)=\scale(L)$ and $\norm(G)=\norm(L)$, where $L$ is the free $\O$-lattice with gram matrix $G$.

\begingroup
\captionof{algorithm}{Hermitian lift}\label{alg:lift}
\endgroup
\begin{algorithmic}[1]
\REQUIRE
$0\leq l \in \ZZ$, $\rho \in E$,
$G =\bar G^t \in E^{n\times n}$,
$F \in \GL_n(\O)$ such that
\begin{itemize}
 \item $\Tr^E_K(\rho) =1$,
\item  $\scale(G^{-1})\subseteq \FP^{1+a}$, $\rho\norm(G^{-1})\subseteq \FP^{1+a}$,
\item $R:=G - FG\bar F^T$ with $\scale(R) \subseteq \FP^{l-a}$, $\rho\norm(R) \subseteq \FP^{l-a}$.
\end{itemize}
  \ENSURE $F' \in \GL_n(\O)$ such that for $l' = 2l + 1$ and $R' = G - F'G \bar F'^t$ the following hold
\begin{itemize}
  \item $F' \equiv F \mod \FP^{l} \O^{n \times n}\pi^{-a}G^{-1} \subseteq\FP^{l+1}\O^{n \times n}$,
  \item $\scale(R') \subseteq \FP^{l'-a}$, $\rho \norm(R') \subseteq \FP^{l'-a}$.
\end{itemize}
\STATE $R \gets G - F G \bar F^t$
\STATE Write $R = U + D + \bar U^t$ with $U$ upper triangular and $D$ diagonal.
\RETURN $F + (U + \rho D)\bar F^{-t}G^{-1}$
\end{algorithmic}\hrulefill \\

\begin{theorem}
\Cref{alg:lift} is correct.
\end{theorem}
\begin{proof}
 With $X=(U + \rho D)\bar F^{-t}$ and $F' = F + XG^{-1}$, we calculate
 \begin{eqnarray*}
 F' G \bar F'^t%
 &=& F G \bar F^t + U + \bar U^t + \Tr^E_K(\rho)D + XG^{-1} \bar X^t\\
 &= & G + X G^{-1} \bar X^t.
 \end{eqnarray*}
  Hence $R' = -XG^{-1}X^t$.
Since $R \equiv 0 \mod \scale(R)$, $D\equiv 0 \mod \norm(R)$, we have
  $U + \rho D \equiv 0 \mod \scale(R)+\rho \norm(R) \subseteq \FP^{l-a}$, hence $U + \rho D \equiv 0 \mod  \FP^{l-a}$. Together with $F \in \GL_n(\O)$ this implies
\[X \equiv 0 \mod \FP^{l-a}.\]
Hence
\[\scale(R')=\scale(X G^{-1} \overline{X}^t) \subseteq \FP^{2l-2a}\scale(G^{-1})\subseteq\FP^{2l+1-a}\]
and
\[\rho \norm(R') = \FP^{2l-2a}\rho\norm(G^{-1})\subseteq\FP^{2l+1-a}=\FP^{2l+1-a}.\]
It remains to show that $F' \in \GL_n(\O)$.
Since $\scale(G^{-1})\subseteq \FP^{1+a}$, we have $\FP^l \O^{n \times n} \pi^{-a}G^{-1} \subseteq \FP^{l+1}\O^{n \times n}$. Therefore $F \equiv F' \mod \FP^{l+1}$.
\end{proof}

\begin{theorem}\label{Usurj}
  Let $K$ be a non-archimedian local field of characteristic $0$ and $L$ a hermitian $\O$-lattice with even trace lattice.
 Then
  $U(L) \to U(D_L)$
 is surjective.
\end{theorem}
\begin{proof}
 We take an orthogonal splitting $L = M \perp N$ with $M$ being
 ${\FD_{E/F}^{-1}}$-modular and $\scale(N) \subsetneq {\FD_{E/F}^{-1}}$.
 Then ${\FD_{E/F}^{-1}} L^\sharp/L \cong {\FD_{E/F}^{-1}} N^\sharp/N$.
 After replacing $L$ with by $N$ we may and will assume that $ \scale(L) \subsetneq {\FD_{E/F}^{-1}}$.

 Recall that $\FP^a =\FD_{E/F}$.
  Identify $L^\vee={\FD_{E/F}^{-1}}L^\sharp$ with $\O^n= \O^{1 \times n}$ by choosing a basis. Let $G$ be the respective gram matrix of ${\FD_{E/F}^{-1}}L^\sharp$. We have $L = \O^n \pi^{-a}  G^{-1}$ and $p^{-a} G^{-1}$ is the corresponding Gram matrix of $L$.
 Therefore
 \[\scale(G^{-1}) \subseteq \FP^{1+a} \text{ and }\norm(G^{-1})\subseteq \FP^{e+a}.\]

 If $E/K$ is unramified or split, then $e=0$, $\scale(G^{-1})=\norm(G^{-1})$ by \cref{norm-vs-scale} and we find $\rho \in \mathcal{O}$ with $\Tr(\rho)=1$. Therefore $\rho\norm(G^{-1}) =\rho \scale(G^{-1}) \subseteq \scale(G^{-1})\subseteq \FP^{1+a}$ holds.
 If $E/K$ is ramified we find $\rho \in E$ with $\rho \O = \FP^{1-e}$. Then
 $\rho \norm(G^{-1}) \subseteq \FP^{1-e+e+a}=\FP^{1+a}$ as well.

 Let $f \in U(\disc{L})$ be represented by $F \in \GL_n(\O)$,
 that is,
 \[f(x + L) = x F + L = x F + \O^{n} \pi^{-a}G^{-1} .\]
 Set $R = G - F G \bar F^t$. Since $f$ preserves $\bar h$ and $\bar q$, we have
 \[\scale(R) \subseteq \FP^{-a} \text{ and } \rho\norm(R) \subseteq \rho\mathcal{D}_{K/F}^{-1}\subseteq\rho\FP^{e-a}\subseteq\FP^{-a}.\]
 where in the last equality we used that $\rho\FP^e \subseteq \mathcal{O}$ irrespective of $E/K$ being inert, split or ramified.
 Set $F_0 = F$. We inductively define a sequence by setting $F_{i+1}$ to be the output of
 \Cref{alg:lift} with $l = 2^i-1$, $F \gets F_i$, $G \gets G$ as given and
  $R \gets R_i:= G - F_i G \bar F_i^t$. 
 Then $\scale(R_i) \subseteq \FP^{2^i-1 -a }$ and $\rho\norm(R_i) \subseteq \FP^{2^i-1-a}$.
 Since $F_i \equiv F_{i+1} \mod \FP^{2^i}$, the sequence $(F_i)_{i \in \NN}$ converges.
 Its limit is the desired lift.
\end{proof}

For a hermitian lattice $L$ with even trace lattice, we denote by $U^\sharp(L)$ the kernel of $U(L) \to U(\disc{L})$.

\subsection{Local to global}
Let $E/K$ be a quadratic extension of number fields with non-trivial automorphism $\bar{\phantom{x}} \colon E \to E$.
Let $\o$ be the maximal order of $K$ and $\O$ the maximal order of $E$.
In this section $L$ is a hermitian $\O$-lattice with even trace lattice.
The goal of this subsection is to compute
the image of the natural map
\[D\colon U(L) \rightarrow U(\disc{L}).\]

Denote by $\AA_K$ the ring of finite adeles of $K$.
Denote by $\o_{\AA}$ the ring of integral finite adeles of $K$.
We have natural isomorphisms $\disc{L} \cong \disc{L} \otimes \o_{\AA} \cong \disc{L \otimes \o_{\AA}}$.
Via the diagonal embedding we view $K$ as a subring of $\AA_K$.
This induces the inclusion $U(V)\subseteq U(L \otimes \AA_K)$.
Let $\det \colon U(L \otimes \AA_K) \to \prod_\fp E_\fp$ denote the componentwise determinant.
Set
\[\F(E) = \{(x)_\fp \in \prod_{\fp} E_\fp \mid x \in E,\,  x \bar x = 1\},\]\\[-25pt]
\begin{align*}
&\F(L_\fp) = \det(U(L_\fp)), &&\F^\sharp(L_\fp)=\det(U^\sharp(L_\fp)), \\
&\F(L) =\det(U(L \otimes \o_{\AA})), & &\F^\sharp(L) = \det(U^\sharp(L\otimes \o_{\AA}))
\end{align*}
Note that $U(L \otimes \o_{\AA}) \to U(\disc{L})$ is surjective by \Cref{Usurj}.
The following commutative diagram
\[
\begin{tikzcd}
  1\arrow[r]& U^\sharp(L\otimes \o_{\AA}) \arrow[r]\arrow[d,"\det"] &U(L\otimes \o_{\AA}) \arrow[r]\arrow[d,"\det"]& U(\disc{L})\arrow[r]\arrow[d] & 1\\
1\arrow[r]& \F^\sharp(L)\arrow[d] \arrow[r] &\F(L)\arrow[d]\arrow[r] & \F(L)/\F^\sharp(L)\arrow[d]\arrow[r]&1\\
& 1   &1 &1\\
\end{tikzcd}\]
with exact rows and columns summarizes the situation.

\begin{proposition}\label{kernel:weak-approx}
Let $V$ be a non-degenerate hermitian space over $E/K$.
Then $\det(U(V))=\F(E)$.
\end{proposition}
\begin{proof}
 We know $\det(U(V))\subseteq \F(E)$. The other inclusion is clear when $\dim_E V$ is one.
 Since we can always split a subspace of dimension one, the statement follows.
\end{proof}
Set $\O_\AA:= \O \otimes \o_\AA$.
For an isometry $f\colon L \to M$ of hermitian $\O_{\AA}$-lattices, we denote by $D_f=(D_{f_\fp})_\fp$ the induced map on the discriminant forms.
Let $L \subseteq L^\vee$ be a hermitian $\O_{\AA}$-lattice. Then $\sigma \in O(L \otimes \AA_K)$ induces an isometry $\sigma \colon L \to \sigma(L)$ of $\O_{\AA}$-lattices and an isometry
$D_\sigma\colon \disc{L} \to \disc{\sigma(L)}$ of the respective discriminant groups.

\begin{proposition}\label{kernel:global}
 Let $L$ be an indefinite hermitian $\O$-lattice with $\rk(L)\geq 2$.
  For $\sigma \in U(L\otimes \AA_K)$ the following are equivalent:
 \begin{enumerate}
   \item There is a map $\varphi \in U(L \otimes K)$ such that $D_\varphi=D_\sigma$ and $\varphi(L \otimes \o_{\AA})=\sigma(L\otimes \o_{\AA})$.
  \item $\det(\sigma) \in \F(E)\cdot \F^\sharp(L)$.
 \end{enumerate}
\end{proposition}
\begin{proof}
  First suppose that a map $\varphi$ as in (1) exists.
  Since $\varphi(L\otimes \o_{\AA})=\sigma(L\otimes \o_{\AA})$ and $D_\varphi = D_\sigma$, we have $\varphi^{-1} \circ \sigma \in U^\sharp(L\otimes \o_{\AA})$.
  Thus
\[\det(\sigma) \in\F(E)\cdot \F^\sharp(L).\]

Now suppose that $\det(\sigma) \in \F(E)\cdot \F^\sharp(L)$.
Then there exists $u \in \F(E)$ and $\rho \in U^\sharp(L \otimes {\o_{\AA}})$ such that
$\det(\sigma) = u \det(\rho)$. By \Cref{kernel:weak-approx} there exists $\psi \in U(L\otimes K)$ with $\det(\psi)=u$.
Let $\phi := \psi^{-1} \circ \sigma \circ \rho^{-1}$.
Then $\det(\phi)=1$.
By the strong approximation theorem \cite{kneser1966}, there exists
$\eta \in U(L\otimes K)$ with $\eta(L \otimes \o_{\AA})=\phi(L\otimes \o_{\AA})$ and $D_\eta = D_\phi$
(approximate $\phi$ at the finitely many primes dividing the discriminant and those with $\phi_\fp(L_\fp)\neq L_\fp$).
Set $\varphi := \psi \circ \eta \in U(L \otimes K)$.
Then
\[\varphi(L\otimes \o_{\AA}) = (\psi \circ \eta)(L\otimes \o_{\AA}) = (\psi \circ \phi)(L\otimes \o_{\AA}) = (\sigma \circ \rho^{-1})(L\otimes \o_{\AA})=\sigma(L\otimes \o_{\AA}).\]
  Further
\[D_{\varphi} =D_{\psi} \circ D_{\eta}= D_{\sigma} \circ D_{\rho}^{-1} = D_{\sigma}\]
since $D_{\rho}$ is the identity because $\rho\in U^\sharp(L \otimes \o_{\AA})$.
\end{proof}

\begin{theorem}\label{hermitianUtoUDL}
 Let $L$ be an indefinite hermitian $\O$-lattice with $\rk(L)\geq 2$. Then there is an exact sequence
 \[U(L)\rightarrow U(\disc{L}) \xrightarrow[]{\delta} \F(L)/(\F(E)\cap \F(L))\cdot \F^\sharp(L) \rightarrow 1\]
 where $\delta$ is induced by the determinant.
\end{theorem}
\begin{proof}

Let $\hat\gamma \in U(\disc{L})$ and lift it to some $\gamma \in U(L \otimes \o_{\AA})$ with $D_{\gamma}=\bar \gamma$.
By \Cref{kernel:global}, $\gamma$ lies in $U(L\otimes K)$ if and only if $\det(\gamma) \in \F(E)\cdot \F^\sharp(L)$
which is equivalent to
\[\det(\gamma) \in (\F(E)\cdot \F^\sharp(L)) \cap \F(L)= (\F(E) \cap \F(L)) \cdot \F^\sharp(L) \]
and this does not depend on the choice of lift $\gamma$ of $\bar \gamma$.
We conclude with the general fact that $U(L) = U(L \otimes K) \cap U(L \otimes \o_{\AA})$.
\end{proof}
In order to make \Cref{hermitianUtoUDL} effective,
we will compute the groups $\F^\sharp(L_\fp)$ and $\F(L_\fp)$ in  \Cref{kergen}. See \Cref{detUsharp} for the exact values.

\begin{remark}\label{rem:hermitianUtoUDL}
Let $S$ be the set of primes of $K$ dividing the order of $\disc{L}$.
For practical purposes we note that $\F(L_\fp)=\F^\sharp(L_\fp)$ for the primes not in $S$.
Hence $\F(L)/\F^\sharp(L)\cong \prod_{\fp \in S}\F(L_\fp)/\F^\sharp(L_\fp)$ and it is enough to compute $\delta_\fp$ for the primes in $S$.
This can be achieved by lifting $\bar \gamma \in U(\disc{L_\fp})$ to some $\gamma \in U(L_\fp)$ with sufficient precision using \Cref{alg:lift}.
\end{remark}

\subsection{Generation of $\mathbf{U^\sharp(L)}$ by symmetries.}\label{kergen}

Let $K$ be a finite extension of $F= \QQ_p$ and $E/K$ a ramified quadratic extension. Let $\Tr = \Tr^E_K$ be the trace.
Recall that $\FP\subseteq \O$ is the largest ideal invariant under the involution of $E/K$, $\fp$ the maximal ideal of $\o$ and the integers $a, e$ satisfy $\FP^e = \FD_{E/K}$ and $\FP^a = \FD_{E/F}$.
Note that as $E/K$ is ramified we have $a \equiv e \mod 2$.
Let $\pi \in \O$ be a prime element and $p = \pi \bar \pi$.
For any $v \equiv e \mod 2$,
there exists a skew element $\omega \in E^\times$ with $\nu_\FP(\omega)=v$.

Let $V$ be a non-degenerate hermitian space over $E$.
In what follows $L$ is a full $\O$-lattice in $V$ with even trace form.
Therefore its scale and norm satisfy
\[\scale(L)=:\FP^i\subseteq \FP^{-a}\text{ and }\norm(L)=:\fp^k\subseteq \FD_{K/F}^{-1} = \FP^{e-a}.\]
This gives the inequalities $0 \leq i+a$ and $0 \leq 2k+a-e$ and by \cref{norm-vs-scale} $i \leq 2k \leq i+e$.
We say that $L$ is \emph{subnormal} if $\norm(L)\O\subsetneq \scale(L)$, i.e., $i<2k$. A sublattice of rank two is called a plane and a sublattice of rank one a line.
By \cite[Propositions 4.3, 4.4]{Jacobowitz1962}, the lattice $L$ decomposes into an orthogonal direct sum of lines and subnormal planes.

The group $U^\sharp(L)$ is the kernel of the natural map
\[U(L) \rightarrow U(\disc{L}).\]
For $\varphi \in U(L)$ we have $\varphi \in U^\sharp(L)$ if and only if $(\varphi - \id_L)(L^\vee) \subseteq L$.
For $x,y \in L$ we write $x \equiv y \mod \FP^i$ if $x-y \in \FP^i L$.

We single out the elements of $U(V)$ fixing a hyperplane -- the symmetries.
\begin{definition}
  Let $V$ be a hermitian space, $s \in V$ and $\sigma \in E^\times$ with $\sprodq{s}{s}=\Tr(\sigma)$. We call the linear map
  \[S_{s,\sigma} \colon V \to V, \quad x \mapsto x -\sprod x s \sigma^{-1}s\]
  a \emph{symmetry} of $V$. It preserves the hermitian form $\sprod{\cdot\,}{\cdot}$.
  If $s$ is isotropic, then we have $\det(S_{s, \sigma}) = 1$ and
  otherwise $\det(S_{s,\sigma})=-\overline\sigma/\sigma$.
  The inverse is given by $S_{s,\sigma}^{-1}=S_{s,\bar \sigma}$.
  Note that the symmetry $S_{s, \sigma}$ of $V$ preserves $L$
  if $s \in L$ and $\sprod L s \subseteq \O \sigma$.
  We denote the subgroup of $U(L)$ generated by the symmetries preserving $L$ by $S(L)$ and set $S^\sharp(L) = U^\sharp(L) \cap S(L)$.
 \end{definition}

By \cite{brandhorsthofmann2021} symmetries generate the unitary group $S(L)=U(L)$ if $\O/\FP \neq \FF_2$. Otherwise one has to include so called rescaled Eichler isometries, which are isometries fixing subspaces of codimension $2$.
Fortunately, as we will see, symmetries suffice to generate $U^\sharp(L)$. The condition that the trace form on $L$ is even eliminates all the technical difficulties of \cite{brandhorsthofmann2021}.

\begin{lemma}\label{ker:congruence}
 Let $\varphi \in U^\sharp(L)$ and $x \in L$.
  Then $\varphi(x) - x \in \sprod{x}{L} \FP^a L$.
\end{lemma}
\begin{proof}
For any $x \in L$, the inclusion
$\sprod{x}{L}^{-1}x \subseteq L^\sharp$ gives
\[\FP^{-a}\sprod{x}{L}^{-1} x \subseteq \FP^{-a}L^\sharp = L^\vee .\]
  Hence, $(\varphi(x) -x)\sprod{x}{L}^{-1}\FP^{-a} \subseteq (\varphi - \id_V)(L^\vee) \subseteq L$.
  Multiply by the ideal $\sprod{x}{L}\FP^{a}$ to reach the conclusion.
\end{proof}

\begin{lemma}\label{lem:refl-ker}
  Let $S_{s, \sigma}$ be a symmetry of $V$ with $s \in \FP^{i+a} L$. Then $S_{s, \sigma}
  \in U^\sharp(L)$ if $\FP^{2i+a} \subseteq \sigma \O$.
\end{lemma}
\begin{proof}
  We have  $(\id_V-S_{s,\sigma})(\FP^{-a}L^\sharp) = \sprod{\FP^{-a}L^\sharp}{s}\sigma^{-1}s \subseteq \FP^{i} \sigma^{-1} s \subseteq L$.
\end{proof}

\begin{lemma}\label{line}
 Let $x, x' \in L$ with $\sprodq{x}{x}=\sprodq{x'}{x'}$,
 $\sprod{x}{L}=\sprod{x'}{L}=\FP^i$ and $x \equiv x' \mod \FP^{i+a}$.
  Then there is an element $\varphi \in S^\sharp(L)$ with $\varphi(x) = x'$.
\end{lemma}
\begin{proof}
 Note that $\sprod{x}{x-x'} \in \FP^{2i+a}$.
  If $\sprod{x}{x-x'}\O =\FP^{2i+a}$. Then with $\sigma = \sprod{x}{x-x'}$ and $s = x - x'$ we have $S_{s, \sigma}(x) = x'$, and \Cref{lem:refl-ker} implies that 
  $S_{s,\sigma} \in U^\sharp(L)$.

 If $\sprod{x}{x-x'}\O \subseteq \FP^{2i+a+1}$,
 choose $s \in \FP^{i+a}L$ with
 \[\sprod{s}{x}\O = \sprod{s}{x'}\O = \FP^{2i+a} \]
 which is possible since $\sprod{x}{L}=\sprod{x'}{L}=\FP^i$.
 We have $\nu_\FP(\sprodq s s \rho)\geq 2i+2a+2k+1-e > 2i + a$ and $2i+a \equiv e \mod 2$.
 With $\omega \in E$ a skew element of valuation $2i+a$, $\sigma := \sprodq{s}{s}\rho + \omega$ satisfies $\nu_\FP(\sigma)=2i+a$.
 By \Cref{lem:refl-ker} we have
 $S_{s,\sigma} \in S^\sharp(L)$.
  Then \[\sprod{x}{x - S_{s,\sigma}(x')}=\sprod{x}{x-x'} + \sprod{x}{s}\sprod{s}{x'}\bar \sigma^{-1}\]
  gives $\sprod{x}{x - S_{s,\sigma}(x')}\O = \FP^{2i+a}$. Further $x \equiv S_{s,\sigma}(x') \mod \FP^{i+a}$. Thus by the first case we can map $x$ to $S_{s,\sigma}(x')$.
\end{proof}

\begin{lemma}\label{plane}
  Let $L =  P \perp M$, with $P$ a subnormal plane.
  Then $U^\sharp(L) = S^\sharp(L)U^\sharp(M)$.
\end{lemma}
\begin{proof}
  By~\cite{Jacobowitz1962} there exists a basis $u, v \in P$ with $\sprodq{u}{u} = p^k$, $\sprodq{v}{v}\in \fp^{k}$ and
  $\sprod{u}{v}=\pi^i$. Note that $L$ subnormal implies $i<2k$. Let $\varphi \in U^\sharp(L)$.
  By \Cref{line} there exists a symmetry $S \in S^\sharp(L)$ with $S(u) = \varphi(u)$. Therefore we may and will assume that $\varphi(u) = u$.

  Write $\varphi(v) = \gamma u + \delta v + m$ for some $m \in \FP^{i+a}L$ and $\gamma,1-\delta \in \FP^{i+a}$.
  Then we have
  \[\sprod{v}{v-\varphi(v)}\O=(-\bar \gamma \bar \pi^i+(1-\bar\delta)\sprodq{v}{v})\O\]
  \begin{equation}\label{deltaunit}
    \sprod{u}{v-\varphi(v)} = \sprod{u}{v}-\sprod{\varphi(u)}{\varphi(v)} = 0.
  \end{equation}
  The symmetry $S_{s,\sigma} \in U(L\otimes E)$ with $s=v - \varphi(v)$ and $\sigma = \sprod{v}{v-\varphi(v)}$ preserves $u$ and maps $v$ to $\varphi(v)$.
  If $\nu_\FP(\gamma) = i+a$, then
  \[ \nu_\FP((1 - \bar \delta)\sprodq v v)  \geq i + a + 2k > 2i + a = \nu_\FP(-\bar \gamma \bar \pi^i). \]
  Thus $\sprod{v}{v-\varphi(v)}\O = \FP^{2i+a}$.
  It follows that $S_{s,\sigma} \in S^\sharp(L)$ by \Cref{lem:refl-ker} and we are done.

  Let now $\nu_\FP(\gamma)> i +a$. We
  consider $v' = u - \pi^i p^{k-i}v \in L$. It satisfies
  \[\sprod u {v'} = 0, \quad \sprod{v'}{v}\equiv \pi^i \mod \FP^{i+1} \quad \text{and} \quad v' \equiv u \mod \FP.\]
  In particular, $v_\FP(\sprod {v'}v) = i$.
  Set $s = \pi^{i+a}v'$ and let $\omega\in E$ be a skew element such that $\nu_\FP(\omega) = 2i + a$.
  Since $\nu_\FP(\sprodq{s}{s} \rho) %
  > 2i+a$, the element $\sigma = \rho \sprodq{s}{s} + \omega$ satisfies $\Tr(\sigma) = \sprodq{s}{s}$ and
  $\nu_\FP(\sigma) = 2i + a$.

  We have $S_{s,\sigma} \in S^\sharp(L)$,
  $S_{s,\sigma}(u) = u$ and $S_{s,\sigma}(\varphi(v)) = \gamma' u + \delta' v + w$
  with $\gamma',\delta' \in \O$ and
  \[\gamma' = \gamma - \sprod{\gamma u + \delta v}{\pi^{i+a}v'}\sigma^{-1} \pi^{i+a}
  = \gamma - \delta \sprod{v}{v'} p^{i+a}\sigma^{-1}.\]
  Since $\delta \in \O^\times$ by \cref{deltaunit} and $\nu_\FP(\gamma)>i+a$, we have
  $\nu_\FP(\gamma') = i+a$. We conclude as in the first case.
\end{proof}

\begin{theorem}\label{ker:ramif}
 Let $E/K$ be ramified. Then we have $U^\sharp(L) = S^\sharp(L)$.
\end{theorem}
\begin{proof}
We proceed by induction on the rank of $L$. We know that
$L = M \perp N$ with $M$ a line or a subnormal plane. By \Cref{line,plane} we have $U^\sharp(L) = S^\sharp(L)U^\sharp(N)$. By induction $U^\sharp(N)=S^\sharp(N)$.
\end{proof}

\begin{remark}
For $E/K$ unramified or $E=K \times K$ one can prove that $U^\sharp(L)= S^\sharp(L)$ as well. Since we do not need this result for the computation of $\det(U^\sharp(L))$, the proof is omitted.
\end{remark}

\subsection{Determinants of the kernel}
We use the same assumptions and notation as in \Cref{subsec:localsurjectivity}.
In particular we are in the local setting.
Let $\delta \in E$ be of norm $\delta \bar \delta = 1$ and $x \in V$ be anisotropic.
A \emph{quasi-reflection} is a map of the form
\[\tau_{x,\delta}\colon V \rightarrow V, y \mapsto y + (\delta-1)\frac{\sprod{y}{x}}{\sprodq{x}{x}} x.\]
We have $\tau_{x,\delta} \in U(V)$ and $\det(\tau_{x,\delta}) = \delta$.
Let $s = x$ and $\sigma = \sprodq{x}{x}(1-\delta)^{-1}$.
Then $\tau_{x,\delta}=S_{s,\sigma}$.
Conversely, if $s$ is anisotropic and $\sigma \in E$ with $\Tr(\sigma)=\sprodq{s}{s}$, set $\delta = - \bar \sigma / \sigma$, then
$S_{s,\sigma} = \tau_{x,\delta}$. Thus the quasi-reflections are exactly the symmetries at anisotropic vectors. The symmetries at isotropic vectors are called transvections.

\begin{lemma}\label{ker:refl}
Let $x \in L$ be primitive, anisotropic and $\delta \in E$ of norm $\delta \bar \delta = 1$. Then
$\tau_{x,\delta} \in U^\sharp(L)$ if and only if
$(\delta - 1) \in \FP^{a} \sprodq{x}{x}.$
\end{lemma}
\begin{proof}
 We have $\tau_{x,\delta} \in U^\sharp(L)$ if and only if
 $(\tau_{x,\delta}-\id_V)(L^\vee)\subseteq L$. This amounts to $(\delta-1)\FP^{-a}\sprodq{x}{x}^{-1}x \in L$.
 The lemma follows since $x$ is primitive.
\end{proof}

For $i\geq 0$ set
\begin{eqnarray*}
 \E_0 & = &\{u \in \O^\times \mid u \bar u =1\}\\
\E^i &=& \{u \in \E_0 \mid u \equiv 1 \mod \FP^i\}.
\end{eqnarray*}
Note that $\E_0 = \E^0 = \E^{e-1}$, $\E_1:= \{u \bar u ^{-1} \mid u \in \O^\times\} = \E^e$ and $[\E_0:\E_1]=2$ by \cite[3.4, 3.5]{Kirschmer2019}.

\begin{theorem}\label{detUsharp}
 Let $F=\QQ_p$, $K/F$ a finite field extension, $E/K$ an étale $K$-algebra of dimension $2$ with absolute different $\FP^a:=\FD_{E/F}$.
 Suppose that $L$ is a hermitian $\O$-lattice with $\fp^k:=\norm(L) \subseteq \FD_{K/F}^{-1}$.
 Then $\det(U^\sharp(L)) =:\F^\sharp(L) = \E^{2k+a}$.
\end{theorem}
\begin{proof}
 Let $x \in L$ be a norm generator, i.e. $\sprodq{x}{x} \o = \norm(L) = \fp^k$.
 Let $\delta \in \E^{2k+a}$. Then
 \[(\delta-1) \in \FP^{a}\sprodq{x}{x}\subseteq \FP^{a}\norm(L).\]
 By \Cref{ker:refl} we have $\tau_{x,\delta} \in U^\sharp(L)$ and so $\delta \in \det(U^\sharp(L))$. Hence
 \[\E^{2k+a} \subseteq \det(U^\sharp(L)).\]

 Let $\FP^i:=\scale(L)$ and $\varphi \in U^\sharp(L)$. By \Cref{ker:congruence} $\varphi \equiv \id \mod \FP^{i+a}$. Thus $\det(\varphi) \equiv 1 \mod \FP^{i+a}$, that is, $\det(\varphi) \in \E^{i+a}$.

 If $E/K$ is unramified, then $\scale(L)=\norm(L)\O$, so that $i=2k$ and
 \[\det(U^\sharp(L)) \subseteq \E^{i+a}=\E^{2k+a}.\]

 Now suppose that $E/K$ is ramified.
 By \Cref{ker:ramif} the group $U^\sharp(L) = S^\sharp(L)$ is generated by symmetries.
 Since transvections have determinant one,
 it is enough to consider the determinants of the quasi-reflections in $U^\sharp(L)$.
 Let $\tau_{x,\delta} \in U^\sharp(L)$ be a quasi-reflection. Recall that $\det(\tau_{x,\delta}) = \delta$.
 By \Cref{ker:refl} we have $(\delta -1) \in \FP^a \sprodq{x}{x} \subseteq \FP^a \norm(L)$.
 This proves $\det(U^\sharp(L)) \subseteq \E^{2k + a}$.
\end{proof}

\subsection{Computing in $\F(L_\fp)/\F^\sharp(L_\fp)$}
Let $E/K$ be a quadratic extension number fields with rings of integers $\O$ and $\o$ respectively. Let $L$ a hermitian $\O$-lattice.
Determining the image of $\delta$ in \Cref{hermitianUtoUDL} requires the computation in the finite quotient $\F(L_\fp)/\F^\sharp(L_\fp)$, where $\fp$ is a prime ideal of $\o$ (see also \Cref{rem:hermitianUtoUDL}).
To simplify notation we now assume that $K$ is a local field of characteristic $0$, $E/K$ an étale $K$-algebra of dimension $2$ and the notation as in \Cref{kergen}.
Hence our aim is to be able to do computations in $\F(L)/\F^\sharp(L)$.
As we are only interested in computing in the abelian group as opposed to determining it completely, it is sufficient to describe the computation of the supergroup $\E_{0}/\F^\sharp(L)$.
By \Cref{detUsharp} we know that $\F^\sharp(L) = \E^i$ for some $i \in \ZZ_{\geq 0}$.
It is thus sufficient to describe the computation of $\E_{0}/\E^i$.
By definition this group is isomorphic to $\ker(\overline \Nr_i)$, where
\[ \overline \Nr_i \colon \O^\times{/}(1 + \FP^i) \longrightarrow \o^\times{/}\Nr(1 + \FP^i), \bar u \longmapsto \overline{\Nr(u)}. \]
Depending on the structure of the extension $E/K$, this kernel can be described as follows:
\begin{itemize}
  \item
    If $E \cong K \times K$, then $\E_0/\E^i$ is isomorphic to $(\o/\fp^i)^\times$. 
  \item
  If $E/K$ is an unramified extension of local fields, then $\E_0/\E^i$ is isomorphic to 
  the kernel of the map $(\O/\FP^i)^\times \rightarrow (\o/\fp^i)^\times, \, \bar u \mapsto \overline{\Nr(u)}$. 
\item
  If $E/K$ is a ramified extension of local fields, then the situation is more complicated due to the norm not being surjective.
  Using the fact that by definition we have 
\[ \E^i/\E^{i+1} \cong \ker \left((1 + \FP^i)/(1 + \FP^{i + 1}) \longrightarrow \Nr(1 + \FP^i)/\Nr(1 + \FP^{i+1}),\,  \overline{u} \longmapsto \overline{\Nr(u)}\right), \] 
  this quotient can be determined using explicit results on the image of the multiplicative groups $1 + \FP^i$ under the norm map, found for example in
    \cite[Chap. V]{Serre1979}.
    Applying this iteratively we obtain $\E_0/\E^i$.
\end{itemize}
    In all three cases the computations of the quotient groups $\E_0/\E^i$ reduce to determining unit groups of residue class rings or kernels of morphisms between such groups. These unit groups are finitely generated abelian groups, whose structure can be determined using classical algorithms from algebraic number theory, see for example \cite[Sec. 4.2]{Cohen2000}.

\section{Fixed Points}\label{sec:fixed}
We classify the fixed point sets of purely non-symplectic automorphisms of finite order $n$ on complex K3 surfaces.
We only use the description of the fixed locus for $n=p$ a prime (see e.g. \cite{artebani-sarti-taki}); hence providing an independent proof in the known cases and completing the classification in all other cases.

Given the action of $\sigma$ on some lattice $L \cong H^2(X,\ZZ)$, we want to derive the invariants $((a_1,\dots a_s),k,l,g)$
of the fixed locus $X^\sigma$ as defined in the introduction.
The topological and holomorphic Lefschetz' fixed point formula \cite[Thm. 4.6]{atiyah-singer1968} yield the following relations
\[\sum_{i=1}^s a_i -2k + l(2 - 2g)
= 2 + \Tr \sigma^*|H^2(X,\CC)\]
and
\[1+\zeta_n^{-1} = \sum_{i=1}^s \frac{a_i}{(1-\zeta_n^{i+1})(1-\zeta_n^{-i})}+l(1-g) \frac{1+\zeta_n}{(1-\zeta_n)^2}.\]
We adopt the following strategy:
By induction, we know the invariants of the fixed loci of $\sigma^p$ for $p \mid n$.
Note that $\sigma$ acts with order dividing $p$ on $X^{\sigma^p}$ and $X^\sigma \subseteq X^{\sigma^p}$.
From the fixed loci of $\sigma^p$, we derive
the obvious upper bounds on $k$, $l$ and $g$.
Then for each possible tuple $(k,l,g)$, we find all $(a_1,\dots a_s)$ satisfying the Lefschetz formulas, which amounts to enumerating integer points in a bounded polygon.
The result is a finite list of possibilities for the invariants of $X^\sigma$.

In what follows, we derive compatibility conditions coming from the description of $X^\sigma$ as the fixed point set of the action of $\sigma$ on $X^{\sigma^p}$.
\begin{lemma}\label{fixed points on curves}
 Let $P$ be an isolated fixed point of type $i$ of $\sigma$. Let $p m = n$ with $p$ prime. Then $P$ is a fixed point of $\sigma^p$ of type
 \[t(i) = \min\{i + 1 \bmod m,n -i \bmod m\} - 1.\]
 Moreover, we have
 \[a_i(\sigma) \leq a_{t(i)}(\sigma^p) \quad \mbox{ and } \quad \sum_{\{i | t(i)=j\}}a_i(\sigma) \equiv a_{j}(\sigma^p) \mod p\]
  where $1 \leq j \leq (m-1)/2$.
\end{lemma}
\begin{proof}
  In local coordinates $\sigma^p(x,y) = (\zeta_m^{i+1},\zeta_m^{-i})$. Note that $\sigma$ acts on the set of fixed points of type $j$ of $\sigma^p$. Hence the number of fixed points $ \sum_{\{i|t(i)=j\}}a_i(\sigma)$ of $\sigma$ is congruent to the order $a_j(\sigma^p)$ of this set modulo $p$.
\end{proof}

In particular, from the invariants of $\sigma$, we can infer how many isolated fixed points of $\sigma$ lie on a fixed curve of $X^{\sigma^p}$. More precisely, by \Cref{fixed points on curves} $\sum_{\{i|t(i)=0\}}a_i(\sigma)$ is the number of isolated fixed points of $\sigma$ which lie on a fixed curve of $\sigma^p$.
The number of such points is bounded above and below by the following lemma.
\begin{lemma}
 Let $p$ be a prime number, $C$ a smooth curve of genus $g$ and $\sigma \in \Aut(C)$ an automorphism of order $p$ with $C/\sigma$ of genus $g'$. Then $\sigma$ fixes
 \[r = \frac{2g -2 - p(2g'-2)}{p-1}\]
 points. In particular, given $p$ and $g$ there is a finite number of possibilities for $r$.
 Note that for $g = 0$ we have $r = 2$.
\end{lemma}
\begin{proof}
 The canonical map $\pi\colon C \to C/\sigma$ is ramified precisely in the fixed points and with multiplicity $p$.
 By the Hurwitz formula $2g -2 = p (2 g' - 2) + (p-1)r$.
\end{proof}

Inductively carrying out this strategy, we obtain a unique possibility in most cases and in the remaining $33$ cases two possibilities. We call
the corresponding automorphisms ambiguous. In what follows we disambiguate by using elliptic fibrations.
We say that a curve $C$ is fixed by $\sigma$ if $\sigma|_C=\id_C$. If merely $\sigma(C)=C$ we say that it is invariant.
\begin{lemma}\label{lemma:sect}
 Let $p$ be a prime divisor of $n$. Suppose that $\sigma^{n/p}$ fixes an elliptic curve $E$. Denote by $\pi\colon X \to \PP^1$ the elliptic fibration induced by the linear system $|E|$.
 Suppose $\sigma$ has no isolated fixed points on $E$ and $\sigma^m$ leaves invariant a section of $\pi$ for some $m \mid n, m \neq n$.
 Then $\sigma$ fixes $E$ if and only if $\pi$ admits a $\sigma$-invariant section.
\end{lemma}
\begin{proof}
 Since we assume that $\sigma$ has no isolated fixed points on $E$,
 either $\sigma$ fixes $E$ entirely or no point on $E$ at all.

 Suppose that $S$ is a $\sigma$-invariant section. Then $E \cap S$ is a fixed point. Therefore $\sigma$ fixes $E$.

Conversely, suppose that there is no $\sigma$-invariant section. By assumption we find a $\sigma^m$-invariant section $S$ which must satisfy $\sigma(S) \neq S$.
If $\sigma$ acts trivially on $E$, then $\{P\}=E \cap S \cap \sigma(S)$ is a fixed point. The three $\sigma^m$-invariant curves $E$, $S$ and $\sigma(S)$ pass through $P$. This contradicts the local description of the action around $P$.  Therefore $\sigma$ must act as a translation on $E$.
\end{proof}

\begin{lemma}\label{identifyE}
 Let $\tau$ be an automorphism of prime order $p$ of a K3 surface $X$ acting trivially on $\NS(X)$ and fixing an elliptic curve $E$. If $f \in \NS(X)$ is isotropic, primitive and nef such that $f^\perp /\ZZ f$ is not an overlattice of a root lattice, then $f = [E]$.
\end{lemma}
\begin{proof}
 Since $\tau$ acts trivially on $\NS(X)$, it lies in the center of the automorphism group $\Aut(X)$ and fixes $E$. Hence every automorphism leaves $E$ invariant, i.e., $E$ is a curve canonically defined on $X$.
 Let $\pi\colon X \to \PP^1$ be the genus one fibration defined by $|E|$.
The fibration $\pi$ is canonically defined, therefore $\Aut(X)$ is virtually abelian of rank $t$ given by the rank of the Mordell--Weil group of (the Jacobian of) $\pi$.
Let $R$ be the root sublattice of $[E]^\perp/\ZZ[E]$. Then by the Shioda--Tate formula \cite[5.2]{shioda1990} $t=\rk \NS(X)-\rk R -2$. Since $\Aut(X)$ is virtually abelian, there is at most one elliptic fibration of positive rank. By \cite[\S 3]{shafarevich1972}, $f$ is the class of a fiber of an elliptic fibration. Since $f^\perp/\ZZ f$ is not an overlattice of a root lattice, the Shioda--Tate formula implies that it has positive Mordell--Weil rank. Hence it must coincide with $\pi$ and so $[E] = f$.
\end{proof}

\begin{lemma}\label{hassection}
 Let $p$ be a prime divisor of $n$. Suppose that $\sigma^p$ fixes an elliptic curve $E$ inducing the elliptic fibration $\pi$. Set $\tau=\sigma^{n/p}$ and $N= \NS(X)^{\tau}$.
 Let $f\in \NS(X)^\sigma$ be isotropic and primitive such that $f^{\perp N}/\ZZ f$ is not an overlattice of a root lattice.
 Then $\pi$ has a $\tau$-invariant section if and only if $\langle f, N \rangle = \ZZ$.
 Similarly, $\pi$ has a $\sigma$-invariant section if and only if $\langle f ,\NS(X)^\sigma \rangle=\ZZ$.
\end{lemma}
\begin{proof}
By \cite[Lemma 1.7 and Theorem 1.8]{oguiso-sakurai} there exists an element $\delta$ of the  Weyl group $W(\NS(X))$ commuting with $\sigma$ such that $\delta(f)$ is nef. Hence we may assume  that $f$ is nef.
In order to apply \Cref{identifyE} set $\tau = \sigma^{n/p}$ and consider $G=\langle \tau \rangle$. Choose a marking $\eta\colon H^2(X,\ZZ)\to L$ and set $H= \eta\rho_X(G) \eta^{-1}$. Then we can deform the $H$-marked K3 surface $(X,G,\eta)$ to $(X',G',\eta')$ such that
$\NS(X')^{G'} = \NS(X')$. Let $E'$ be the elliptic curve fixed by $G'$. It satisfies $\eta'([E']) = \eta([E])$. Note that $f' = \eta'^{-1} \circ \eta(f)$ is still nef since $\eta'(\NS(X'))\subseteq \eta(\NS(X))$.
Hence by \Cref{identifyE} we get $f'= [E']$ which gives $f = [E]$.

Finally, if $\langle [E], \NS(X)\rangle = \ZZ$, we can find $s \in \NS(X)$ with $\langle [E],s\rangle=1$ and $s^2=-2$. After possibly replacing $s$ by $-s$ we may assume that $s$ is effective.
We can write $s = \delta_1 + \dots + \delta_n$ for $(-2)$-curves $\delta_i \in \NS(X)$. Now, $1=\langle f , s \rangle$ and $\langle f, \delta_i \rangle \geq 0$ ($f$ is nef) imply that $\langle f , \delta_i \rangle =1$ for a single $1\leq i \leq n $; $\delta_i$ is the desired section. Note that $\tau$ (respectively $\sigma$) preserves $\delta_i$ if and only if it preserves $s$.
\end{proof}

Among the $22$ ambiguous automorphisms $\sigma$ of order $n=4$, there are $16$ cases where $\sigma^2$ fixes a single elliptic curve,
$4$ cases where $\sigma^2$ fixes a curve of genus $2$ and $3$, $5$, $7$ or $9$ rational curves, and in the remaining $2$ cases $\sigma^2$ fixes two curves of genus one.
The ambiguity is whether $\sigma$ fixes some curve or not.

First let $\sigma^2$ fix a unique elliptic curve $E$ of genus $1$. This means that $\sigma$ is compatible with an elliptic fibration $\pi\colon X \to \PP^1$.
Moreover, $\sigma^2$ leaves invariant a section $S$ of $\pi$ because $L^{\sigma^2}$ contains a copy of the hyperbolic plane $U$. Note that $\sigma^2$ must act non-trivially on the base $\PP^1$ of the fibration, as otherwise its action at the tangent space to the point in $E\cap S$ would be trivial.
Hence $\sigma^2$ has exactly two fixed points in $\PP^1$ giving two invariant fibers; one is $E$ and the other one we call $C$.
The rational curves fixed by $\sigma^2$ must be components of the fiber $C$, because all fixed points lie in $E \cup C$.
We fix the fiber type of $C$ and consider each fiber type separately.
We know that $C$ is a singular fiber of Kodaira type $I_{4m}$, $m=1,2,3,4$ or $IV^*$. They correspond to $3,3,3,4$ and $3$ ambiguous cases.\\[4pt]
The following figure shows the dual graph of the irreducible components of $C$. Each node corresponds to a smooth rational curve and two nodes are joined by an edge if and only if the corresponding curves intersect. The square nodes are curves fixed pointwise by $\sigma^2$ and the round nodes are curves which are invariant but not fixed by $\sigma^2$.
The automorphism $\sigma$ acts on the graph with order dividing $2$ and maps squares to squares.

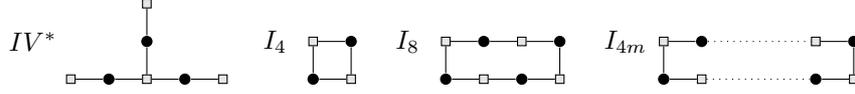
\begin{center}
 \begin{tikzpicture}[scale=0.5]
\node (label) at (1,1) {$IV^*$};
  \begin{scope}[every node/.style={draw, fill=black!10,                      inner sep=1.6pt}]
  \node (2) at (2,0) {};
  \node (4) at (4,0) {};
  \node (4b) at (4,2) {};
  \node (6) at (6,0) {};
 \end{scope}
 \begin{scope}[circle,every node/.style={fill, inner sep=1.5pt}]
  \node (4a) at (4,1) {};
  \node (3) at (3,0) {};
  \node (5) at (5,0) {};
 \end{scope}
  \path [-] (4) edge node {} (4a);
  \path [-] (4a) edge node {} (4b);
  \path [-] (2) edge node {} (3);
  \path [-] (3) edge node {} (4);
  \path [-] (4) edge node {} (5);
  \path [-] (5) edge node {} (6);
 \end{tikzpicture}\quad
 \begin{tikzpicture}[scale=0.5]
\node (label) at (-1,1) {$I_{4}$};
  \begin{scope}[every node/.style={draw, fill=black!10,                      inner sep=1.6pt}]
  \node (01) at (0,1) {};
  \node (10) at (1,0) {};
 \end{scope}
 \begin{scope}[circle,every node/.style={fill, inner sep=1.5pt}]
  \node (00) at (0,0) {};
  \node (11) at (1,1) {};
 \end{scope}
  \path [-] (00) edge node {} (01);
  \path [-] (01) edge node {} (11);
  \path [-] (10) edge node {} (11);
  \path [-] (00) edge node {} (10);
 \end{tikzpicture}\quad
 \begin{tikzpicture}[scale=0.5]
\node (label) at (-1,1) {$I_{8}$};
  \begin{scope}[every node/.style={draw, fill=black!10,                      inner sep=1.6pt}]
  \node (01) at (0,1) {};
  \node (10) at (1,0) {};
  \node (21) at (2,1) {};
  \node (30) at (3,0) {};
 \end{scope}
 \begin{scope}[circle,every node/.style={fill, inner sep=1.5pt}]
  \node (00) at (0,0) {};
  \node (11) at (1,1) {};
  \node (20) at (2,0) {};
  \node (31) at (3,1) {};
 \end{scope}
  \path [-] (00) edge node {} (01);
  \path [-] (00) edge node {} (10);
  \path [-] (01) edge node {} (11);
  \path [-] (11) edge node {} (21);
  \path [-] (21) edge node {} (31);
  \path [-] (10) edge node {} (20);
  \path [-] (20) edge node {} (30);
  \path [-] (30) edge node {} (31);
 \end{tikzpicture}\quad
  \begin{tikzpicture}[scale=0.5]
\node (label) at (-1,1) {$I_{4m}$};
  \begin{scope}[every node/.style={draw, fill=black!10,                      inner sep=1.6pt}]
  \node (01) at (0,1) {};
  \node (10) at (1,0) {};
  \node (41) at (4,1) {};
  \node (50) at (5,0) {};
 \end{scope}
 \begin{scope}[circle,every node/.style={fill, inner sep=1.5pt}]
  \node (00) at (0,0) {};
  \node (11) at (1,1) {};
  \node (40) at (4,0) {};
  \node (51) at (5,1) {};
 \end{scope}
  \path [-] (00) edge node {} (01);
  \path [-] (00) edge node {} (10);
  \path [-] (01) edge node {} (11);
  \path [dotted] (11) edge node {} (41);
  \path [-] (41) edge node {} (51);
  \path [dotted] (10) edge node {} (40);
  \path [-] (40) edge node {} (50);
  \path [-] (50) edge node {} (51);
 \end{tikzpicture}\\[4pt]
\end{center}

We start by determining the fixed locus in the case that $C$ is of type $I_{4m}$.
Lefschetz calculations show that for the given $\sigma^2$ we have the following possibilities for the fixed locus $X^\sigma$ of $\sigma$:
four isolated points $((4),0,0,1)$ or four isolated points and a curve of genus one $((4),0,1,1)$. %

The automorphism $\sigma^2$ fixes every second node of the circular $I_{4m}$ configuration and the zero section intersects a $\sigma^2$-fixed curve in $I_{4m}$.

The action of $\sigma$ on the intersection graph is visible on the lattice side, since we can identify the class $F$ of $E$ by \Cref{identifyE} and choose simple roots in $F^\perp$ giving its components.
Carrying out this computation gives the following $3$ cases.
\begin{enumerate}
\item The curves in $I_{4m}$ are rotated by $\sigma$. Then $C$ does not have any $\sigma$ fixed points and $E$ contains $4$ isolated fixed points for 0.4.4.9, 0.4.3.6, 0.4.2.5, 0.4.1.3.

\item The automorphism $\sigma$ acts as a reflection on the graph $I_{4m}$ leaving invariant two of the $\sigma^2$-fixed curves and the section $S$ passes through one of them. Then $I_{4m}$ contains $4$ isolated fixed points and $E$ does not contain any isolated fixed points. However the intersection $S \cap E$ is fixed but cannot be isolated. Hence $\sigma$ fixes $E$ for 0.4.4.7, 0.4.3.8, 0.4.2.3, 0.4.1.6.

\item The automorphism $\sigma$ acts as a reflection leaving invariant two of the $\sigma^2$-fixed curves and the section $S$ does not pass through them. Then $\sigma$ fixes 4 isolated points on $I_{4m}$ and no isolated points on $E$. By \Cref{lemma:sect} $\sigma$ does not fix $E$ for 0.4.4.8, 0.4.3.7, 0.4.2.4, 0.4.1.4, 0.4.1.5.
\end{enumerate}

Let $C = IV^*$ and $\sigma$ ambiguous. We know that $\sigma$ fixes $6$ isolated points, one rational curve and possibly $E$ of genus $1$.
The central curve as well as the three leaves must be fixed by $\sigma^2$. There are $3$ possible actions.
It can leave invariant each component of the $IV^*$ fiber. Then the central component is fixed and the leaves carry two fixed points each. Hence the action on $E$ does not have an isolated fixed point. Therefore $\sigma$ fixes $E$ if and only if some section is preserved. This is the case for 0.4.3.11 but not for 0.4.3.10.
In the third case $\sigma$ swaps two of the branches. Therefore the central component cannot be fixed by $\sigma$, so it contains $2$ isolated fixed points.
The invariant leaf must be the fixed rational curve. Then there are $4$ fixed points left, they must lie on $E$ giving 0.4.3.9.
This settles the first $16$ ambiguous cases.

In the next $4$ ambiguous cases $\sigma^2$ fixes a curve of genus $2$ and $3$, $5$, $7$ or $9$ rational curves.
In each case we know that $\sigma$ fixes exactly $4$ isolated fixed points. The ambiguity is whether $\sigma$ fixes $1$ rational curve and the genus $2$ curve, or no curve at all.
For each case we exhibit a $\sigma$-invariant hyperbolic plane $U$. Since $\sigma$ fixes a curve of genus $2$, $\Aut(X)$ is finite, and so $K = U^\perp$ is a root lattice. Then $\NS(X)=U\perp K$ and $K$ determines the ADE-types of the singular fibers of the $\sigma$-equivariant fibration induced by $U$. The square nodes are fixed by $\sigma^2$ while the round nodes are not.
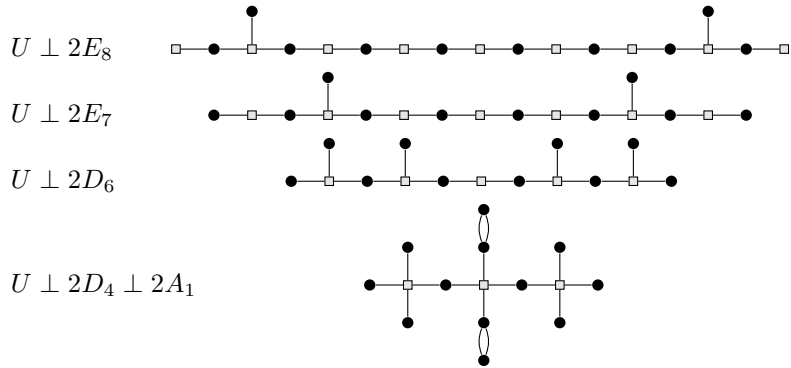
\begin{center}
\begin{minipage}{10cm}
 \begin{tikzpicture}[scale=0.5]
  \node (label) at (-3,0) {$U \perp 2E_8$};
 \begin{scope}[every node/.style={draw, fill=black!10,                      inner sep=1.6pt}]
  \node (0) at (0,0) {};
  \node (2) at (2,0) {};
  \node (4) at (4,0) {};
  \node (6) at (6,0) {};
  \node (8) at (8,0) {};
  \node (10) at (10,0) {};
  \node (12) at (12,0) {};
  \node (14) at (14,0) {};
  \node (16) at (16,0) {};
 \end{scope}
 \begin{scope}[circle,every node/.style={fill, inner sep=1.5pt}]
  \node (1) at (1,0) {};
  \node (3) at (3,0) {};
  \node (2a) at (2,1) {};
  \node (5) at (5,0) {};
  \node (7) at (7,0) {};
  \node (9) at (9,0) {};
  \node (11) at (11,0) {};
  \node (13) at (13,0) {};
  \node (15) at (15,0) {};
  \node (14a) at (14,1) {};
 \end{scope}
  \path [-] (14) edge node {} (14a);
  \path [-] (2) edge node {} (2a);
  \path [-] (0) edge node {} (1);
  \path [-] (1) edge node {} (2);
  \path [-] (2) edge node {} (3);
  \path [-] (3) edge node {} (4);
  \path [-] (4) edge node {} (5);
  \path [-] (5) edge node {} (6);
  \path [-] (6) edge node {} (7);
  \path [-] (7) edge node {} (8);
  \path [-] (8) edge node {} (9);
  \path [-] (9) edge node {} (10);
  \path [-] (10) edge node {} (11);
  \path [-] (11) edge node {} (12);
  \path [-] (12) edge node {} (13);
  \path [-] (13) edge node {} (14);
  \path [-] (14) edge node {} (15);
  \path [-] (15) edge node {} (16);
 \end{tikzpicture}

 \begin{tikzpicture}[scale=0.5]
 \node (label) at (-3,0) {$U \perp 2E_7$};
 \begin{scope}[every node/.style={draw, fill=black!10,                      inner sep=1.6pt}]
  \node (2) at (2,0) {};
  \node (4) at (4,0) {};
  \node (6) at (6,0) {};
  \node (8) at (8,0) {};
  \node (10) at (10,0) {};
  \node (12) at (12,0) {};
  \node (14) at (14,0) {};
 \end{scope}
 \begin{scope}[circle,every node/.style={fill, inner sep=1.5pt}]
  \node (1) at (1,0) {};
  \node (3) at (3,0) {};
  \node (4a) at (4,1) {};
  \node (5) at (5,0) {};
  \node (7) at (7,0) {};
  \node (9) at (9,0) {};
  \node (11) at (11,0) {};
  \node (13) at (13,0) {};
  \node (15) at (15,0) {};
  \node (12a) at (12,1) {};
 \end{scope}
  \path [-] (12) edge node {} (12a);
  \path [-] (4) edge node {} (4a);
  \path [-] (1) edge node {} (2);
  \path [-] (2) edge node {} (3);
  \path [-] (3) edge node {} (4);
  \path [-] (4) edge node {} (5);
  \path [-] (5) edge node {} (6);
  \path [-] (6) edge node {} (7);
  \path [-] (7) edge node {} (8);
  \path [-] (8) edge node {} (9);
  \path [-] (9) edge node {} (10);
  \path [-] (10) edge node {} (11);
  \path [-] (11) edge node {} (12);
  \path [-] (12) edge node {} (13);
  \path [-] (13) edge node {} (14);
  \path [-] (14) edge node {} (15);
 \end{tikzpicture}

 \begin{tikzpicture}[scale=0.5]
\node (label) at (-5,0) {$U \perp 2D_6$};
  \begin{scope}[every node/.style={draw, fill=black!10,                      inner sep=1.6pt}]
  \node (2) at (2,0) {};
  \node (4) at (4,0) {};
  \node (6) at (6,0) {};
  \node (8) at (8,0) {};
  \node (10) at (10,0) {};
 \end{scope}
 \begin{scope}[circle,every node/.style={fill, inner sep=1.5pt}]
  \node (2a) at (2,1) {};
  \node (10a) at (10,1) {};
  \node (4a) at (4,1) {};
  \node (8a) at (8,1) {};
  \node (1) at (1,0) {};
  \node (3) at (3,0) {};
  \node (5) at (5,0) {};
  \node (7) at (7,0) {};
  \node (9) at (9,0) {};
  \node (11) at (11,0) {};
 \end{scope}
  \path [-] (2) edge node {} (2a);
  \path [-] (4) edge node {} (4a);
  \path [-] (8) edge node {} (8a);
  \path [-] (10) edge node {} (10a);
  \path [-] (1) edge node {} (2);
  \path [-] (2) edge node {} (3);
  \path [-] (3) edge node {} (4);
  \path [-] (4) edge node {} (5);
  \path [-] (5) edge node {} (6);
  \path [-] (6) edge node {} (7);
  \path [-] (7) edge node {} (8);
  \path [-] (8) edge node {} (9);
  \path [-] (9) edge node {} (10);
  \path [-] (10) edge node {} (11);
 \end{tikzpicture}

 \begin{tikzpicture}[scale=0.5]
\node (label) at (-6,0) {$U \perp 2D_4\perp 2 A_1$};
 \begin{scope}[every node/.style={draw, fill=black!10,                      inner sep=1.6pt}]
  \node (2) at (2,0) {};
  \node (4) at (4,0) {};
  \node (6) at (6,0) {};
 \end{scope}
 \begin{scope}[circle,every node/.style={fill, inner sep=1.5pt}]
  \node (2a) at (2,1) {};
  \node (2b) at (2,-1) {};
  \node (6a) at (6,1) {};
  \node (6b) at (6,-1) {};
  \node (1) at (1,0) {};
  \node (3) at (3,0) {};
  \node (5) at (5,0) {};
  \node (7) at (7,0) {};
  \node (4a) at (4,1) {};
  \node (4b) at (4,2) {};
  \node (4c) at (4,-1) {};
  \node (4d) at (4,-2) {};
 \end{scope}
  \path [-] (2) edge node {} (2a);
  \path [-] (2) edge node {} (2b);
  \path [-] (6) edge node {} (6a);
  \path [-] (6) edge node {} (6b);
  \path [-] (4) edge node {} (4a);
  \path [-] (4) edge node {} (4c);
  \path [-] (4a) edge [bend left=20] node{} (4b);
  \path [-] (4a) edge [bend right=20] node{} (4b);
  \path [-] (4c) edge [bend left=20] node{} (4d);
  \path [-] (4c) edge [bend right=20] node{} (4d);
  \path [-] (1) edge node {} (2);
  \path [-] (2) edge node {} (3);
  \path [-] (3) edge node {} (4);
  \path [-] (4) edge node {} (5);
  \path [-] (5) edge node {} (6);
  \path [-] (6) edge node {} (7);
 \end{tikzpicture}
\end{minipage}
 \end{center}
 Note that $\sigma$ must act non-trivially on the graph, because otherwise it has too many fixed points or fixed curves.
 Since $\sigma$ maps squares to squares, we see that $\sigma$ must act as a reflection preserving the central square.
 However, the corresponding curve cannot be fixed, because the two adjacent ones and hence the corresponding intersection points with the central node are swapped.
 Thus $\sigma$ cannot fix a curve and the four ambiguous cases are settled.

Consider the ambiguous cases
0.4.5.12,  0.4.5.14 0.6.2.29, 0.6.3.36, 0.8.1.7, 0.8.1.8, 0.8.2.8, 0.8.2.10, 0.9.1.3, 0.9.1.4 and 0.10.1.11.
The question is whether or not $\sigma$ fixes an elliptic curve.
In each case Lefschetz calculations show that there are no isolated fixed points on the elliptic curve in question. Hence in view of
\Cref{lemma:sect}
this can be decided by whether or not there exists a section of the corresponding fibration. This is settled by \Cref{hassection}.
Indeed, we can randomly search until we find $f \in L^\sigma$ corresponding to $[E]$.

For  0.10.2.1 we have the two possibilities $((0, 0, 1, 6), 0, 0, 1)$ and  $((5, 0, 0, 0), 1, 0, 1)$ for the fixed locus of $\sigma$. We know that $\sigma^2$ fixes an elliptic curve $E$ and one rational curve. It is the central component of the invariant fiber $C$ of type $\tilde{E_7}$. Since $\NS(X)^\sigma$ has rank $6$, $\sigma$ must act non-trivially on $\tilde{E_7} \subseteq \NS(X)$.
This means that it swaps two of the three arms of the configuration. Hence it cannot act trivially on the central component and so $\sigma$ cannot fix a rational curve.
The fixed locus is $((0, 0, 1, 6), 0, 0, 1)$.

For the last ambiguous case 0.12.1.12 the automorphism $\sigma$ has a unique fixed point and the ambiguity is whether or not it fixes an elliptic curve. We know that $\sigma^i$ for $i=2,3,4,6$ fixes a unique elliptic curve $E_i$.
Since $E_2\subseteq E_4,E_6$, we have $E_2=E_4=E_6$ and similarly $E_6 \subseteq E_3$ implies $E_6=E_3$. Therefore $E=E_i$ is independent of $i$.
Now $\sigma^3$ and $\sigma^4$ fix $E$ hence their product $\sigma^7$ fixes $E$ as well. But $\sigma \in \langle \sigma^7\rangle$ and so $\sigma$ fixes the elliptic curve $E$. The fixed locus of 0.12.1.12 is therefore $((1, 0, 0, 0, 0), 0, 1, 1)$.
 
\bibliographystyle{amsplain}
\bibliography{../unitary}
\newpage
\appendix

\section{Finite groups with mixed action on a K3 surface}
\label{appendixA}
The following table lists all finite groups $G$ admitting a faithful, saturated, mixed action on some K3 surface, their symplectic subgroups $1\neq G_s < G$ as well as the number $k(G)$ of deformation types.
Note that in 3 cases an entry appears twice because
the normal subgroup $G_s < G$ does not lie in the same $\Aut(G)$-orbit.
The notation $G = G_s.\mu_n$ means that $G$ is an extension of $G_s$ by $\mu_n$. It may or may not split. Our notation for the groups $G_s$ follows Hashimoto \cite{Hashimoto2012}. Isomorphism classes of groups will be referred to
either using standard notation for classical families or
using the id as provided by the library of small groups \cite{Besche2002}.

\renewcommand{\arraystretch}{1.2}
{\small
\begin{longtable}{ccc|ccc|ccc}
  \caption{Finite groups with faithful, saturated non-symplectic action on some K3 surface}\label{table:action}\\
  $G$&$\mathrm{id} $& $k(G)$ & $G$&$\mathrm{id} $& $k(G)$ & $G$ & id & $k(G)$ \\
  \hline
  \endhead
\rowcolor{lightgray}
$C_2.\mu_{2}$&(4, 1)&5 & $C_3.\mu_{15}$&(45, 2)&1   &  $C_4.\mu_{12}$&(48, 22)&1\\
$C_2.\mu_{2}$&(4, 2)&354 & $C_3.\mu_{18}$&(54, 4)&1   &  $C_4.\mu_{12}$&(48, 23)&1\\
\rowcolor{lightgray}
$C_2.\mu_{3}$&(6, 2)&26 & $C_3.\mu_{18}$&(54, 9)&1   &  $C_4.\mu_{12}$&(48, 24)&1\\
$C_2.\mu_{4}$&(8, 1)&3 & $C_2^2.\mu_{2}$&(8, 2)&4   &  $D_6.\mu_{2}$&(12, 4)&140\\
\rowcolor{lightgray}
$C_2.\mu_{4}$&(8, 2)&200 & $C_2^2.\mu_{2}$&(8, 3)&40   &  $D_6.\mu_{3}$&(18, 3)&21\\
$C_2.\mu_{5}$&(10, 2)&6 & $C_2^2.\mu_{2}$&(8, 5)&330  &  $D_6.\mu_{4}$&(24, 5)&18\\
\rowcolor{lightgray}
$C_2.\mu_{6}$&(12, 2)&11 & $C_2^2.\mu_{3}$&(12, 3)&4  &  $D_6.\mu_{5}$&(30, 1)&1\\
$C_2.\mu_{6}$&(12, 5)&99 & $C_2^2.\mu_{3}$&(12, 5)&11  &  $D_6.\mu_{6}$&(36, 12)&33\\
\rowcolor{lightgray}
$C_2.\mu_{7}$&(14, 2)&4 & $C_2^2.\mu_{4}$&(16, 3)&73 & $D_6.\mu_{8}$&(48, 4)&4\\
$C_2.\mu_{8}$&(16, 5)&50 & $C_2^2.\mu_{4}$&(16, 5)&1 & $D_6.\mu_{10}$&(60, 11)&1\\
\rowcolor{lightgray}
$C_2.\mu_{9}$&(18, 2)&2 & $C_2^2.\mu_{4}$&(16, 6)&2 & $D_6.\mu_{12}$&(72, 27)&2\\
$C_2.\mu_{10}$&(20, 2)&1 & $C_2^2.\mu_{4}$&(16, 10)&77 & $C_2^3.\mu_{2}$&(16, 3)&3\\
\rowcolor{lightgray}
$C_2.\mu_{10}$&(20, 5)&12 & $C_2^2.\mu_{6}$&(24, 10)&13 & $C_2^3.\mu_{2}$&(16, 10)&2\\
$C_2.\mu_{12}$&(24, 2)&2 & $C_2^2.\mu_{6}$&(24, 13)&19 & $C_2^3.\mu_{2}$&(16, 11)&34\\
\rowcolor{lightgray}
$C_2.\mu_{12}$&(24, 9)&24 & $C_2^2.\mu_{6}$&(24, 15)&14 & $C_2^3.\mu_{2}$&(16, 14)&72\\
$C_2.\mu_{14}$&(28, 4)&4 & $C_2^2.\mu_{8}$&(32, 5)&14 & $C_2^3.\mu_{3}$&(24, 13)&2\\
\rowcolor{lightgray}
$C_2.\mu_{15}$&(30, 4)&3 & $C_2^2.\mu_{8}$&(32, 36)&6 & $C_2^3.\mu_{4}$&(32, 6)&6\\
$C_2.\mu_{16}$&(32, 16)&3 & $C_2^2.\mu_{9}$&(36, 3)&2 & $C_2^3.\mu_{4}$&(32, 7)&3\\
\rowcolor{lightgray}
$C_2.\mu_{18}$&(36, 5)&3   &  $C_2^2.\mu_{12}$&(48, 21)&4 & $C_2^3.\mu_{4}$&(32, 22)&31\\
$C_2.\mu_{20}$&(40, 9)&3   &  $C_2^2.\mu_{12}$&(48, 31)&3 & $C_2^3.\mu_{4}$&(32, 45)&7\\
\rowcolor{lightgray}
$C_2.\mu_{24}$&(48, 23)&3  &  $C_2^2.\mu_{12}$&(48, 44)&1 & $C_2^3.\mu_{6}$&(48, 31)&1\\
$C_2.\mu_{28}$&(56, 8)&1   &  $C_2^2.\mu_{18}$&(72, 16)&2 & $C_2^3.\mu_{6}$&(48, 49)&5\\
\rowcolor{lightgray}
$C_2.\mu_{30}$&(60, 13)&2  &  $C_4.\mu_{2}$&(8, 1)&3 & $C_2^3.\mu_{7}$&(56, 11)&1\\
$C_3.\mu_{2}$&(6, 1)&46    &  $C_4.\mu_{2}$&(8, 2)&98 & $C_2^3.\mu_{7}$&(56, 11)&1\\
\rowcolor{lightgray}
$C_3.\mu_{2}$&(6, 2)&51    &  $C_4.\mu_{2}$&(8, 3)&102 & $C_2^3.\mu_{8}$&(64, 4)&3\\
$C_3.\mu_{3}$&(9, 2)&26    &  $C_4.\mu_{2}$&(8, 4)&3 & $C_2^3.\mu_{8}$&(64, 87)&1\\
\rowcolor{lightgray}
$C_3.\mu_{4}$&(12, 1)&15   &  $C_4.\mu_{3}$&(12, 2)&7 & $C_2^3.\mu_{12}$&(96, 196)&2\\
$C_3.\mu_{4}$&(12, 2)&15   &  $C_4.\mu_{4}$&(16, 2)&21 & $D_8.\mu_{2}$&(16, 7)&6\\
\rowcolor{lightgray}
$C_3.\mu_{5}$&(15, 1)&4    &  $C_4.\mu_{4}$&(16, 4)&23 & $D_8.\mu_{2}$&(16, 8)&2\\
$C_3.\mu_{6}$&(18, 3)&38   &  $C_4.\mu_{4}$&(16, 5)&10 & $D_8.\mu_{2}$&(16, 11)&202\\
\rowcolor{lightgray}
$C_3.\mu_{6}$&(18, 5)&36   &  $C_4.\mu_{4}$&(16, 6)&9 & $D_8.\mu_{2}$&(16, 13)&11\\
$C_3.\mu_{7}$&(21, 2)&1    &  $C_4.\mu_{5}$&(20, 2)&1 & $D_8.\mu_{3}$&(24, 10)&3\\
\rowcolor{lightgray}
$C_3.\mu_{8}$&(24, 1)&3    &  $C_4.\mu_{6}$&(24, 2)&1 & $D_8.\mu_{4}$&(32, 9)&10\\
$C_3.\mu_{8}$&(24, 2)&3    &  $C_4.\mu_{6}$&(24, 9)&5 & $D_8.\mu_{4}$&(32, 11)&4\\
\rowcolor{lightgray}
$C_3.\mu_{9}$&(27, 2)&3    &  $C_4.\mu_{6}$&(24, 10)&8 & $D_8.\mu_{4}$&(32, 25)&19\\
$C_3.\mu_{10}$&(30, 1)&4   &  $C_4.\mu_{6}$&(24, 11)&1 & $D_8.\mu_{4}$&(32, 38)&1\\
\rowcolor{lightgray}
$C_3.\mu_{10}$&(30, 4)&2   &  $C_4.\mu_{8}$&(32, 3)&4 & $D_8.\mu_{6}$&(48, 26)&1\\
$C_3.\mu_{12}$&(36, 6)&5   &  $C_4.\mu_{8}$&(32, 12)&4 & $D_8.\mu_{6}$&(48, 45)&4\\
\rowcolor{lightgray}
$C_3.\mu_{12}$&(36, 8)&4   &  $C_4.\mu_{10}$&(40, 9)&2 & $D_8.\mu_{8}$&(64, 6)&5\\
$C_3.\mu_{14}$&(42, 6)&2   &  $C_4.\mu_{12}$&(48, 20)&1 & $D_8.\mu_{12}$&(96, 52)&1\\
\rowcolor{lightgray}
$Q_8.\mu_{2}$&(16, 8)&7 & $C_2^4.\mu_{4}$&(64, 32)&2 & $A_{3,3}.\mu_{2}$&(36, 10)&11\\
$Q_8.\mu_{2}$&(16, 9)&1 & $C_2^4.\mu_{4}$&(64, 60)&2 & $A_{3,3}.\mu_{2}$&(36, 13)&10\\
\rowcolor{lightgray}
$Q_8.\mu_{2}$&(16, 12)&2 & $C_2^4.\mu_{4}$&(64, 90)&1 & $A_{3,3}.\mu_{3}$&(54, 5)&2\\
$Q_8.\mu_{2}$&(16, 13)&11 & $C_2^4.\mu_{4}$&(64, 193)&1 & $A_{3,3}.\mu_{3}$&(54, 13)&3\\
\rowcolor{lightgray}
$Q_8.\mu_{3}$&(24, 3)&2 & $C_2^4.\mu_{5}$&(80, 49)&1 & $A_{3,3}.\mu_{4}$&(72, 21)&1\\
$Q_8.\mu_{3}$&(24, 11)&1 & $C_2^4.\mu_{6}$&(96, 70)&1 & $A_{3,3}.\mu_{4}$&(72, 39)&1\\
\rowcolor{lightgray}
$Q_8.\mu_{4}$&(32, 11)&5 & $C_2^4.\mu_{6}$&(96, 197)&1 & $A_{3,3}.\mu_{4}$&(72, 45)&3\\
$Q_8.\mu_{4}$&(32, 38)&2 & $C_2^4.\mu_{6}$&(96, 228)&1 & $A_{3,3}.\mu_{6}$&(108, 25)&4\\
\rowcolor{lightgray}
$Q_8.\mu_{6}$&(48, 26)&1 & $C_2^4.\mu_{6}$&(96, 229)&1 & $A_{3,3}.\mu_{6}$&(108, 36)&1\\
$Q_8.\mu_{6}$&(48, 32)&2 & $C_2^4.\mu_{7}$&(112, 41)&1 & $A_{3,3}.\mu_{6}$&(108, 38)&4\\
\rowcolor{lightgray}
$Q_8.\mu_{6}$&(48, 33)&1 & $C_2^4.\mu_{7}$&(112, 41)&1 & $A_{3,3}.\mu_{6}$&(108, 43)&1\\
$Q_8.\mu_{6}$&(48, 46)&1 & $C_2^4.\mu_{8}$&(128, 48)&1 & $A_{3,3}.\mu_{8}$&(144, 185)&1\\
\rowcolor{lightgray}
$D_{10}.\mu_{2}$&(20, 3)&3 & $C_2^4.\mu_{10}$&(160, 235)&1 & $Hol(C_5).\mu_{2}$&(40, 12)&9\\
$D_{10}.\mu_{2}$&(20, 4)&19 & $C_2^4.\mu_{12}$&(192, 994)&1 & $Hol(C_5).\mu_{3}$&(60, 6)&1\\
\rowcolor{lightgray}
$D_{10}.\mu_{3}$&(30, 2)&3 & $C_2 \times D_8.\mu_{2}$&(32, 6)&3 & $Hol(C_5).\mu_{4}$&(80, 30)&1\\
$D_{10}.\mu_{4}$&(40, 5)&3 & $C_2 \times D_8.\mu_{2}$&(32, 7)&1 & $C_7:C_3.\mu_{2}$&(42, 1)&4\\
\rowcolor{lightgray}
$D_{10}.\mu_{4}$&(40, 12)&4 & $C_2 \times D_8.\mu_{2}$&(32, 27)&14 & $C_7:C_3.\mu_{2}$&(42, 2)&2\\
$D_{10}.\mu_{5}$&(50, 3)&1 & $C_2 \times D_8.\mu_{2}$&(32, 28)&2 & $C_7:C_3.\mu_{3}$&(63, 3)&1\\
\rowcolor{lightgray}
$D_{10}.\mu_{6}$&(60, 6)&1 & $C_2 \times D_8.\mu_{2}$&(32, 28)&2 & $C_7:C_3.\mu_{4}$&(84, 2)&1\\
$D_{10}.\mu_{6}$&(60, 10)&2 & $C_2 \times D_8.\mu_{2}$&(32, 30)&1 & $C_7:C_3.\mu_{6}$&(126, 7)&1\\
\rowcolor{lightgray}
$D_{10}.\mu_{8}$&(80, 28)&1 & $C_2 \times D_8.\mu_{2}$&(32, 34)&4 & $C_7:C_3.\mu_{6}$&(126, 10)&1\\
$D_{10}.\mu_{10}$&(100, 14)&1 & $C_2 \times D_8.\mu_{2}$&(32, 39)&2 & $S_4.\mu_{2}$&(48, 48)&74\\
\rowcolor{lightgray}
$D_{10}.\mu_{12}$&(120, 17)&1 & $C_2 \times D_8.\mu_{2}$&(32, 43)&2 & $S_4.\mu_{3}$&(72, 42)&2\\
$A_4.\mu_{2}$&(24, 12)&40 & $C_2 \times D_8.\mu_{2}$&(32, 46)&46 & $S_4.\mu_{4}$&(96, 186)&8\\
\rowcolor{lightgray}
$A_4.\mu_{2}$&(24, 13)&47 & $C_2 \times D_8.\mu_{2}$&(32, 48)&1 & $S_4.\mu_{6}$&(144, 188)&2\\
$A_4.\mu_{3}$&(36, 11)&7 & $C_2 \times D_8.\mu_{2}$&(32, 49)&13 & $2^4C_2.\mu_{2}$&(64, 32)&2\\
\rowcolor{lightgray}
$A_4.\mu_{4}$&(48, 30)&10 & $C_2 \times D_8.\mu_{4}$&(64, 12)&1 & $2^4C_2.\mu_{2}$&(64, 138)&8\\
$A_4.\mu_{4}$&(48, 31)&7 & $C_2 \times D_8.\mu_{4}$&(64, 34)&2 & $2^4C_2.\mu_{2}$&(64, 202)&12\\
\rowcolor{lightgray}
$A_4.\mu_{6}$&(72, 42)&5 & $C_2 \times D_8.\mu_{4}$&(64, 67)&5 & $2^4C_2.\mu_{2}$&(64, 215)&3\\
$A_4.\mu_{6}$&(72, 47)&4 & $C_2 \times D_8.\mu_{4}$&(64, 71)&2 & $2^4C_2.\mu_{2}$&(64, 216)&1\\
\rowcolor{lightgray}
$A_4.\mu_{12}$&(144, 123)&1 & $C_2 \times D_8.\mu_{4}$&(64, 90)&7 & $2^4C_2.\mu_{2}$&(64, 226)&4\\
$D_{12}.\mu_{2}$&(24, 6)&4 & $C_2 \times D_8.\mu_{4}$&(64, 92)&4 & $2^4C_2.\mu_{2}$&(64, 241)&1\\
\rowcolor{lightgray}
$D_{12}.\mu_{2}$&(24, 8)&2 & $C_2 \times D_8.\mu_{4}$&(64, 99)&2 & $2^4C_2.\mu_{3}$&(96, 70)&1\\
$D_{12}.\mu_{2}$&(24, 14)&54 & $C_2 \times D_8.\mu_{4}$&(64, 101)&1 & $2^4C_2.\mu_{4}$&(128, 621)&1\\
\rowcolor{lightgray}
$D_{12}.\mu_{3}$&(36, 12)&7 & $C_2 \times D_8.\mu_{4}$&(64, 102)&1 & $2^4C_2.\mu_{4}$&(128, 645)&1\\
$D_{12}.\mu_{4}$&(48, 14)&3 & $C_2 \times D_8.\mu_{4}$&(64, 196)&2 & $2^4C_2.\mu_{4}$&(128, 850)&3\\
\rowcolor{lightgray}
 $D_{12}.\mu_{4}$&(48, 35)&3 & $C_2 \times D_8.\mu_{4}$&(64, 199)&1 & $2^4C_2.\mu_{4}$&(128, 853)&1\\
$D_{12}.\mu_{6}$&(72, 28)&2 & $C_2 \times D_8.\mu_{8}$&(128, 2)&1 & $2^4C_2.\mu_{4}$&(128, 1090)&1\\
\rowcolor{lightgray}
$D_{12}.\mu_{6}$&(72, 30)&1 & $C_2 \times D_8.\mu_{8}$&(128, 50)&1 & $2^4C_2.\mu_{6}$&(192, 1000)&1\\
$D_{12}.\mu_{6}$&(72, 48)&5 & $C_2 \times D_8.\mu_{8}$&(128, 206)&1 & $Q_8 * Q_8.\mu_{2}$&(64, 134)&2\\
\rowcolor{lightgray}
$D_{12}.\mu_{8}$&(96, 27)&1 & $SD_{16}.\mu_{2}$&(32, 40)&2 & $Q_8 * Q_8.\mu_{2}$&(64, 138)&4\\
$D_{12}.\mu_{8}$&(96, 106)&1 & $SD_{16}.\mu_{2}$&(32, 42)&2 & $Q_8 * Q_8.\mu_{2}$&(64, 139)&1\\
\rowcolor{lightgray}
$D_{12}.\mu_{12}$&(144, 79)&1 & $SD_{16}.\mu_{2}$&(32, 43)&6 & $Q_8 * Q_8.\mu_{2}$&(64, 257)&1\\
$C_2^4.\mu_{2}$&(32, 27)&6 & $SD_{16}.\mu_{3}$&(48, 26)&1 & $Q_8 * Q_8.\mu_{2}$&(64, 264)&4\\
\rowcolor{lightgray}
$C_2^4.\mu_{2}$&(32, 46)&5 & $SD_{16}.\mu_{4}$&(64, 124)&1 & $Q_8 * Q_8.\mu_{2}$&(64, 266)&1\\
$C_2^4.\mu_{2}$&(32, 51)&7 & $SD_{16}.\mu_{4}$&(64, 125)&1 & $Q_8 * Q_8.\mu_{3}$&(96, 201)&1\\
\rowcolor{lightgray}
$C_2^4.\mu_{3}$&(48, 49)&2 & $SD_{16}.\mu_{6}$&(96, 180)&1 & $Q_8 * Q_8.\mu_{3}$&(96, 204)&1\\
$C_2^4.\mu_{3}$&(48, 50)&1 & $A_{3,3}.\mu_{2}$&(36, 9)&2 & $Q_8 * Q_8.\mu_{4}$&(128, 134)&1\\
\rowcolor{lightgray}
$Q_8 * Q_8.\mu_{4}$&(128, 522)&1 & $C_2 \times S_4.\mu_{4}$&(192, 1469)&1 & $2^4D_6.\mu_{6}$&(576, 8656)&1\\
$Q_8 * Q_8.\mu_{4}$&(128, 524)&2 & $T_{48}.\mu_{2}$&(96, 189)&1 & $S_5.\mu_{2}$&(240, 189)&12\\
\rowcolor{lightgray}
$Q_8 * Q_8.\mu_{4}$&(128, 1633)&1 & $T_{48}.\mu_{2}$&(96, 193)&2 & $L_2(7).\mu_{2}$&(336, 208)&8\\
$Q_8 * Q_8.\mu_{6}$&(192, 201)&1 & $T_{48}.\mu_{3}$&(144, 122)&1 & $L_2(7).\mu_{2}$&(336, 209)&4\\
\rowcolor{lightgray}
$Q_8 * Q_8.\mu_{6}$&(192, 1504)&1 & $T_{48}.\mu_{6}$&(288, 900)&1 & $L_2(7).\mu_{4}$&(672, 1046)&1\\
$Q_8 * Q_8.\mu_{6}$&(192, 1509)&1 & $A_5.\mu_{2}$&(120, 34)&7 & $4^2A_4.\mu_{2}$&(384, 591)&1\\
\rowcolor{lightgray}
$Q_8 * Q_8.\mu_{8}$&(256, 332)&1 & $A_5.\mu_{2}$&(120, 35)&12 & $4^2A_4.\mu_{2}$&(384, 18134)&1\\
$Q_8 * Q_8.\mu_{12}$&(384, 5816)&1 & $A_5.\mu_{3}$&(180, 19)&2 & $4^2A_4.\mu_{2}$&(384, 18135)&2\\
\rowcolor{lightgray}
$3^2C_4.\mu_{2}$&(72, 39)&2 & $A_5.\mu_{4}$&(240, 91)&1 & $4^2A_4.\mu_{2}$&(384, 18235)&2\\
$3^2C_4.\mu_{2}$&(72, 40)&6 & $A_5.\mu_{6}$&(360, 119)&1 & $4^2A_4.\mu_{2}$&(384, 18236)&1\\
\rowcolor{lightgray}
$3^2C_4.\mu_{2}$&(72, 41)&1 & $A_5.\mu_{6}$&(360, 122)&1 & $4^2A_4.\mu_{3}$&(576, 5129)&1\\
$3^2C_4.\mu_{2}$&(72, 45)&5 & $\Gamma_{25}a_1.\mu_{2}$&(128, 928)&3 & $4^2A_4.\mu_{4}$&(768, 1083510)&1\\
\rowcolor{lightgray}
$3^2C_4.\mu_{3}$&(108, 36)&1 & $\Gamma_{25}a_1.\mu_{2}$&(128, 932)&1 & $4^2A_4.\mu_{4}$&(768, 1088651)&1\\
$3^2C_4.\mu_{4}$&(144, 120)&2 & $\Gamma_{25}a_1.\mu_{2}$&(128, 1755)&4 & $4^2A_4.\mu_{4}$&(768, 1088659)&1\\
\rowcolor{lightgray}
$3^2C_4.\mu_{4}$&(144, 185)&1 & $\Gamma_{25}a_1.\mu_{2}$&(128, 1758)&1 & $4^2A_4.\mu_{6}$&(1152, 155469)&1\\
$3^2C_4.\mu_{6}$&(216, 157)&1 & $\Gamma_{25}a_1.\mu_{2}$&(128, 1759)&1 & $H_{195}.\mu_{2}$&(384, 17948)&3\\
\rowcolor{lightgray}
$S_{3,3}.\mu_{2}$&(72, 40)&3 & $\Gamma_{25}a_1.\mu_{3}$&(192, 201)&1 & $T_{192}.\mu_{2}$&(384, 5602)&1\\
$S_{3,3}.\mu_{2}$&(72, 46)&8 & $\Gamma_{25}a_1.\mu_{4}$&(256, 6029)&1 & $T_{192}.\mu_{2}$&(384, 5608)&1\\
\rowcolor{lightgray}
$S_{3,3}.\mu_{3}$&(108, 38)&2 & $\Gamma_{25}a_1.\mu_{6}$&(384, 5837)&1 & $T_{192}.\mu_{2}$&(384, 20097)&1\\
$S_{3,3}.\mu_{4}$&(144, 115)&1 & $A_{4,3}.\mu_{2}$&(144, 183)&7 & $T_{192}.\mu_{3}$&(576, 8277)&1\\
\rowcolor{lightgray}
$S_{3,3}.\mu_{6}$&(216, 157)&1 & $A_{4,3}.\mu_{2}$&(144, 189)&3 & $T_{192}.\mu_{6}$&(1152, 157515)&1\\
$S_{3,3}.\mu_{6}$&(216, 170)&1 & $A_{4,3}.\mu_{3}$&(216, 92)&1 & $A_{4,4}.\mu_{2}$&(576, 8653)&1\\
\rowcolor{lightgray}
$2^4C_3.\mu_{2}$&(96, 70)&4 & $A_{4,3}.\mu_{3}$&(216, 164)&1 & $A_{4,4}.\mu_{2}$&(576, 8654)&1\\
$2^4C_3.\mu_{2}$&(96, 227)&9 & $A_{4,3}.\mu_{6}$&(432, 535)&1 & $A_{4,4}.\mu_{2}$&(576, 8657)&1\\
\rowcolor{lightgray}
$2^4C_3.\mu_{2}$&(96, 229)&6 & $A_{4,3}.\mu_{6}$&(432, 745)&1 & $A_{4,4}.\mu_{4}$&(1152, 157850)&1\\
$2^4C_3.\mu_{3}$&(144, 184)&2 & $N_{72}.\mu_{2}$&(144, 182)&1 & $A_6.\mu_{2}$&(720, 763)&2\\
\rowcolor{lightgray}
$2^4C_3.\mu_{4}$&(192, 184)&1 & $N_{72}.\mu_{2}$&(144, 186)&2 & $A_6.\mu_{2}$&(720, 764)&6\\
$2^4C_3.\mu_{4}$&(192, 191)&1 & $N_{72}.\mu_{4}$&(288, 841)&1 & $A_6.\mu_{2}$&(720, 766)&4\\
\rowcolor{lightgray}
$2^4C_3.\mu_{4}$&(192, 1495)&1 & $M_9.\mu_{2}$&(144, 182)&2 & $A_6.\mu_{4}$&(1440, 4595)&1\\
$2^4C_3.\mu_{5}$&(240, 191)&1 & $M_9.\mu_{2}$&(144, 187)&1 & $F_{384}.\mu_{2}$&(768, 1086051)&1\\
\rowcolor{lightgray}
$2^4C_3.\mu_{6}$&(288, 1025)&2 & $M_9.\mu_{3}$&(216, 153)&1 & $F_{384}.\mu_{2}$&(768, 1090134)&1\\
$2^4C_3.\mu_{6}$&(288, 1029)&1 & $M_9.\mu_{6}$&(432, 735)&1 & $F_{384}.\mu_{2}$&(768, 1090135)&1\\
\rowcolor{lightgray}
$C_2 \times S_4.\mu_{2}$&(96, 187)&2 & $2^4D_6.\mu_{2}$&(192, 955)&8 & $F_{384}.\mu_{4}$&(1536, 'no id')&1\\
$C_2 \times S_4.\mu_{2}$&(96, 195)&1 & $2^4D_6.\mu_{2}$&(192, 1538)&9 & $M_{20}.\mu_{2}$&(1920, 240993)&1\\
\rowcolor{lightgray}
$C_2 \times S_4.\mu_{2}$&(96, 226)&15 & $2^4D_6.\mu_{3}$&(288, 1025)&1 & $M_{20}.\mu_{2}$&(1920, 240995)&2\\
$C_2 \times S_4.\mu_{4}$&(192, 972)&2 & $2^4D_6.\mu_{4}$&(384, 5566)&1 & $M_{20}.\mu_{4}$&(3840, 'no id')&1\\
\bottomrule
\end{longtable}
}

\section{Fixed loci of purely non-symplectic automorphisms}
\label{appendixB}
 Let $X$ be a K3 surface and $\sigma \in \Aut(X)$ a purely non-symplectic automorphism of order $n$ acting by $\zeta_n$ on $\HH^0(X,\Omega_X^2)$.
 Recall that the fixed locus $X^\sigma$ is the disjoint union of $N = \sum_{i=1}^s a_i$ isolated fixed points, $k$ smooth rational curves and either a curve of genus $>1$
or $0, 1, 2$ curves of genus 1.
Denote by $l\geq 0$ the number of genus $g \geq 1$ curves fixed by $\sigma$. If no such curve is fixed, set $g=1$. Thus
\[X^\sigma = \{p_1, \dots, p_N\} \sqcup R_1 , \dots \sqcup R_k \sqcup C_1 \dots \sqcup C_l\]
where the $R_i$'s are smooth rational curves and the $C_j$'s smooth curves of genus $g$.
Let $P \in X^\sigma$ be an isolated fixed point. Recall that there are local coordinates $(x,y)$ in a small neighborhood centered at $P$ such that
\[\sigma(x,y) = (\zeta_n^{i+1}x,\zeta_n^{-i} y)\quad \mbox{  with } \quad 1 \leq i \leq s = \left\lfloor\frac{n-1}{2}\right\rfloor.\]
We call $P$ a fixed point of type $i$ and denote the number of fixed points of type $i$ by $a_i$.

In the following we list for each deformation class of $(X, \sigma)$ the invariants $((a_1, \dots, a_s), k, l, g)$ of the fixed locus of $\sigma$ and its powers.
The column labeled `K3 id' contains the label of the K3 surface in the database \cite{K3Groups}.
The following columns contain the invariants of the fixed locus of $\sigma^{n/j}$ where $n$ is the order of $\sigma$ and $j$ the label of the column.

{\small\centering\setlength{\tabcolsep}{3pt}
\begin{longtable}{ccccc|ccccc|ccccc|ccccc}
\caption{Fixed loci of purely non-symplectic automorphisms of order 2}\\
 & \multicolumn{4}{c}{$j = 2$} & & \multicolumn{4}{c}{$j = 2$} & & \multicolumn{4}{c}{$j = 2$} & & \multicolumn{4}{c}{$j = 2$}\\
\cmidrule(lr){2-5}
\cmidrule(lr){7-10}
\cmidrule(lr){12-15}
\cmidrule(lr){17-20}
K3 id & $(a_i)$ & $k$ & $l$ & $g$ & K3 id & $(a_i)$ & $k$ & $l$ & $g$ & K3 id & $(a_i)$ & $k$ & $l$ & $g$ & K3 id & $(a_i)$ & $k$ & $l$ & $g$\\
\hline
\endhead
\rowcolor{lightgray}
0.2.1.1 & () & 10 & 0 & 1 & 0.2.7.3 & () & 5 & 1 & 2 & 0.2.11.1 & () & 5 & 1 & 6 & 0.2.13.4 & () & 0 & 1 & 3\\
0.2.2.1 & () & 9 & 0 & 1 & 0.2.7.4 & () & 4 & 1 & 1 & 0.2.11.2 & () & 4 & 1 & 5 & 0.2.14.1 & () & 0 & 1 & 4\\
\rowcolor{lightgray}
0.2.2.2 & () & 9 & 1 & 1 & 0.2.7.5 & () & 4 & 1 & 1 & 0.2.11.3 & () & 4 & 1 & 5 & 0.2.14.2 & () & 1 & 1 & 5\\
0.2.3.1 & () & 9 & 1 & 2 & 0.2.7.6 & () & 4 & 0 & 1 & 0.2.11.4 & () & 3 & 1 & 4 & 0.2.14.3 & () & 2 & 1 & 6\\
\rowcolor{lightgray}
0.2.3.2 & () & 8 & 1 & 1 & 0.2.8.1 & () & 3 & 0 & 1 & 0.2.11.5 & () & 3 & 1 & 4 & 0.2.15.1 & () & 2 & 1 & 7\\
0.2.3.3 & () & 8 & 1 & 1 & 0.2.8.2 & () & 3 & 1 & 1 & 0.2.11.6 & () & 2 & 1 & 3 & 0.2.15.2 & () & 1 & 1 & 6\\
\rowcolor{lightgray}
0.2.3.4 & () & 8 & 0 & 1 & 0.2.8.3 & () & 4 & 1 & 2 & 0.2.11.7 & () & 2 & 1 & 3 & 0.2.15.3 & () & 1 & 1 & 6\\
0.2.3.5 & () & 8 & 0 & 1 & 0.2.8.4 & () & 5 & 1 & 3 & 0.2.11.8 & () & 0 & 2 & 1 & 0.2.15.4 & () & 0 & 1 & 5\\
\hline\end{longtable}
}

{\small\centering\setlength{\tabcolsep}{3pt}
\begin{longtable}{ccccc|ccccc|ccccc|ccccc}
\caption{Fixed loci of purely non-symplectic automorphisms of order 3}\\
 & \multicolumn{4}{c}{$j = 3$} & & \multicolumn{4}{c}{$j = 3$} & & \multicolumn{4}{c}{$j = 3$} & & \multicolumn{4}{c}{$j = 3$}\\
\cmidrule(lr){2-5}
\cmidrule(lr){7-10}
\cmidrule(lr){12-15}
\cmidrule(lr){17-20}
K3 id & $(a_i)$ & $k$ & $l$ & $g$ & K3 id & $(a_i)$ & $k$ & $l$ & $g$ & K3 id & $(a_i)$ & $k$ & $l$ & $g$ & K3 id & $(a_i)$ & $k$ & $l$ & $g$\\
\hline
\endhead
\rowcolor{lightgray}
0.3.0.1 & (9) & 6 & 0 & 1 & 0.3.3.2 & (6) & 3 & 0 & 1 & 0.3.5.3 & (4) & 1 & 1 & 1 & 0.3.7.1 & (2) & 1 & 1 & 3\\
0.3.1.1 & (8) & 5 & 1 & 1 & 0.3.4.1 & (5) & 3 & 1 & 2 & 0.3.5.4 & (4) & 1 & 0 & 1 & 0.3.7.2 & (2) & 0 & 1 & 2\\
\rowcolor{lightgray}
0.3.1.2 & (8) & 5 & 0 & 1 & 0.3.4.2 & (5) & 2 & 1 & 1 & 0.3.6.1 & (3) & 2 & 1 & 3 & 0.3.8.1 & (1) & 1 & 1 & 4\\
0.3.2.1 & (7) & 4 & 1 & 1 & 0.3.4.3 & (5) & 2 & 0 & 1 & 0.3.6.2 & (3) & 1 & 1 & 2 & 0.3.8.2 & (1) & 0 & 1 & 3\\
\rowcolor{lightgray}
0.3.2.2 & (7) & 4 & 0 & 1 & 0.3.5.1 & (4) & 3 & 1 & 3 & 0.3.6.3 & (3) & 0 & 1 & 1 & 0.3.9.1 & (0) & 1 & 1 & 5\\
0.3.3.1 & (6) & 3 & 1 & 1 & 0.3.5.2 & (4) & 2 & 1 & 2 & 0.3.6.4 & (3) & 0 & 0 & 1 & 0.3.9.2 & (0) & 0 & 1 & 4\\
\hline\end{longtable}
}

{\small\centering\setlength{\tabcolsep}{3pt}
\begin{longtable}{ccccccccc|ccccccccc|ccccccccc}
\caption{Fixed loci of purely non-symplectic automorphisms of order 4}\\
 & \multicolumn{4}{c}{$j = 4$} & \multicolumn{4}{c}{$j = 2$} & & \multicolumn{4}{c}{$j = 4$} & \multicolumn{4}{c}{$j = 2$} & & \multicolumn{4}{c}{$j = 4$} & \multicolumn{4}{c}{$j = 2$}\\
\cmidrule(lr){2-5}
\cmidrule(lr){6-9}
\cmidrule(lr){11-14}
\cmidrule(lr){15-18}
\cmidrule(lr){20-23}
\cmidrule(lr){24-27}
K3 id & $(a_i)$ & $k$ & $l$ & $g$ & $(a_i)$ & $k$ & $l$ & $g$ & K3 id & $(a_i)$ & $k$ & $l$ & $g$ & $(a_i)$ & $k$ & $l$ & $g$ & K3 id & $(a_i)$ & $k$ & $l$ & $g$ & $(a_i)$ & $k$ & $l$ & $g$\\
\hline
\endhead
\rowcolor{lightgray}
0.4.0.1 & (4) & 0 & 0 & 1 & () & 10 & 0 & 1 & 0.4.1.4 & (4) & 0 & 0 & 1 & () & 8 & 1 & 1 & 0.4.1.10 & (6) & 1 & 0 & 1 & () & 8 & 0 & 1\\
0.4.0.2 & (6) & 1 & 0 & 1 & () & 10 & 0 & 1 & 0.4.1.5 & (4) & 0 & 0 & 1 & () & 8 & 1 & 1 & 0.4.1.11 & (8) & 2 & 0 & 1 & () & 8 & 0 & 1\\
\rowcolor{lightgray}
0.4.0.3 & (12) & 4 & 0 & 1 & () & 10 & 0 & 1 & 0.4.1.6 & (4) & 0 & 1 & 1 & () & 8 & 1 & 1 & 0.4.1.12 & (10) & 3 & 0 & 1 & () & 8 & 0 & 1\\
0.4.1.1 & (4) & 0 & 0 & 1 & () & 9 & 1 & 2 & 0.4.1.7 & (12) & 4 & 0 & 1 & () & 8 & 1 & 1 & 0.4.2.1 & (4) & 0 & 0 & 1 & () & 7 & 1 & 2\\
\rowcolor{lightgray}
0.4.1.2 & (12) & 4 & 0 & 1 & () & 9 & 1 & 2 & 0.4.1.8 & (4) & 0 & 0 & 1 & () & 8 & 0 & 1 & 0.4.2.2 & (10) & 3 & 0 & 1 & () & 7 & 1 & 2\\
0.4.1.3 & (4) & 0 & 0 & 1 & () & 8 & 1 & 1 & 0.4.1.9 & (4) & 0 & 0 & 1 & () & 8 & 0 & 1 & 0.4.2.3 & (4) & 0 & 1 & 1 & () & 6 & 1 & 1\\
\rowcolor{lightgray}
0.4.2.4 & (4) & 0 & 0 & 1 & () & 6 & 1 & 1 & 0.4.4.2 & (8) & 2 & 0 & 1 & () & 4 & 1 & 3 & 0.4.5.12 & (4) & 0 & 0 & 1 & () & 0 & 2 & 1\\
0.4.2.5 & (4) & 0 & 0 & 1 & () & 6 & 1 & 1 & 0.4.4.3 & (2) & 0 & 1 & 2 & () & 3 & 1 & 2 & 0.4.5.13 & (4) & 0 & 0 & 1 & () & 1 & 1 & 2\\
\hline\end{longtable}
}

{\small\centering\setlength{\tabcolsep}{3pt}
\begin{longtable}{ccccc|ccccc|ccccc}
\caption{Fixed loci of purely non-symplectic automorphisms of order 5}\\
 & \multicolumn{4}{c}{$j = 5$} & & \multicolumn{4}{c}{$j = 5$} & & \multicolumn{4}{c}{$j = 5$}\\
\cmidrule(lr){2-5}
\cmidrule(lr){7-10}
\cmidrule(lr){12-15}
K3 id & $(a_i)$ & $k$ & $l$ & $g$ & K3 id & $(a_i)$ & $k$ & $l$ & $g$ & K3 id & $(a_i)$ & $k$ & $l$ & $g$\\
\hline
\endhead
\rowcolor{lightgray}
0.5.0.1 & (9\:4) & 3 & 0 & 1 & 0.5.2.2 & (5\:2) & 1 & 0 & 1 & 0.5.4.1 & (1\:0) & 0 & 1 & 2\\
  0.5.1.1 & (7\:3) & 2 & 0 & 1 & 0.5.3.1 & (3\:1) & 0 & 1 & 1 & & & & \\
\rowcolor{lightgray}
  0.5.2.1 & (5\:2) & 1 & 1 & 1 & 0.5.3.2 & (3\:1) & 0 & 0 & 1 & & & & & \\
\hline\end{longtable}
}

{\small\centering\setlength{\tabcolsep}{3pt}
\begin{longtable}{ccccccccccccc|ccccccccccccc}
\caption{Fixed loci of purely non-symplectic automorphisms of order 6}\\
 & \multicolumn{4}{c}{$j = 6$} & \multicolumn{4}{c}{$j = 3$} & \multicolumn{4}{c}{$j = 2$} & & \multicolumn{4}{c}{$j = 6$} & \multicolumn{4}{c}{$j = 3$} & \multicolumn{4}{c}{$j = 2$}\\
\cmidrule(lr){2-5}
\cmidrule(lr){6-9}
\cmidrule(lr){10-13}
\cmidrule(lr){15-18}
\cmidrule(lr){19-22}
\cmidrule(lr){23-26}
K3 id & $(a_i)$ & $k$ & $l$ & $g$ & $(a_i)$ & $k$ & $l$ & $g$ & $(a_i)$ & $k$ & $l$ & $g$ & K3 id & $(a_i)$ & $k$ & $l$ & $g$ & $(a_i)$ & $k$ & $l$ & $g$ & $(a_i)$ & $k$ & $l$ & $g$\\
\hline
\endhead
\rowcolor{lightgray}
0.6.0.1 & (3\:0) & 0 & 0 & 1 & (3) & 0 & 0 & 1 & () & 9 & 1 & 1 & 0.6.1.6 & (0\:6) & 0 & 0 & 1 & (8) & 5 & 0 & 1 & () & 0 & 2 & 1\\
0.6.0.2 & (3\:0) & 0 & 0 & 1 & (9) & 6 & 0 & 1 & () & 0 & 1 & 1 & 0.6.1.7 & (4\:4) & 1 & 0 & 1 & (8) & 5 & 0 & 1 & () & 3 & 1 & 2\\
\rowcolor{lightgray}
0.6.0.3 & (1\:4) & 0 & 0 & 1 & (9) & 6 & 0 & 1 & () & 1 & 1 & 1 & 0.6.1.8 & (6\:6) & 2 & 0 & 1 & (8) & 5 & 0 & 1 & () & 6 & 1 & 2\\
0.6.0.4 & (3\:6) & 1 & 0 & 1 & (9) & 6 & 0 & 1 & () & 4 & 1 & 1 & 0.6.1.9 & (2\:2) & 0 & 0 & 1 & (2) & 0 & 1 & 2 & () & 8 & 0 & 1\\
\rowcolor{lightgray}
0.6.0.5 & (9\:6) & 3 & 0 & 1 & (9) & 6 & 0 & 1 & () & 9 & 1 & 1 & 0.6.1.10 & (4\:4) & 1 & 0 & 1 & (4) & 1 & 1 & 1 & () & 8 & 0 & 1\\
0.6.0.6 & (2\:2) & 0 & 0 & 1 & (4) & 1 & 0 & 1 & () & 8 & 0 & 1 & 0.6.1.11 & (6\:6) & 2 & 0 & 1 & (6) & 3 & 1 & 1 & () & 8 & 0 & 1\\
\rowcolor{lightgray}
0.6.0.7 & (2\:2) & 0 & 0 & 1 & (8) & 5 & 0 & 1 & () & 2 & 0 & 1 & 0.6.1.12 & (3\:0) & 0 & 0 & 1 & (3) & 0 & 0 & 1 & () & 6 & 1 & 1\\
0.6.0.8 & (6\:6) & 2 & 0 & 1 & (8) & 5 & 0 & 1 & () & 8 & 0 & 1 & 0.6.1.13 & (3\:0) & 0 & 0 & 1 & (7) & 4 & 0 & 1 & () & 0 & 1 & 1\\
\hline\end{longtable}
}

{\small\centering\setlength{\tabcolsep}{3pt}
\begin{longtable}{ccccc|ccccc}
\caption{Fixed loci of purely non-symplectic automorphisms of order 7}\\
 & \multicolumn{4}{c}{$j = 7$} & & \multicolumn{4}{c}{$j = 7$}\\
\cmidrule(lr){2-5}
\cmidrule(lr){7-10}
K3 id & $(a_i)$ & $k$ & $l$ & $g$ & K3 id & $(a_i)$ & $k$ & $l$ & $g$\\
\hline
\endhead
\rowcolor{lightgray}
0.7.0.1 & (6\:5\:2) & 2 & 0 & 1 & 0.7.2.1 & (2\:1\:0) & 0 & 1 & 1\\
0.7.1.1 & (4\:3\:1) & 1 & 1 & 1 & 0.7.2.2 & (2\:1\:0) & 0 & 0 & 1\\
\rowcolor{lightgray}
0.7.1.2 & (4\:3\:1) & 1 & 0 & 1\\
\hline\end{longtable}
}

{\small\centering\setlength{\tabcolsep}{3pt}
\begin{longtable}{ccccccccccccc|ccccccccccccc}
\caption{Fixed loci of purely non-symplectic automorphisms of order 8}\\
 & \multicolumn{4}{c}{$j = 8$} & \multicolumn{4}{c}{$j = 4$} & \multicolumn{4}{c}{$j = 2$} & & \multicolumn{4}{c}{$j = 8$} & \multicolumn{4}{c}{$j = 4$} & \multicolumn{4}{c}{$j = 2$}\\
\cmidrule(lr){2-5}
\cmidrule(lr){6-9}
\cmidrule(lr){10-13}
\cmidrule(lr){15-18}
\cmidrule(lr){19-22}
\cmidrule(lr){23-26}
K3 id & $(a_i)$ & $k$ & $l$ & $g$ & $(a_i)$ & $k$ & $l$ & $g$ & $(a_i)$ & $k$ & $l$ & $g$ & K3 id & $(a_i)$ & $k$ & $l$ & $g$ & $(a_i)$ & $k$ & $l$ & $g$ & $(a_i)$ & $k$ & $l$ & $g$\\
\hline
\endhead
\rowcolor{lightgray}
0.8.0.1 & (2\:0\:0) & 0 & 0 & 1 & (4) & 0 & 0 & 1 & () & 8 & 1 & 1 & 0.8.2.1 & (4\:2\:2) & 1 & 0 & 1 & (8) & 2 & 0 & 1 & () & 4 & 1 & 5\\
0.8.0.2 & (2\:0\:0) & 0 & 0 & 1 & (12) & 4 & 0 & 1 & () & 8 & 1 & 1 & 0.8.2.2 & (1\:1\:2) & 0 & 0 & 1 & (6) & 1 & 0 & 1 & () & 3 & 1 & 4\\
\rowcolor{lightgray}
0.8.0.3 & (0\:2\:4) & 0 & 0 & 1 & (12) & 4 & 0 & 1 & () & 8 & 1 & 1 & 0.8.2.3 & (4\:2\:2) & 1 & 0 & 1 & (8) & 2 & 0 & 1 & () & 3 & 1 & 4\\
0.8.0.4 & (6\:4\:4) & 2 & 0 & 1 & (12) & 4 & 0 & 1 & () & 8 & 1 & 1 & 0.8.2.4 & (1\:1\:2) & 0 & 0 & 1 & (6) & 1 & 0 & 1 & () & 2 & 1 & 3\\
\rowcolor{lightgray}
0.8.0.5 & (1\:1\:2) & 0 & 0 & 1 & (6) & 1 & 0 & 1 & () & 8 & 0 & 1 & 0.8.2.5 & (2\:0\:0) & 0 & 0 & 1 & (8) & 2 & 0 & 1 & () & 2 & 1 & 3\\
0.8.0.6 & (1\:1\:2) & 0 & 0 & 1 & (10) & 3 & 0 & 1 & () & 8 & 0 & 1 & 0.8.2.6 & (5\:1\:0) & 1 & 0 & 1 & (6) & 1 & 0 & 1 & () & 2 & 1 & 3\\
\rowcolor{lightgray}
0.8.0.7 & (3\:3\:4) & 1 & 0 & 1 & (10) & 3 & 0 & 1 & () & 8 & 0 & 1 & 0.8.2.7 & (0\:2\:4) & 0 & 0 & 1 & (8) & 2 & 0 & 1 & () & 2 & 1 & 3\\
0.8.1.1 & (7\:3\:2) & 2 & 0 & 1 & (10) & 3 & 0 & 1 & () & 6 & 1 & 3 & 0.8.2.8 & (2\:0\:0) & 0 & 1 & 1 & (4) & 0 & 1 & 1 & () & 0 & 2 & 1\\
\hline\end{longtable}
}

{\small\centering\setlength{\tabcolsep}{3pt}
\begin{longtable}{ccccccccc|ccccccccc}
\caption{Fixed loci of purely non-symplectic automorphisms of order 9}\\
 & \multicolumn{4}{c}{$j = 9$} & \multicolumn{4}{c}{$j = 3$} & & \multicolumn{4}{c}{$j = 9$} & \multicolumn{4}{c}{$j = 3$}\\
\cmidrule(lr){2-5}
\cmidrule(lr){6-9}
\cmidrule(lr){11-14}
\cmidrule(lr){15-18}
K3 id & $(a_i)$ & $k$ & $l$ & $g$ & $(a_i)$ & $k$ & $l$ & $g$ & K3 id & $(a_i)$ & $k$ & $l$ & $g$ & $(a_i)$ & $k$ & $l$ & $g$\\
\hline
\endhead
\rowcolor{lightgray}
0.9.0.1 & (6\:5\:2\:1) & 2 & 0 & 1 & (7) & 4 & 1 & 1 & 0.9.1.5 & (0\:1\:4\:1) & 0 & 0 & 1 & (4) & 1 & 1 & 1\\
0.9.0.2 & (1\:1\:1\:0) & 0 & 0 & 1 & (7) & 4 & 0 & 1 & 0.9.1.6 & (4\:3\:0\:0) & 1 & 0 & 1 & (4) & 1 & 1 & 1\\
\rowcolor{lightgray}
0.9.0.3 & (3\:3\:3\:1) & 1 & 0 & 1 & (7) & 4 & 0 & 1 & 0.9.1.7 & (1\:1\:1\:0) & 0 & 0 & 1 & (4) & 1 & 0 & 1\\
0.9.1.1 & (3\:3\:3\:1) & 1 & 0 & 1 & (4) & 3 & 1 & 3 & 0.9.2.1 & (0\:1\:4\:1) & 0 & 0 & 1 & (1) & 1 & 1 & 4\\
\rowcolor{lightgray}
0.9.1.2 & (3\:3\:3\:1) & 1 & 0 & 1 & (4) & 2 & 1 & 2 & 0.9.2.2 & (1\:1\:1\:0) & 0 & 0 & 1 & (1) & 0 & 1 & 3\\
0.9.1.3 & (1\:1\:1\:0) & 0 & 0 & 1 & (4) & 1 & 1 & 1 & 0.9.2.3 & (0\:1\:4\:1) & 0 & 0 & 1 & (1) & 0 & 1 & 3\\
\rowcolor{lightgray}
  0.9.1.4 & (1\:1\:1\:0) & 0 & 1 & 1 & (4) & 1 & 1 & 1 & & & & & & & & & \\
\hline\end{longtable}
}

{\small\centering\setlength{\tabcolsep}{3pt}
\begin{longtable}{ccccccccccccc|ccccccccccccc}
\caption{Fixed loci of purely non-symplectic automorphisms of order 10}\\
 & \multicolumn{4}{c}{$j = 10$} & \multicolumn{4}{c}{$j = 5$} & \multicolumn{4}{c}{$j = 2$} & & \multicolumn{4}{c}{$j = 10$} & \multicolumn{4}{c}{$j = 5$} & \multicolumn{4}{c}{$j = 2$}\\
\cmidrule(lr){2-5}
\cmidrule(lr){6-9}
\cmidrule(lr){10-13}
\cmidrule(lr){15-18}
\cmidrule(lr){19-22}
\cmidrule(lr){23-26}
K3 id & $(a_i)$ & $k$ & $l$ & $g$ & $(a_i)$ & $k$ & $l$ & $g$ & $(a_i)$ & $k$ & $l$ & $g$ & K3 id & $(a_i)$ & $k$ & $l$ & $g$ & $(a_i)$ & $k$ & $l$ & $g$ & $(a_i)$ & $k$ & $l$ & $g$\\
\hline
\endhead
\rowcolor{lightgray}
0.10.0.1 & (1\:0\:0\:2) & 0 & 0 & 1 & (9\:4) & 3 & 0 & 1 & () & 0 & 1 & 2 & 0.10.1.13 & (1\:1\:0\:0) & 0 & 0 & 1 & (3\:1) & 0 & 0 & 1 & () & 1 & 1 & 1\\
0.10.0.2 & (0\:2\:1\:2) & 0 & 0 & 1 & (9\:4) & 3 & 0 & 1 & () & 1 & 1 & 2 & 0.10.1.14 & (1\:1\:0\:0) & 0 & 0 & 1 & (3\:1) & 0 & 0 & 1 & () & 1 & 0 & 1\\
\rowcolor{lightgray}
0.10.0.3 & (3\:2\:2\:4) & 1 & 0 & 1 & (9\:4) & 3 & 0 & 1 & () & 5 & 1 & 2 & 0.10.2.1 & (0\:0\:1\:6) & 0 & 0 & 1 & (5\:2) & 1 & 1 & 1 & () & 1 & 1 & 6\\
0.10.0.4 & (6\:4\:3\:2) & 2 & 0 & 1 & (9\:4) & 3 & 0 & 1 & () & 8 & 1 & 2 & 0.10.2.2 & (3\:2\:2\:4) & 1 & 0 & 1 & (5\:2) & 1 & 1 & 1 & () & 4 & 1 & 6\\
\rowcolor{lightgray}
0.10.0.5 & (1\:1\:0\:0) & 0 & 0 & 1 & (3\:1) & 0 & 0 & 1 & () & 6 & 0 & 1 & 0.10.2.3 & (0\:1\:1\:4) & 0 & 0 & 1 & (3\:1) & 0 & 1 & 1 & () & 1 & 1 & 4\\
0.10.0.6 & (1\:1\:0\:0) & 0 & 0 & 1 & (7\:3) & 2 & 0 & 1 & () & 1 & 0 & 1 & 0.10.2.4 & (1\:0\:0\:2) & 0 & 0 & 1 & (5\:2) & 1 & 0 & 1 & () & 0 & 1 & 7\\
\rowcolor{lightgray}
0.10.0.7 & (3\:3\:2\:2) & 1 & 0 & 1 & (7\:3) & 2 & 0 & 1 & () & 6 & 0 & 1 & 0.10.2.5 & (0\:2\:1\:2) & 0 & 0 & 1 & (5\:2) & 1 & 0 & 1 & () & 1 & 1 & 7\\
0.10.1.1 & (3\:2\:2\:4) & 1 & 0 & 1 & (5\:2) & 1 & 1 & 1 & () & 5 & 1 & 2 & 0.10.2.6 & (0\:2\:1\:2) & 0 & 0 & 1 & (5\:2) & 1 & 0 & 1 & () & 0 & 1 & 6\\
\hline\end{longtable}
}

{\small\centering\setlength{\tabcolsep}{3pt}
\begin{longtable}{ccccc}
\caption{Fixed loci of purely non-symplectic automorphisms of order 11}\\
 & \multicolumn{4}{c}{$j = 11$}\\
\cmidrule(lr){2-5}
K3 id & $(a_i)$ & $k$ & $l$ & $g$\\
\hline
\endhead
\rowcolor{lightgray}
0.11.0.1 & (3\:2\:2\:3\:1) & 1 & 0 & 1\\
0.11.1.1 & (1\:0\:0\:1\:0) & 0 & 1 & 1\\
\rowcolor{lightgray}
0.11.1.2 & (1\:0\:0\:1\:0) & 0 & 0 & 1\\
\hline\end{longtable}
}

{\small\centering\setlength{\tabcolsep}{3pt}
\begin{longtable}{ccccccccccccccccccccc}
\caption{Fixed loci of purely non-symplectic automorphisms of order 12}\\
 & \multicolumn{4}{c}{$j = 12$} & \multicolumn{4}{c}{$j = 6$} & \multicolumn{4}{c}{$j = 4$} & \multicolumn{4}{c}{$j = 3$} & \multicolumn{4}{c}{$j = 2$}\\
\cmidrule(lr){2-5}
\cmidrule(lr){6-9}
\cmidrule(lr){10-13}
\cmidrule(lr){14-17}
\cmidrule(lr){18-21}
K3 id & $(a_i)$ & $k$ & $l$ & $g$ & $(a_i)$ & $k$ & $l$ & $g$ & $(a_i)$ & $k$ & $l$ & $g$ & $(a_i)$ & $k$ & $l$ & $g$ & $(a_i)$ & $k$ & $l$ & $g$\\
\hline
\endhead
\rowcolor{lightgray}
0.12.0.1 & (0\:2\:0\:0\:2) & 0 & 0 & 1 & (8\:8) & 3 & 0 & 1 & (4) & 0 & 0 & 1 & (8) & 5 & 1 & 1 & () & 9 & 1 & 2\\
0.12.0.2 & (6\:4\:2\:2\:2) & 2 & 0 & 1 & (8\:8) & 3 & 0 & 1 & (12) & 4 & 0 & 1 & (8) & 5 & 1 & 1 & () & 9 & 1 & 2\\
\rowcolor{lightgray}
0.12.0.3 & (0\:0\:2\:2\:2) & 0 & 0 & 1 & (4\:4) & 1 & 0 & 1 & (8) & 2 & 0 & 1 & (4) & 1 & 1 & 1 & () & 8 & 0 & 1\\
0.12.0.4 & (3\:3\:1\:1\:2) & 1 & 0 & 1 & (6\:6) & 2 & 0 & 1 & (8) & 2 & 0 & 1 & (6) & 3 & 1 & 1 & () & 8 & 0 & 1\\
\rowcolor{lightgray}
0.12.0.5 & (1\:0\:0\:0\:0) & 0 & 0 & 1 & (3\:0) & 0 & 0 & 1 & (4) & 0 & 0 & 1 & (3) & 0 & 0 & 1 & () & 6 & 1 & 1\\
0.12.0.6 & (1\:0\:0\:0\:0) & 0 & 0 & 1 & (3\:0) & 0 & 0 & 1 & (4) & 0 & 0 & 1 & (7) & 4 & 0 & 1 & () & 0 & 1 & 1\\
\rowcolor{lightgray}
0.12.0.7 & (1\:0\:0\:0\:0) & 0 & 0 & 1 & (3\:0) & 0 & 0 & 1 & (10) & 3 & 0 & 1 & (3) & 0 & 0 & 1 & () & 6 & 1 & 1\\
0.12.0.8 & (1\:0\:0\:0\:0) & 0 & 0 & 1 & (7\:4) & 2 & 0 & 1 & (4) & 0 & 0 & 1 & (7) & 4 & 0 & 1 & () & 6 & 1 & 1\\
\hline\end{longtable}
}

{\small\centering\setlength{\tabcolsep}{3pt}
\begin{longtable}{ccccccccccccc}
\caption{Fixed loci of purely non-symplectic automorphisms of order 14}\\
 & \multicolumn{4}{c}{$j = 14$} & \multicolumn{4}{c}{$j = 7$} & \multicolumn{4}{c}{$j = 2$}\\
\cmidrule(lr){2-5}
\cmidrule(lr){6-9}
\cmidrule(lr){10-13}
K3 id & $(a_i)$ & $k$ & $l$ & $g$ & $(a_i)$ & $k$ & $l$ & $g$ & $(a_i)$ & $k$ & $l$ & $g$\\
\hline
\endhead
\rowcolor{lightgray}
0.14.0.1 & (0\:1\:2\:0\:0\:0) & 0 & 0 & 1 & (6\:5\:2) & 2 & 0 & 1 & () & 0 & 1 & 3\\
0.14.0.2 & (0\:0\:0\:1\:2\:4) & 0 & 0 & 1 & (6\:5\:2) & 2 & 0 & 1 & () & 2 & 1 & 3\\
\rowcolor{lightgray}
0.14.0.3 & (3\:2\:2\:1\:1\:2) & 1 & 0 & 1 & (6\:5\:2) & 2 & 0 & 1 & () & 5 & 1 & 3\\
0.14.0.4 & (0\:0\:1\:1\:2\:2) & 0 & 0 & 1 & (4\:3\:1) & 1 & 0 & 1 & () & 3 & 0 & 1\\
\rowcolor{lightgray}
0.14.1.1 & (3\:2\:1\:1\:1\:4) & 1 & 0 & 1 & (4\:3\:1) & 1 & 1 & 1 & () & 5 & 1 & 6\\
0.14.1.2 & (0\:1\:1\:0\:0\:2) & 0 & 0 & 1 & (4\:3\:1) & 1 & 0 & 1 & () & 0 & 1 & 6\\
\rowcolor{lightgray}
0.14.1.3 & (0\:0\:1\:1\:2\:2) & 0 & 0 & 1 & (4\:3\:1) & 1 & 0 & 1 & () & 1 & 1 & 6\\
0.14.1.4 & (0\:0\:1\:1\:2\:2) & 0 & 0 & 1 & (4\:3\:1) & 1 & 0 & 1 & () & 2 & 1 & 7\\
\rowcolor{lightgray}
0.14.1.5 & (4\:1\:1\:0\:0\:0) & 1 & 0 & 1 & (4\:3\:1) & 1 & 0 & 1 & () & 2 & 1 & 6\\
0.14.1.6 & (0\:0\:0\:1\:2\:4) & 0 & 0 & 1 & (2\:1\:0) & 0 & 1 & 1 & () & 2 & 1 & 3\\
\rowcolor{lightgray}
0.14.2.1 & (0\:1\:0\:0\:0\:4) & 0 & 0 & 1 & (2\:1\:0) & 0 & 1 & 1 & () & 0 & 1 & 9\\
0.14.2.2 & (0\:0\:0\:1\:2\:4) & 0 & 0 & 1 & (2\:1\:0) & 0 & 1 & 1 & () & 1 & 1 & 9\\
\hline\end{longtable}
}

{\small\centering\setlength{\tabcolsep}{3pt}
\begin{longtable}{ccccccccccccc}
\caption{Fixed loci of purely non-symplectic automorphisms of order 15}\\
 & \multicolumn{4}{c}{$j = 15$} & \multicolumn{4}{c}{$j = 5$} & \multicolumn{4}{c}{$j = 3$}\\
\cmidrule(lr){2-5}
\cmidrule(lr){6-9}
\cmidrule(lr){10-13}
K3 id & $(a_i)$ & $k$ & $l$ & $g$ & $(a_i)$ & $k$ & $l$ & $g$ & $(a_i)$ & $k$ & $l$ & $g$\\
\hline
\endhead
\rowcolor{lightgray}
0.15.0.1 & (0\:0\:1\:2\:2\:0\:0) & 0 & 0 & 1 & (7\:3) & 2 & 0 & 1 & (2) & 0 & 1 & 2\\
0.15.0.2 & (3\:3\:1\:1\:1\:0\:0) & 1 & 0 & 1 & (7\:3) & 2 & 0 & 1 & (4) & 2 & 1 & 2\\
\rowcolor{lightgray}
0.15.0.3 & (3\:2\:1\:1\:1\:3\:1) & 1 & 0 & 1 & (7\:3) & 2 & 0 & 1 & (5) & 3 & 1 & 2\\
0.15.0.4 & (0\:0\:0\:0\:3\:3\:1) & 0 & 0 & 1 & (3\:1) & 0 & 1 & 1 & (6) & 3 & 0 & 1\\
\rowcolor{lightgray}
0.15.0.5 & (3\:2\:2\:3\:0\:0\:0) & 1 & 0 & 1 & (5\:2) & 1 & 1 & 1 & (6) & 3 & 0 & 1\\
0.15.1.1 & (0\:1\:0\:0\:3\:0\:0) & 0 & 0 & 1 & (3\:1) & 0 & 1 & 1 & (0) & 0 & 1 & 4\\
\rowcolor{lightgray}
0.15.1.2 & (0\:0\:0\:0\:3\:3\:1) & 0 & 0 & 1 & (3\:1) & 0 & 1 & 1 & (1) & 1 & 1 & 4\\
0.15.1.3 & (0\:0\:1\:2\:2\:0\:0) & 0 & 0 & 1 & (1\:0) & 0 & 1 & 2 & (2) & 0 & 1 & 2\\
\hline\end{longtable}
}

{\small\centering\setlength{\tabcolsep}{3pt}
\begin{longtable}{ccccccccccccccccc}
\caption{Fixed loci of purely non-symplectic automorphisms of order 16}\\
 & \multicolumn{4}{c}{$j = 16$} & \multicolumn{4}{c}{$j = 8$} & \multicolumn{4}{c}{$j = 4$} & \multicolumn{4}{c}{$j = 2$}\\
\cmidrule(lr){2-5}
\cmidrule(lr){6-9}
\cmidrule(lr){10-13}
\cmidrule(lr){14-17}
K3 id & $(a_i)$ & $k$ & $l$ & $g$ & $(a_i)$ & $k$ & $l$ & $g$ & $(a_i)$ & $k$ & $l$ & $g$ & $(a_i)$ & $k$ & $l$ & $g$\\
\hline
\endhead
\rowcolor{lightgray}
0.16.0.1 & (3\:2\:1\:1\:1\:2\:2) & 1 & 0 & 1 & (7\:3\:2) & 2 & 0 & 1 & (10) & 3 & 0 & 1 & () & 6 & 1 & 3\\
0.16.0.2 & (0\:1\:0\:0\:0\:1\:2) & 0 & 0 & 1 & (3\:3\:4) & 1 & 0 & 1 & (10) & 3 & 0 & 1 & () & 5 & 1 & 2\\
\rowcolor{lightgray}
0.16.0.3 & (3\:2\:2\:2\:1\:0\:0) & 1 & 0 & 1 & (3\:3\:4) & 1 & 0 & 1 & (10) & 3 & 0 & 1 & () & 5 & 1 & 2\\
0.16.0.4 & (0\:0\:0\:2\:1\:1\:2) & 0 & 0 & 1 & (3\:3\:4) & 1 & 0 & 1 & (6) & 1 & 1 & 1 & () & 4 & 1 & 1\\
\rowcolor{lightgray}
0.16.0.5 & (3\:3\:2\:0\:0\:0\:0) & 1 & 0 & 1 & (3\:3\:4) & 1 & 0 & 1 & (6) & 1 & 1 & 1 & () & 4 & 1 & 1\\
0.16.1.1 & (4\:1\:0\:0\:0\:1\:0) & 1 & 0 & 1 & (5\:1\:0) & 1 & 0 & 1 & (6) & 1 & 0 & 1 & () & 2 & 1 & 7\\
\rowcolor{lightgray}
0.16.1.2 & (0\:1\:0\:0\:0\:1\:2) & 0 & 0 & 1 & (5\:1\:0) & 1 & 0 & 1 & (6) & 1 & 0 & 1 & () & 1 & 1 & 6\\
\hline\end{longtable}
}

{\small\centering\setlength{\tabcolsep}{3pt}
\begin{longtable}{ccccccccccccccccccccc}
\caption{Fixed loci of purely non-symplectic automorphisms of order 18}\\
 & \multicolumn{4}{c}{$j = 18$} & \multicolumn{4}{c}{$j = 9$} & \multicolumn{4}{c}{$j = 6$} & \multicolumn{4}{c}{$j = 3$} & \multicolumn{4}{c}{$j = 2$}\\
\cmidrule(lr){2-5}
\cmidrule(lr){6-9}
\cmidrule(lr){10-13}
\cmidrule(lr){14-17}
\cmidrule(lr){18-21}
K3 id & $(a_i)$ & $k$ & $l$ & $g$ & $(a_i)$ & $k$ & $l$ & $g$ & $(a_i)$ & $k$ & $l$ & $g$ & $(a_i)$ & $k$ & $l$ & $g$ & $(a_i)$ & $k$ & $l$ & $g$\\
\hline
\endhead
\rowcolor{lightgray}
0.18.0.1 & (0\:1\:0\:1\:0\:0\:0\:0) & 0 & 0 & 1 & (6\:5\:2\:1) & 2 & 0 & 1 & (1\:4) & 0 & 0 & 1 & (7) & 4 & 1 & 1 & () & 0 & 1 & 3\\
0.18.0.2 & (3\:2\:1\:1\:1\:1\:1\:2) & 1 & 0 & 1 & (6\:5\:2\:1) & 2 & 0 & 1 & (5\:8) & 2 & 0 & 1 & (7) & 4 & 1 & 1 & () & 6 & 1 & 3\\
\rowcolor{lightgray}
0.18.0.3 & (0\:0\:0\:1\:1\:1\:1\:2) & 0 & 0 & 1 & (3\:3\:3\:1) & 1 & 0 & 1 & (5\:2) & 1 & 0 & 1 & (7) & 4 & 0 & 1 & () & 2 & 1 & 3\\
0.18.0.4 & (3\:2\:2\:1\:1\:1\:0\:0) & 1 & 0 & 1 & (3\:3\:3\:1) & 1 & 0 & 1 & (7\:4) & 2 & 0 & 1 & (7) & 4 & 0 & 1 & () & 5 & 1 & 3\\
\rowcolor{lightgray}
0.18.0.5 & (0\:0\:1\:1\:1\:1\:0\:0) & 0 & 0 & 1 & (0\:1\:4\:1) & 0 & 0 & 1 & (4\:4) & 1 & 0 & 1 & (4) & 1 & 1 & 1 & () & 5 & 0 & 1\\
0.18.0.6 & (3\:3\:1\:1\:0\:0\:0\:0) & 1 & 0 & 1 & (3\:3\:3\:1) & 1 & 0 & 1 & (4\:4) & 1 & 0 & 1 & (4) & 2 & 1 & 2 & () & 5 & 0 & 1\\
\rowcolor{lightgray}
0.18.0.7 & (0\:0\:0\:0\:0\:3\:2\:2) & 0 & 0 & 1 & (4\:3\:0\:0) & 1 & 0 & 1 & (2\:2) & 0 & 1 & 1 & (4) & 1 & 1 & 1 & () & 2 & 1 & 1\\
0.18.1.1 & (3\:2\:2\:1\:1\:1\:0\:0) & 1 & 0 & 1 & (3\:3\:3\:1) & 1 & 0 & 1 & (4\:10) & 2 & 0 & 1 & (4) & 3 & 1 & 3 & () & 5 & 1 & 6\\
\hline\end{longtable}
}

{\small\centering\setlength{\tabcolsep}{3pt}
\begin{longtable}{ccccccccccccccccccccc}
\caption{Fixed loci of purely non-symplectic automorphisms of order 20}\\
 & \multicolumn{4}{c}{$j = 20$} & \multicolumn{4}{c}{$j = 10$} & \multicolumn{4}{c}{$j = 5$} & \multicolumn{4}{c}{$j = 4$} & \multicolumn{4}{c}{$j = 2$}\\
\cmidrule(lr){2-5}
\cmidrule(lr){6-9}
\cmidrule(lr){10-13}
\cmidrule(lr){14-17}
\cmidrule(lr){18-21}
K3 id & $(a_i)$ & $k$ & $l$ & $g$ & $(a_i)$ & $k$ & $l$ & $g$ & $(a_i)$ & $k$ & $l$ & $g$ & $(a_i)$ & $k$ & $l$ & $g$ & $(a_i)$ & $k$ & $l$ & $g$\\
\hline
\endhead
\rowcolor{lightgray}
0.20.0.1 & (0\:0\:0\:2\:0\:0\:0\:1\:2) & 0 & 0 & 1 & (3\:2\:2\:4) & 1 & 0 & 1 & (5\:2) & 1 & 1 & 1 & (2) & 0 & 1 & 2 & () & 5 & 1 & 2\\
0.20.0.2 & (3\:2\:2\:2\:0\:0\:0\:0\:0) & 1 & 0 & 1 & (3\:2\:2\:4) & 1 & 0 & 1 & (5\:2) & 1 & 1 & 1 & (10) & 3 & 0 & 1 & () & 5 & 1 & 2\\
\rowcolor{lightgray}
0.20.0.3 & (0\:0\:1\:1\:1\:0\:0\:0\:0) & 0 & 0 & 1 & (0\:2\:1\:2) & 0 & 0 & 1 & (5\:2) & 1 & 0 & 1 & (6) & 1 & 0 & 1 & () & 1 & 1 & 2\\
0.20.0.4 & (0\:0\:0\:1\:1\:1\:2\:0\:0) & 0 & 0 & 1 & (0\:2\:1\:2) & 0 & 0 & 1 & (5\:2) & 1 & 0 & 1 & (2) & 0 & 1 & 2 & () & 1 & 1 & 2\\
\rowcolor{lightgray}
0.20.0.5 & (0\:0\:1\:0\:0\:0\:0\:2\:2) & 0 & 0 & 1 & (4\:2\:1\:0) & 1 & 0 & 1 & (5\:2) & 1 & 0 & 1 & (2) & 0 & 1 & 2 & () & 3 & 1 & 2\\
0.20.0.6 & (0\:0\:0\:0\:0\:1\:2\:2\:2) & 0 & 0 & 1 & (4\:2\:1\:0) & 1 & 0 & 1 & (5\:2) & 1 & 0 & 1 & (8) & 2 & 0 & 1 & () & 3 & 1 & 2\\
\rowcolor{lightgray}
0.20.1.1 & (0\:0\:1\:1\:1\:0\:0\:0\:0) & 0 & 0 & 1 & (0\:0\:1\:6) & 0 & 0 & 1 & (1\:0) & 0 & 1 & 2 & (6) & 1 & 0 & 1 & () & 1 & 1 & 6\\
\hline\end{longtable}
}

{\small\centering\setlength{\tabcolsep}{3pt}
\begin{longtable}{ccccccccccccc}
\caption{Fixed loci of purely non-symplectic automorphisms of order 21}\\
 & \multicolumn{4}{c}{$j = 21$} & \multicolumn{4}{c}{$j = 7$} & \multicolumn{4}{c}{$j = 3$}\\
\cmidrule(lr){2-5}
\cmidrule(lr){6-9}
\cmidrule(lr){10-13}
K3 id & $(a_i)$ & $k$ & $l$ & $g$ & $(a_i)$ & $k$ & $l$ & $g$ & $(a_i)$ & $k$ & $l$ & $g$\\
\hline
\endhead
\rowcolor{lightgray}
0.21.0.1 & (3\:2\:1\:1\:1\:3\:0\:0\:0\:0) & 1 & 0 & 1 & (4\:3\:1) & 1 & 1 & 1 & (4) & 3 & 1 & 3\\
0.21.0.2 & (0\:0\:1\:0\:0\:1\:1\:1\:0\:0) & 0 & 0 & 1 & (4\:3\:1) & 1 & 0 & 1 & (1) & 0 & 1 & 3\\
\rowcolor{lightgray}
0.21.0.3 & (0\:0\:0\:0\:0\:1\:1\:1\:3\:1) & 0 & 0 & 1 & (4\:3\:1) & 1 & 0 & 1 & (2) & 1 & 1 & 3\\
0.21.0.4 & (3\:3\:1\:0\:0\:0\:0\:1\:0\:0) & 1 & 0 & 1 & (4\:3\:1) & 1 & 0 & 1 & (3) & 2 & 1 & 3\\
\hline\end{longtable}
}

{\small\centering\setlength{\tabcolsep}{3pt}
\begin{longtable}{ccccccccccccc}
\caption{Fixed loci of purely non-symplectic automorphisms of order 22}\\
 & \multicolumn{4}{c}{$j = 22$} & \multicolumn{4}{c}{$j = 11$} & \multicolumn{4}{c}{$j = 2$}\\
\cmidrule(lr){2-5}
\cmidrule(lr){6-9}
\cmidrule(lr){10-13}
K3 id & $(a_i)$ & $k$ & $l$ & $g$ & $(a_i)$ & $k$ & $l$ & $g$ & $(a_i)$ & $k$ & $l$ & $g$\\
\hline
\endhead
\rowcolor{lightgray}
0.22.0.1 & (0\:0\:0\:1\:1\:0\:0\:0\:1\:2) & 0 & 0 & 1 & (3\:2\:2\:3\:1) & 1 & 0 & 1 & () & 1 & 1 & 5\\
0.22.0.2 & (3\:2\:2\:1\:1\:0\:0\:0\:0\:0) & 1 & 0 & 1 & (3\:2\:2\:3\:1) & 1 & 0 & 1 & () & 4 & 1 & 5\\
\rowcolor{lightgray}
0.22.0.3 & (3\:2\:1\:1\:1\:2\:1\:0\:0\:0) & 1 & 0 & 1 & (3\:2\:2\:3\:1) & 1 & 0 & 1 & () & 5 & 1 & 5\\
0.22.1.1 & (0\:0\:0\:1\:0\:0\:0\:0\:1\:4) & 0 & 0 & 1 & (1\:0\:0\:1\:0) & 0 & 1 & 1 & () & 1 & 1 & 10\\
\hline\end{longtable}
}

{\small\centering\setlength{\tabcolsep}{3pt}
\begin{longtable}{ccccccccccccc}
\caption{Fixed loci of purely non-symplectic automorphisms of order 24}\\
 & \multicolumn{4}{c}{$j = 24$} & \multicolumn{4}{c}{$j = 12$} & \multicolumn{4}{c}{$j = 8$}\\
\cmidrule(lr){2-5}
\cmidrule(lr){6-9}
\cmidrule(lr){10-13}
K3 id & $(a_i)$ & $k$ & $l$ & $g$ & $(a_i)$ & $k$ & $l$ & $g$ & $(a_i)$ & $k$ & $l$ & $g$\\
\hline
\endhead
\rowcolor{lightgray}
0.24.0.1 & (3\:2\:1\:1\:1\:1\:1\:1\:0\:0\:0) & 1 & 0 & 1 & (3\:2\:2\:2\:2) & 1 & 0 & 1 & (7\:3\:2) & 2 & 0 & 1\\
0.24.0.2 & (0\:0\:1\:0\:0\:0\:1\:1\:0\:0\:0) & 0 & 0 & 1 & (0\:0\:4\:1\:0) & 0 & 0 & 1 & (3\:3\:4) & 1 & 0 & 1\\
\rowcolor{lightgray}
0.24.0.3 & (0\:0\:0\:0\:0\:0\:0\:3\:1\:1\:2) & 0 & 0 & 1 & (3\:3\:3\:0\:0) & 1 & 0 & 1 & (1\:1\:2) & 0 & 1 & 1\\
0.24.0.4 & (3\:3\:1\:0\:0\:0\:0\:0\:0\:0\:0) & 1 & 0 & 1 & (3\:3\:3\:0\:0) & 1 & 0 & 1 & (3\:3\:4) & 1 & 0 & 1\\
\rowcolor{lightgray}
0.24.0.5 & (0\:0\:0\:0\:1\:1\:1\:1\:1\:0\:0) & 0 & 0 & 1 & (0\:1\:1\:1\:2) & 0 & 0 & 1 & (5\:1\:0) & 1 & 0 & 1\\
0.24.0.6 & (0\:0\:0\:0\:0\:0\:1\:1\:1\:2\:2) & 0 & 0 & 1 & (4\:1\:1\:1\:0) & 1 & 0 & 1 & (4\:2\:2) & 1 & 0 & 1\\
\rowcolor{lightgray}
0.24.0.7 & (0\:0\:0\:0\:0\:2\:3\:0\:0\:0\:0) & 0 & 0 & 1 & (0\:0\:0\:3\:4) & 0 & 0 & 1 & (2\:0\:0) & 0 & 1 & 1\\
0.24.0.8 & (0\:0\:0\:1\:2\:0\:0\:0\:0\:0\:0) & 0 & 0 & 1 & (0\:0\:0\:3\:4) & 0 & 0 & 1 & (0\:2\:4) & 0 & 0 & 1\\
\rowcolor{lightgray}
0.24.1.1 & (0\:0\:0\:0\:1\:1\:1\:1\:1\:0\:0) & 0 & 0 & 1 & (0\:1\:1\:1\:2) & 0 & 0 & 1 & (5\:1\:0) & 1 & 0 & 1\\
\hline\end{longtable}
}

{\small\centering\setlength{\tabcolsep}{3pt}
\begin{longtable}{ccccccccccccccccc}
 & \multicolumn{4}{c}{$j = 6$} & \multicolumn{4}{c}{$j = 4$} & \multicolumn{4}{c}{$j = 3$} & \multicolumn{4}{c}{$j = 2$}\\
\cmidrule(lr){2-5}
\cmidrule(lr){6-9}
\cmidrule(lr){10-13}
\cmidrule(lr){14-17}
K3 id & $(a_i)$ & $k$ & $l$ & $g$ & $(a_i)$ & $k$ & $l$ & $g$ & $(a_i)$ & $k$ & $l$ & $g$ & $(a_i)$ & $k$ & $l$ & $g$\\
\hline
\endhead
\rowcolor{lightgray}
0.24.0.1 & (5\:8) & 2 & 0 & 1 & (10) & 3 & 0 & 1 & (5) & 3 & 1 & 2 & () & 6 & 1 & 3\\
0.24.0.2 & (1\:4) & 0 & 0 & 1 & (6) & 1 & 1 & 1 & (1) & 0 & 1 & 3 & () & 4 & 1 & 1\\
\rowcolor{lightgray}
0.24.0.3 & (3\:6) & 1 & 0 & 1 & (6) & 1 & 1 & 1 & (3) & 2 & 1 & 3 & () & 4 & 1 & 1\\
0.24.0.4 & (3\:6) & 1 & 0 & 1 & (6) & 1 & 1 & 1 & (3) & 2 & 1 & 3 & () & 4 & 1 & 1\\
\rowcolor{lightgray}
0.24.0.5 & (5\:2) & 1 & 0 & 1 & (6) & 1 & 0 & 1 & (5) & 2 & 0 & 1 & () & 2 & 1 & 3\\
0.24.0.6 & (5\:2) & 1 & 0 & 1 & (8) & 2 & 0 & 1 & (5) & 2 & 0 & 1 & () & 3 & 1 & 4\\
\rowcolor{lightgray}
0.24.0.7 & (3\:0) & 0 & 1 & 1 & (4) & 0 & 1 & 1 & (3) & 0 & 1 & 1 & () & 0 & 2 & 1\\
0.24.0.8 & (3\:0) & 0 & 1 & 1 & (4) & 0 & 1 & 1 & (3) & 0 & 1 & 1 & () & 0 & 2 & 1\\
\rowcolor{lightgray}
0.24.1.1 & (1\:10) & 1 & 0 & 1 & (6) & 1 & 0 & 1 & (1) & 1 & 1 & 4 & () & 2 & 1 & 7\\
\hline\end{longtable}
}

{\small\centering\setlength{\tabcolsep}{3pt}
\begin{longtable}{ccccccccccccc}
\caption{Fixed loci of purely non-symplectic automorphisms of order 26}\\
 & \multicolumn{4}{c}{$j = 26$} & \multicolumn{4}{c}{$j = 13$} & \multicolumn{4}{c}{$j = 2$}\\
\cmidrule(lr){2-5}
\cmidrule(lr){6-9}
\cmidrule(lr){10-13}
K3 id & $(a_i)$ & $k$ & $l$ & $g$ & $(a_i)$ & $k$ & $l$ & $g$ & $(a_i)$ & $k$ & $l$ & $g$\\
\hline
\endhead
\rowcolor{lightgray}
0.26.0.1 & (0\:0\:0\:1\:0\:0\:0\:0\:0\:1\:1\:2) & 0 & 0 & 1 & (3\:3\:2\:1\:0\:0) & 1 & 0 & 1 & () & 1 & 1 & 6\\
0.26.0.2 & (0\:0\:0\:0\:0\:0\:0\:1\:2\:1\:1\:2) & 0 & 0 & 1 & (3\:3\:2\:1\:0\:0) & 1 & 0 & 1 & () & 2 & 1 & 6\\
\rowcolor{lightgray}
0.26.0.3 & (3\:2\:2\:1\:0\:0\:0\:0\:0\:1\:0\:0) & 1 & 0 & 1 & (3\:3\:2\:1\:0\:0) & 1 & 0 & 1 & () & 4 & 1 & 6\\
\hline\end{longtable}
}

{\small\centering\setlength{\tabcolsep}{3pt}
\begin{longtable}{ccccccccccccc}
\caption{Fixed loci of purely non-symplectic automorphisms of order 27}\\
 & \multicolumn{4}{c}{$j = 27$} & \multicolumn{4}{c}{$j = 9$} & \multicolumn{4}{c}{$j = 3$}\\
\cmidrule(lr){2-5}
\cmidrule(lr){6-9}
\cmidrule(lr){10-13}
K3 id & $(a_i)$ & $k$ & $l$ & $g$ & $(a_i)$ & $k$ & $l$ & $g$ & $(a_i)$ & $k$ & $l$ & $g$\\
\hline
\endhead
\rowcolor{lightgray}
0.27.0.1 & (0\:0\:0\:0\:1\:0\:0\:0\:0\:0\:1\:3\:1) & 0 & 0 & 1 & (0\:1\:4\:1) & 0 & 0 & 1 & (1) & 1 & 1 & 4\\
0.27.0.2 & (0\:0\:0\:1\:1\:0\:0\:0\:0\:0\:1\:0\:0) & 0 & 0 & 1 & (0\:1\:4\:1) & 0 & 0 & 1 & (1) & 0 & 1 & 3\\
\hline\end{longtable}
}

{\small\centering\setlength{\tabcolsep}{3pt}
\begin{longtable}{ccccccccccccccccccccc}
\caption{Fixed loci of purely non-symplectic automorphisms of order 28}\\
 & \multicolumn{4}{c}{$j = 28$} & \multicolumn{4}{c}{$j = 14$} & \multicolumn{4}{c}{$j = 7$} & \multicolumn{4}{c}{$j = 4$} & \multicolumn{4}{c}{$j = 2$}\\
\cmidrule(lr){2-5}
\cmidrule(lr){6-9}
\cmidrule(lr){10-13}
\cmidrule(lr){14-17}
\cmidrule(lr){18-21}
K3 id & $(a_i)$ & $k$ & $l$ & $g$ & $(a_i)$ & $k$ & $l$ & $g$ & $(a_i)$ & $k$ & $l$ & $g$ & $(a_i)$ & $k$ & $l$ & $g$ & $(a_i)$ & $k$ & $l$ & $g$\\
\hline
\endhead
\rowcolor{lightgray}
0.28.0.1 & (3\:2\:1\:1\:1\:2\:0\:0\:0\:0\:0\:0\:0) & 1 & 0 & 1 & (3\:2\:1\:1\:1\:4) & 1 & 0 & 1 & (4\:3\:1) & 1 & 1 & 1 & (8) & 2 & 0 & 1 & () & 5 & 1 & 6\\
0.28.0.2 & (0\:0\:0\:1\:0\:0\:2\:0\:0\:0\:0\:0\:0) & 0 & 0 & 1 & (0\:0\:0\:1\:2\:4) & 0 & 0 & 1 & (2\:1\:0) & 0 & 1 & 1 & (0) & 0 & 1 & 3 & () & 2 & 1 & 3\\
\rowcolor{lightgray}
0.28.0.3 & (0\:0\:0\:0\:0\:0\:2\:2\:1\:0\:0\:0\:0) & 0 & 0 & 1 & (0\:0\:0\:1\:2\:4) & 0 & 0 & 1 & (2\:1\:0) & 0 & 1 & 1 & (8) & 2 & 0 & 1 & () & 2 & 1 & 3\\
\hline\end{longtable}
}

{\small\centering\setlength{\tabcolsep}{3pt}
\begin{longtable}{ccccccccccccc}
\caption{Fixed loci of purely non-symplectic automorphisms of order 30}\\
 & \multicolumn{4}{c}{$j = 30$} & \multicolumn{4}{c}{$j = 15$} & \multicolumn{4}{c}{$j = 10$}\\
\cmidrule(lr){2-5}
\cmidrule(lr){6-9}
\cmidrule(lr){10-13}
K3 id & $(a_i)$ & $k$ & $l$ & $g$ & $(a_i)$ & $k$ & $l$ & $g$ & $(a_i)$ & $k$ & $l$ & $g$\\
\hline
\endhead
\rowcolor{lightgray}
0.30.0.1 & (0\:0\:1\:0\:0\:0\:0\:0\:0\:0\:0\:0\:0\:0) & 0 & 0 & 1 & (0\:0\:1\:2\:2\:0\:0) & 0 & 0 & 1 & (0\:3\:1\:0) & 0 & 0 & 1\\
0.30.0.2 & (0\:0\:0\:0\:0\:0\:0\:0\:1\:1\:1\:1\:1\:2) & 0 & 0 & 1 & (3\:3\:1\:1\:1\:0\:0) & 1 & 0 & 1 & (4\:1\:1\:2) & 1 & 0 & 1\\
\rowcolor{lightgray}
0.30.0.3 & (3\:2\:1\:1\:1\:1\:1\:0\:0\:0\:0\:0\:0\:0) & 1 & 0 & 1 & (3\:2\:1\:1\:1\:3\:1) & 1 & 0 & 1 & (3\:3\:2\:2) & 1 & 0 & 1\\
0.30.0.4 & (0\:0\:0\:0\:1\:1\:1\:0\:0\:0\:0\:0\:0\:0) & 0 & 0 & 1 & (0\:0\:0\:0\:3\:3\:1) & 0 & 0 & 1 & (0\:1\:1\:4) & 0 & 0 & 1\\
\rowcolor{lightgray}
0.30.0.5 & (0\:0\:0\:1\:0\:0\:0\:0\:0\:0\:0\:0\:1\:2) & 0 & 0 & 1 & (3\:2\:2\:3\:0\:0\:0) & 1 & 0 & 1 & (0\:0\:1\:6) & 0 & 0 & 1\\
0.30.0.6 & (3\:2\:2\:1\:0\:0\:0\:0\:0\:0\:0\:0\:0\:0) & 1 & 0 & 1 & (3\:2\:2\:3\:0\:0\:0) & 1 & 0 & 1 & (3\:2\:2\:4) & 1 & 0 & 1\\
\rowcolor{lightgray}
0.30.0.7 & (0\:0\:0\:0\:0\:0\:1\:1\:3\:0\:0\:0\:0\:0) & 0 & 0 & 1 & (0\:0\:0\:0\:3\:3\:1) & 0 & 0 & 1 & (1\:1\:0\:0) & 0 & 1 & 1\\
0.30.0.8 & (0\:0\:0\:0\:2\:0\:0\:0\:0\:0\:1\:0\:0\:0) & 0 & 0 & 1 & (0\:0\:1\:2\:2\:0\:0) & 0 & 0 & 1 & (1\:0\:0\:2) & 0 & 0 & 1\\
\rowcolor{lightgray}
0.30.1.1 & (0\:0\:0\:0\:1\:1\:1\:0\:0\:0\:0\:0\:0\:0) & 0 & 0 & 1 & (0\:0\:0\:0\:3\:3\:1) & 0 & 0 & 1 & (0\:1\:1\:4) & 0 & 0 & 1\\
0.30.1.2 & (0\:0\:1\:0\:0\:0\:0\:0\:0\:0\:0\:0\:0\:0) & 0 & 0 & 1 & (0\:0\:1\:2\:2\:0\:0) & 0 & 0 & 1 & (0\:0\:1\:6) & 0 & 0 & 1\\
\hline\end{longtable}
}
\vspace{-1em}
{\small\centering\setlength{\tabcolsep}{3pt}
\begin{longtable}{ccccccccccccccccc}
 & \multicolumn{4}{c}{$j = 6$} & \multicolumn{4}{c}{$j = 5$} & \multicolumn{4}{c}{$j = 3$} & \multicolumn{4}{c}{$j = 2$}\\
\cmidrule(lr){2-5}
\cmidrule(lr){6-9}
\cmidrule(lr){10-13}
\cmidrule(lr){14-17}
K3 id & $(a_i)$ & $k$ & $l$ & $g$ & $(a_i)$ & $k$ & $l$ & $g$ & $(a_i)$ & $k$ & $l$ & $g$ & $(a_i)$ & $k$ & $l$ & $g$\\
\hline
\endhead
\rowcolor{lightgray}
0.30.0.1 & (0\:6) & 0 & 0 & 1 & (7\:3) & 2 & 0 & 1 & (2) & 0 & 1 & 2 & () & 0 & 1 & 4\\
0.30.0.2 & (2\:8) & 1 & 0 & 1 & (7\:3) & 2 & 0 & 1 & (4) & 2 & 1 & 2 & () & 3 & 1 & 4\\
\rowcolor{lightgray}
0.30.0.3 & (5\:8) & 2 & 0 & 1 & (7\:3) & 2 & 0 & 1 & (5) & 3 & 1 & 2 & () & 5 & 1 & 5\\
0.30.0.4 & (6\:0) & 1 & 0 & 1 & (3\:1) & 0 & 1 & 1 & (6) & 3 & 0 & 1 & () & 1 & 1 & 4\\
\rowcolor{lightgray}
0.30.0.5 & (2\:2) & 0 & 0 & 1 & (5\:2) & 1 & 1 & 1 & (6) & 3 & 0 & 1 & () & 1 & 1 & 6\\
0.30.0.6 & (4\:4) & 1 & 0 & 1 & (5\:2) & 1 & 1 & 1 & (6) & 3 & 0 & 1 & () & 4 & 1 & 6\\
\rowcolor{lightgray}
0.30.0.7 & (1\:4) & 0 & 0 & 1 & (3\:1) & 0 & 1 & 1 & (1) & 1 & 1 & 4 & () & 1 & 1 & 1\\
0.30.0.8 & (0\:0) & 0 & 1 & 2 & (1\:0) & 0 & 1 & 2 & (2) & 0 & 1 & 2 & () & 0 & 1 & 2\\
\rowcolor{lightgray}
0.30.1.1 & (1\:10) & 1 & 0 & 1 & (3\:1) & 0 & 1 & 1 & (1) & 1 & 1 & 4 & () & 1 & 1 & 9\\
0.30.1.2 & (0\:6) & 0 & 0 & 1 & (1\:0) & 0 & 1 & 2 & (2) & 0 & 1 & 2 & () & 0 & 1 & 10\\
\hline\end{longtable}
}

{\small\centering\setlength{\tabcolsep}{3pt}
\begin{longtable}{ccccccccc}
\caption{Fixed loci of purely non-symplectic automorphisms of order 32}\\
 & \multicolumn{4}{c}{$j = 32$} & \multicolumn{4}{c}{$j = 16$}\\
\cmidrule(lr){2-5}
\cmidrule(lr){6-9}
K3 id & $(a_i)$ & $k$ & $l$ & $g$ & $(a_i)$ & $k$ & $l$ & $g$\\
\hline
\endhead
\rowcolor{lightgray}
0.32.0.1 & (0\:0\:0\:0\:0\:1\:0\:0\:0\:0\:0\:0\:1\:2\:2) & 0 & 0 & 1 & (4\:1\:0\:0\:0\:1\:0) & 1 & 0 & 1\\
0.32.0.2 & (0\:0\:0\:0\:0\:1\:1\:1\:0\:0\:0\:0\:1\:0\:0) & 0 & 0 & 1 & (0\:1\:0\:0\:0\:1\:2) & 0 & 0 & 1\\
\hline\end{longtable}
}

{\small\centering\setlength{\tabcolsep}{3pt}
\begin{longtable}{ccccccccccccc}
 & \multicolumn{4}{c}{$j = 8$} & \multicolumn{4}{c}{$j = 4$} & \multicolumn{4}{c}{$j = 2$}\\
\cmidrule(lr){2-5}
\cmidrule(lr){6-9}
\cmidrule(lr){10-13}
K3 id & $(a_i)$ & $k$ & $l$ & $g$ & $(a_i)$ & $k$ & $l$ & $g$ & $(a_i)$ & $k$ & $l$ & $g$\\
\hline
\endhead
\rowcolor{lightgray}
0.32.0.1 & (5\:1\:0) & 1 & 0 & 1 & (6) & 1 & 0 & 1 & () & 2 & 1 & 7\\
0.32.0.2 & (5\:1\:0) & 1 & 0 & 1 & (6) & 1 & 0 & 1 & () & 1 & 1 & 6\\
\hline\end{longtable}
}

{\small\centering\setlength{\tabcolsep}{3pt}
\begin{longtable}{ccccccccccccc}
\caption{Fixed loci of purely non-symplectic automorphisms of order 34}\\
 & \multicolumn{4}{c}{$j = 34$} & \multicolumn{4}{c}{$j = 17$} & \multicolumn{4}{c}{$j = 2$}\\
\cmidrule(lr){2-5}
\cmidrule(lr){6-9}
\cmidrule(lr){10-13}
K3 id & $(a_i)$ & $k$ & $l$ & $g$ & $(a_i)$ & $k$ & $l$ & $g$ & $(a_i)$ & $k$ & $l$ & $g$\\
\hline
\endhead
\rowcolor{lightgray}
0.34.0.1 & (0\:0\:0\:0\:1\:0\:0\:1\:1\:0\:0\:0\:0\:0\:0\:0) & 0 & 0 & 1 & (0\:0\:0\:0\:1\:2\:3\:1) & 0 & 0 & 1 & () & 0 & 1 & 8\\
0.34.0.2 & (0\:0\:0\:0\:0\:0\:0\:1\:1\:2\:1\:0\:0\:0\:0\:0) & 0 & 0 & 1 & (0\:0\:0\:0\:1\:2\:3\:1) & 0 & 0 & 1 & () & 1 & 1 & 8\\
\hline\end{longtable}
}

{\small\centering\setlength{\tabcolsep}{3pt}
\begin{longtable}{ccccccccccccc}
\caption{Fixed loci of purely non-symplectic automorphisms of order 36}\\
 & \multicolumn{4}{c}{$j = 36$} & \multicolumn{4}{c}{$j = 18$} & \multicolumn{4}{c}{$j = 12$}\\
\cmidrule(lr){2-5}
\cmidrule(lr){6-9}
\cmidrule(lr){10-13}
K3 id & $(a_i)$ & $k$ & $l$ & $g$ & $(a_i)$ & $k$ & $l$ & $g$ & $(a_i)$ & $k$ & $l$ & $g$\\
\hline
\endhead
\rowcolor{lightgray}
0.36.0.1 & (3\:2\:1\:1\:1\:1\:0\:0\:0\:0\:0\:0\:0\:1\:0\:0\:0) & 1 & 0 & 1 & (3\:2\:2\:1\:1\:1\:0\:0) & 1 & 0 & 1 & (3\:3\:1\:1\:2) & 1 & 0 & 1\\
0.36.0.2 & (0\:0\:0\:0\:0\:0\:0\:0\:2\:1\:1\:1\:0\:0\:0\:0\:0) & 0 & 0 & 1 & (0\:0\:0\:0\:1\:1\:1\:4) & 0 & 0 & 1 & (4\:2\:0\:0\:0) & 1 & 0 & 1\\
\rowcolor{lightgray}
0.36.0.3 & (0\:0\:0\:1\:0\:0\:0\:0\:0\:0\:0\:0\:0\:0\:1\:0\:0) & 0 & 0 & 1 & (0\:1\:0\:1\:0\:0\:0\:0) & 0 & 0 & 1 & (0\:0\:4\:1\:0) & 0 & 0 & 1\\
\hline\end{longtable}
}

{\small\centering\setlength{\tabcolsep}{3pt}
\begin{longtable}{ccccccccccccccccccccc}
 & \multicolumn{4}{c}{$j = 9$} & \multicolumn{4}{c}{$j = 6$} & \multicolumn{4}{c}{$j = 4$} & \multicolumn{4}{c}{$j = 3$} & \multicolumn{4}{c}{$j = 2$}\\
\cmidrule(lr){2-5}
\cmidrule(lr){6-9}
\cmidrule(lr){10-13}
\cmidrule(lr){14-17}
\cmidrule(lr){18-21}
K3 id & $(a_i)$ & $k$ & $l$ & $g$ & $(a_i)$ & $k$ & $l$ & $g$ & $(a_i)$ & $k$ & $l$ & $g$ & $(a_i)$ & $k$ & $l$ & $g$ & $(a_i)$ & $k$ & $l$ & $g$\\
\hline
\endhead
\rowcolor{lightgray}
0.36.0.1 & (3\:3\:3\:1) & 1 & 0 & 1 & (4\:10) & 2 & 0 & 1 & (8) & 2 & 0 & 1 & (4) & 3 & 1 & 3 & () & 5 & 1 & 6\\
0.36.0.2 & (1\:1\:1\:0) & 0 & 1 & 1 & (4\:4) & 1 & 0 & 1 & (6) & 1 & 0 & 1 & (4) & 1 & 1 & 1 & () & 2 & 1 & 7\\
\rowcolor{lightgray}
0.36.0.3 & (0\:1\:4\:1) & 0 & 0 & 1 & (1\:4) & 0 & 0 & 1 & (0) & 0 & 1 & 3 & (1) & 0 & 1 & 3 & () & 0 & 1 & 3\\
\hline\end{longtable}
}

{\small\centering\setlength{\tabcolsep}{3pt}
\begin{longtable}{ccccccccccccc}
\caption{Fixed loci of purely non-symplectic automorphisms of order 38}\\
 & \multicolumn{4}{c}{$j = 38$} & \multicolumn{4}{c}{$j = 19$} & \multicolumn{4}{c}{$j = 2$}\\
\cmidrule(lr){2-5}
\cmidrule(lr){6-9}
\cmidrule(lr){10-13}
K3 id & $(a_i)$ & $k$ & $l$ & $g$ & $(a_i)$ & $k$ & $l$ & $g$ & $(a_i)$ & $k$ & $l$ & $g$\\
\hline
\endhead
\rowcolor{lightgray}
0.38.0.1 & (0\:0\:0\:0\:0\:1\:1\:0\:0\:0\:0\:0\:0\:1\:0\:0\:0\:0) & 0 & 0 & 1 & (0\:0\:0\:1\:2\:1\:1\:0\:0) & 0 & 0 & 1 & () & 0 & 1 & 9\\
0.38.0.2 & (0\:0\:0\:0\:0\:0\:1\:0\:0\:0\:0\:1\:2\:1\:0\:0\:0\:0) & 0 & 0 & 1 & (0\:0\:0\:1\:2\:1\:1\:0\:0) & 0 & 0 & 1 & () & 1 & 1 & 9\\
\hline\end{longtable}
}

{\small\centering\setlength{\tabcolsep}{3pt}
\begin{longtable}{ccccccccc}
\caption{Fixed loci of purely non-symplectic automorphisms of order 42}\\
 & \multicolumn{4}{c}{$j = 42$} & \multicolumn{4}{c}{$j = 21$}\\
\cmidrule(lr){2-5}
\cmidrule(lr){6-9}
K3 id & $(a_i)$ & $k$ & $l$ & $g$ & $(a_i)$ & $k$ & $l$ & $g$\\
\hline
\endhead
\rowcolor{lightgray}
0.42.0.1 & (3\:2\:1\:1\:1\:1\:0\:0\:0\:0\:0\:0\:0\:0\:0\:0\:0\:0\:0\:0) & 1 & 0 & 1 & (3\:2\:1\:1\:1\:3\:0\:0\:0\:0) & 1 & 0 & 1\\
0.42.0.2 & (0\:0\:0\:0\:0\:0\:0\:0\:0\:1\:1\:1\:1\:1\:0\:0\:0\:0\:0\:0) & 0 & 0 & 1 & (0\:0\:0\:0\:0\:1\:1\:1\:3\:1) & 0 & 0 & 1\\
\rowcolor{lightgray}
0.42.0.3 & (0\:0\:0\:0\:0\:0\:0\:1\:0\:0\:0\:0\:0\:0\:0\:0\:1\:1\:1\:2) & 0 & 0 & 1 & (3\:3\:1\:0\:0\:0\:0\:1\:0\:0) & 1 & 0 & 1\\
\hline\end{longtable}
}

{\small\centering\setlength{\tabcolsep}{3pt}
\begin{longtable}{ccccccccccccccccccccc}
 & \multicolumn{4}{c}{$j = 14$} & \multicolumn{4}{c}{$j = 7$} & \multicolumn{4}{c}{$j = 6$} & \multicolumn{4}{c}{$j = 3$} & \multicolumn{4}{c}{$j = 2$}\\
\cmidrule(lr){2-5}
\cmidrule(lr){6-9}
\cmidrule(lr){10-13}
\cmidrule(lr){14-17}
\cmidrule(lr){18-21}
K3 id & $(a_i)$ & $k$ & $l$ & $g$ & $(a_i)$ & $k$ & $l$ & $g$ & $(a_i)$ & $k$ & $l$ & $g$ & $(a_i)$ & $k$ & $l$ & $g$ & $(a_i)$ & $k$ & $l$ & $g$\\
\hline
\endhead
\rowcolor{lightgray}
0.42.0.1 & (3\:2\:1\:1\:1\:4) & 1 & 0 & 1 & (4\:3\:1) & 1 & 1 & 1 & (4\:10) & 2 & 0 & 1 & (4) & 3 & 1 & 3 & () & 5 & 1 & 6\\
0.42.0.2 & (4\:1\:1\:0\:0\:0) & 1 & 0 & 1 & (4\:3\:1) & 1 & 0 & 1 & (2\:8) & 1 & 0 & 1 & (2) & 1 & 1 & 3 & () & 2 & 1 & 6\\
\rowcolor{lightgray}
0.42.0.3 & (0\:0\:1\:1\:2\:2) & 0 & 0 & 1 & (4\:3\:1) & 1 & 0 & 1 & (1\:10) & 1 & 0 & 1 & (3) & 2 & 1 & 3 & () & 2 & 1 & 7\\
\hline\end{longtable}
}

{\small\centering\setlength{\tabcolsep}{3pt}
\begin{longtable}{cccrccc}
\caption{Fixed loci of purely non-symplectic automorphisms of orders 13, 17, 19, 25, 33, 40 and 44}\\
K3 id &$j$ & $(a_i)$ & $k$ & $l$ & $g$\\
\hline
\rowcolor{lightgray}
0.13.0.1 &$13$& (3\:3\:2\:1\:0\:0) & 1 & 0 & 1\\
0.17.0.1 &$17$& (0\:0\:0\:0\:1\:2\:3\:1) & 0 & 0 & 1\\
\rowcolor{lightgray}
0.19.0.1 &$19$& (0\:0\:0\:1\:2\:1\:1\:0\:0) & 0 & 0 & 1\\
& $25$ & (0\:0\:0\:0\:2\:1\:0\:0\:0\:1\:0\:0) & 0 & 0 & 1 \\
\multirow{-2}{1.2cm}{0.25.0.1}& $5$& (1\:0) & 0 & 1 & 2\\
\rowcolor{lightgray}
&$33$& (0\:0\:0\:0\:0\:1\:0\:0\:0\:0\:3\:1\:0\:0\:0\:0) & 0 & 0 & 1 \\\rowcolor{lightgray}
& $11$& (1\:0\:0\:1\:0) & 0 & 1 & 1 \\
\rowcolor{lightgray}
\multirow{-3}{1.2cm}{0.33.0.1} & $3$& (0) & 1 & 1 & 5\\
& $40$ & (0\:0\:0\:0\:1\:0\:0\:0\:0\:0\:0\:0\:0\:0\:1\:1\:0\:0\:0) & 0 & 0 & 1\\
& $20$ & (0\:0\:1\:1\:1\:0\:0\:0\:0) & 0 & 0 & 1\\
& $10$ & (0\:0\:1\:6) & 0 & 0 & 1 \\
& $8$  & (5\:1\:0) & 1 & 0 & 1 \\
& $5$  & (1\:0) & 0 & 1 & 2 \\
& $4$  & (6) & 1 & 0 & 1 \\
\multirow{-6}{1.2cm}{0.40.0.1}& $2$  & () & 1 & 1 & 6\\
\rowcolor{lightgray}
& $44$ & (0\:0\:0\:0\:0\:0\:0\:0\:1\:2\:0\:0\:0\:0\:0\:0\:1\:0\:0\:0\:0) & 0 & 0 & 1 \\
\rowcolor{lightgray}
& $22$& (0\:0\:0\:1\:0\:0\:0\:0\:1\:4) & 0 & 0 & 1\\
\rowcolor{lightgray}
& $11$& (1\:0\:0\:1\:0) & 0 & 1 & 1 \\
\rowcolor{lightgray}
& $4$& (4) & 0 & 0 & 1 \\
\rowcolor{lightgray}
\multirow{-5}{1.2cm}{0.44.0.1}& $2$& () & 1 & 1 & 10\\
\hline
\end{longtable}}
\newpage
{\small\centering\setlength{\tabcolsep}{3pt}
\begin{longtable}{cccrccc}
\caption{Fixed loci of purely non-symplectic automorphisms of orders 48, 50, 54 and 66}\\
K3 id &$j$ & $(a_i)$ & $k$ & $l$ & $g$\\
\hline
\rowcolor{lightgray}
& $48$& (0\:0\:0\:0\:0\:0\:0\:0\:1\:0\:0\:0\:0\:0\:1\:1\:1\:1\:0\:0\:0\:0\:0) & 0 & 0 & 1\\
\rowcolor{lightgray}
& $24$& (0\:0\:0\:0\:1\:1\:1\:1\:1\:0\:0) & 0 & 0 & 1\\
\rowcolor{lightgray}
& $16$& (4\:1\:0\:0\:0\:1\:0) & 1 & 0 & 1 \\
\rowcolor{lightgray}
& $12$&(0\:1\:1\:1\:2) & 0 & 0 & 1\\
\rowcolor{lightgray}
& $8$& (5\:1\:0) & 1 & 0 & 1\\
\rowcolor{lightgray}
& $6$& (1\:10) & 1 & 0 & 1\\
\rowcolor{lightgray}
& $4$& (6) & 1 & 0 & 1 \\
\rowcolor{lightgray}
& $3$ & (1) & 1 & 1 & 4 \\
\rowcolor{lightgray}
\multirow{-8}{1.2cm}{0.48.0.1}&2 & () & 2 & 1 & 7\\
& $50$ &(0\:0\:0\:0\:0\:1\:0\:0\:0\:0\:0\:0\:0\:1\:0\:0\:0\:0\:0\:0\:0\:0\:0\:0) & 0 & 0 & 1 \\
& $25$& (0\:0\:0\:0\:2\:1\:0\:0\:0\:1\:0\:0) & 0 & 0 & 1\\
&$10$& (0\:0\:1\:6) & 0 & 0 & 1 \\
& $5$& (1\:0) & 0 & 1 & 2 \\
\multirow{-5}{1.2cm}{0.50.0.1}&$2$ & () & 0 & 1 & 10\\
\rowcolor{lightgray}
&$54$ & (0\:0\:0\:0\:0\:0\:0\:0\:0\:0\:1\:1\:1\:0\:0\:0\:0\:0\:0\:0\:1\:0\:0\:0\:0\:0) & 0 & 0 & 1\\
\rowcolor{lightgray}
&$27$ & (0\:0\:0\:0\:1\:0\:0\:0\:0\:0\:1\:3\:1) & 0 & 0 & 1\\
\rowcolor{lightgray}
&$18$ & (0\:0\:1\:1\:1\:1\:0\:0) & 0 & 0 & 1 \\
\rowcolor{lightgray}
&$9$ & (0\:1\:4\:1) & 0 & 0 & 1 \\
\rowcolor{lightgray}
&$6$ & (1\:10) & 1 & 0 & 1 \\
\rowcolor{lightgray}
\multirow{-5}{1.2cm}{0.54.0.1}
& $3$ & (1) & 1 & 1 & 4\\
\rowcolor{lightgray}
& $2$ & () & 1 & 1 & 9\\
& $66$ & (0\:0\:0\:0\:0\:0\:0\:0\:0\:0\:1\:1\:0\:0\:0\:0\:0\:0\:0\:0\:0\:0\:0\:0\:0\:1\:0\:0\:0\:0\:0\:0) & 0 & 0 & 1\\
& $33$ & (0\:0\:0\:0\:0\:1\:0\:0\:0\:0\:3\:1\:0\:0\:0\:0) & 0 & 0 & 1\\
& $22$ & (0\:0\:0\:1\:0\:0\:0\:0\:1\:4) & 0 & 0 & 1 \\
& $11$ & (1\:0\:0\:1\:0) & 0 & 1 & 1\\
& $6$ & (0\:12) & 1 & 0 & 1 \\
& $3$ &(0) & 1 & 1 & 5 \\
\multirow{-7}{1.2cm}{0.66.0.1}
& $2$ &() & 1 & 1 & 10\\
\hline\end{longtable}
}

\end{document}